\documentclass[11pt,twoside, eqno]{article}

\usepackage{extarrows}
\usepackage{graphics}
\usepackage{epsfig}
\usepackage{color}
\usepackage{verbatim}
\usepackage[displaymath,mathlines,pagewise]{lineno}
\usepackage{amssymb,amsmath,amsfonts,amsthm,color,mathrsfs}
\usepackage[Symbol]{upgreek}
\usepackage{txfonts}
\usepackage[nottoc,notlot,notlof]{tocbibind}
\usepackage[active]{srcltx}
\usepackage{picinpar,graphicx} % ²åÈëÍŒÆ¬
\usepackage{bbm}
\usepackage{soul}
\usepackage[colorlinks, citecolor=blue,pagebackref,hypertexnames=false]{hyperref}

\usepackage{tikz}%»­ÍŒ
\usetikzlibrary{positioning} %ÎªÁËÊµÏÖÏà¶ÔÎ»ÖÃµÄÉè¶š »­ÍŒ
\usepackage{xcolor} %ÎªÁËÊµÏÖ²»Í¬µÄÑÕÉ« »­ÍŒ

\allowdisplaybreaks

\def\dis{\displaystyle}

\def\rr{{\mathbb R}}
\def\rn{{{\rr}^n}}

\def\nn{{\mathbb N}}

\def\fz{\infty}
\def\az{\alpha}
\def\supp{{\mathop\mathrm{\,supp\,}}}

\def\lz{\lambda}
\def\dz{\delta}

\def\bz{\beta}

\def\gz{{\gamma}}
\def\oz{{\omega}}

\def\vz{\varphi}

\def\d{\mathrm{d}}

\def\BMO{{\mathrm{BMO}}}
\def\HMO{{\mathrm{HMO}}}

\def\CMO{{\mathrm{CMO}}}
\def\HCMO{{\mathrm{HCMO}}}

\def\LV{{\mathscr{L}}}

\def\d{\mathrm{d}}

\def\BMO{{\mathrm{BMO}}}
\def\HMO{{\mathrm{HMO}}}

\def\L{{\mathcal{L}}}

\def\r{\right}
\def\lf{\left}

\newtheorem{thmA}{Theorem}

\newcommand{\B}{\mathcal{B}}

\def\eqref#1{(\ref{#1})}

\newtheorem{theorem}{Theorem}[section]
\newtheorem{lemma}[theorem]{Lemma}
\newtheorem{corollary}[theorem]{Corollary}

\theoremstyle{definition}
\newtheorem{remark}[theorem]{Remark}
\newtheorem{definition}[theorem]{Definition}

\numberwithin{equation}{section}

\begin{document}
\arraycolsep=1pt

\title{\Large\bf
%Harmonic functions with Neumann boundary condition in $\mathbb{R}^n\times\mathbb{R}_+$
%On the Laplace equation with Neumann boundary condition
%On the Neumann harmonic function in $\mathbb{R}^n\times\mathbb{R}_+$
%The initial value problem for the Laplace equation with Neumann boundary condition on the upper half-space
%On the improved reverse H\"{o}lder order and applications  %to the Schr\"{o}dinger harmonic functions with the CMO traces
H\"{o}lder regularity and Liouville theorem for the Schr\"{o}dinger equation
with certain critical potentials, and applications to Dirichlet problems
\footnotetext{\hspace{-0.35cm}
{\it 2020 Mathematics Subject Classification}.
{Primary 35J10; Secondary 42B35, 43A85.}
\endgraf{\it Key words and phrases.}  Schr\"{o}dinger operator, Reverse H\"{o}lder class, BMO, Metric measure space.
%\endgraf This work is supported by NNSF of China (XXX).
\endgraf $\ast$\, Corresponding author
}}
%\author{ {AAA, BBB}\\ }
\author{Bo Li, Ji Li  and Liangchuan Wu$^*$ }
\date{}
\maketitle

\begin{center}
\begin{minipage}{15cm}\small
{
\noindent
{\bf Abstract.}
{Let $(X,d,\mu)$ be a metric measure space satisfying a doubling property with the upper/lower dimension $Q\ge n>1$, and admitting an $L^2$-Poincar\'e inequality.
In this article, we establish the H\"{o}lder continuity and a Liouville-type theorem for the (elliptic-type) Schr\"odinger equation
$$
   \mathbb L u(x,t)=-\partial^2_{t} u(x,t)+\mathcal L u(x,t)+V(x) u(x,t)=0,\quad x\in X,\, t\in \mathbb R,
$$
where $\mathcal L$ is a non-negative operator generated by a Dirichlet form on $X$, and the non-negative potential $V$  is a  Muckenhoupt weight belonging to
the  reverse H\"older class ${RH}_q(X)$ for some $q> \max\{Q/2,1\}$.
Note that $Q/2$ is critical for the regularity theory of $-\Delta+V$ on $\mathbb{R}^Q$ ($Q\ge3$) by Shen's work [\textit{Ann. Inst. Fourier (Grenoble)}, 1995],
which hints the critical index of $V$ for the regularity results above on  $X\times \mathbb R$ may be $(Q+1)/2$.
Our results show that this critical index is in fact $\max\{Q/2,1\}$.

Our approach primarily relies on the controllable growth of $V$ and the elliptic theory for the operator $\mathbb L$/$-\partial^2_{t}+\mathcal{L}$ on $X\times \mathbb R$,
rather than the analogs for $\mathcal L+V$/$\mathcal{L}$ on $X$, under the critical index setting.
On the one hand, the size condition and H\"older estimate for the Green function on $X\times \mathbb R$ are crucial
for the H\"older continuity of solutions to $ \mathbb L u=0$ in $X\times \mathbb{R}$.
On the other hand, we exploit the Liouville theorem
via establishing various mean value properties (with an exponential decay) for $\mathbb L$-harmonic functions on $X\times \mathbb{R}$.
Moreover, we give another Liouville theorem for $\mathbb L$-harmonic functions on $X\times \mathbb{R}_+$ by the aid of the Poisson kernel,
which indicates the different analysis tool on $X\times \mathbb{R}$ and $X\times \mathbb{R}_+$.

As applications, we further obtain some characterizations for solutions to the Schr\"odinger equation
 $-\partial^2_{t} u(x,t)+\mathcal L u(x,t)+V(x) u(x,t)=0$
  in $X\times \mathbb R_+$ with boundary values
   in BMO/CMO/Morrey spaces related to $V$, improving previous results to the critical index $q> \max\{Q/2,1\}$.
}
}
\end{minipage}
\end{center}

%\tableofcontents

\section{Introduction}
\hskip\parindent%
%One of the most essential and important questions in PDEs is the regularity of the solution.
%On the one hand, for instance, Hilberts's XIXth problem (1900) asked: roughly speaking, whether all solutions to uniformly elliptic variational PDEs are smooth.
%This question was answered positively by De Giorgi (1956), Nash (1957) and Moser (1961), respectively,
%and it is now one of the most famous and vital conclusions in the whole field of PDE.
%and their works have enormously influenced many areas of Mathematics linked one way or another with PDE, including:
%Harmonic Analysis, Probability Theory, Calculus of Variations, Differential Geometry,  Geometric Measure Theory,  Potential theory
%and Mathematical Physics.
%%Continuum and Fluid Mechanics, Continuum and Fluid Mechanics,
%%On the other hand, the classical Liouville theorem characterizes the constancy of a harmonic function in the whole space bounded from above or below.
%On the other hand, the classical Liouville theorem characterizes the constancy of a harmonic function bounded from above or below in the whole space.
%%(i.e., a function  $u$ satisfying the Laplace equation $\Delta u=0$ in $\mathbb R^n$)
%%that is bounded.
%This theorem has been generalized and serves as a cornerstone in the study of elliptic PDEs, particularly in various contexts involving  existence, non-existence, uniqueness and regularity properties.
%
%
One of the essential and most important questions in partial differential equations (PDEs) is the regularity of the solution.
For instance, Hilberts's XIXth problem (1900) asked: roughly speaking, whether all solutions to uniformly elliptic variational PDEs are smooth.
This question was answered positively by De Giorgi (1956), Nash (1957) and Moser (1961), respectively,
which is now one of the most famous and vital theorems in PDEs.
Subsequent works contributed by Nirenberg, Caffarelli, Krylov, Evans, Figalli, and many others,
 have enormously influenced many areas of Mathematics linked one way or another with PDE, including:
Harmonic Analysis, Probability Theory, Calculus of Variations, Differential Geometry,  Geometric Measure Theory,  Potential theory
and Mathematical Physics.
%Continuum and Fluid Mechanics, Continuum and Fluid Mechanics,
%On the other hand, the classical Liouville theorem characterizes the constancy of a harmonic function in the whole space bounded from above or below.
%On the other hand, the classical Liouville theorem characterizes the constancy of a harmonic function bounded from above or below in the whole space.
%(i.e., a function  $u$ satisfying the Laplace equation $\Delta u=0$ in $\mathbb R^n$)
%that is bounded.
Among these, the Liouville theorem and its generalized versions %have been generalized and
serve as a cornerstone in the study of elliptic PDEs, particularly in various contexts involving existence, non-existence, uniqueness and regularity properties.

When extending the Laplace operator to the Schr\"odinger operator $L=-\Delta +V$ on $\mathbb R^n$ with $n\geq 3$,
the corresponding  regularity theory becomes more complicated due to the presence of the potential $V$,
where   $V\geq 0$ belongs to the reverse H\"older class ${RH}_q(\mathbb R^n)$ for some $q> n/2$ in the sense that  $V\in L_{\text{loc}}^q(\mathbb R^n)$ and there exists a constant $C>0$ such that the reverse H\"older inequality
$$
    \left(\frac{1}{|B|}\int_B V(x)^q \d x\right)^{1/q	}\leq \frac{C}{|B|}\int_B V(x)\d x
$$
holds for all balls $B$ in $\mathbb R^n$.
In 1995, Shen \cite{Sh1995} proved that if $V\in {RH}_q(\mathbb R^n)$ for some $q> n/2$,
then the solution $u$ to the Schr\"{o}dinger equation $(-\Delta+V)u=0$ is H\"{o}lder continuous.
Moreover, by combining the upper estimate on the fundamental solution of $-\Delta +V$ with the subharmonic property of $|u|$ (see \cite{Sh1995, Sh1999}),
%(see  \cite[Theorem 2.7]{Sh1995})  and \cite[Lemma 2.9]{Sh1999} by Z. Shen, respectively),
Duong--Yan--Zhang \cite{DYZ2014} established the following Liouville theorem:
if a weak solution to the Schr\"odinger equation $Lu=(-\Delta +V)u=0$ in $\rn$ satisfies
\begin{equation}\label{eqn:Liouville-condition-DYZ}
	  \int_{\mathbb R^n} \frac{|u(x)|}{1+|x|^{n+s}} \d x<\infty
\end{equation}
for some $s>0$, then $u= 0$ in $\mathbb R^n$, where $V\in RH_q(\rn)$ for some $q> n/2$.

%
%Let $V\in RH_q(\rn)$ for some $q> n/2$. Assume $u\in W_{\text{loc}}^{1,2}(\mathbb R^n)$ is a weak solution of $Lu=0$ in $\mathbb R^n$ and assume that there exists a constant $s>0$ such that
%\begin{equation}\label{eqn:Liouville-condition-DYZ}
%	  \int_{\mathbb R^n} \frac{|u(x)|}{1+|x|^{n+s}} \d x<\infty,
%\end{equation}
%then $u= 0$ in $\mathbb R^n$.
%

These regularity results are crucial for Dirichlet problems in the upper half-space with traces in several typical function spaces.
We will illustrate this topic when the trace belongs to the endpoint space $\rm BMO$,
the classical space of bounded mean oscillation introduced by John--Nirenberg in 1961.
 %Dirichlet problems of elliptic equations are fundamental in  PDEs,
Let's begin by considering the case $V= 0$ and recalling
the Carleson measure characterization for solutions to the Laplace equation in the upper half-space with boundary values in BMO.
%the characterization for solutions of the BMO-Dirichlet problem for the Laplace equation in the upper half-space.
It's well-known from the seminal work  of Fefferman--Stein \cite{FS1972} that a function $f$ is in ${\rm BMO}(\mathbb R^n)$
if and only if  its  Poisson integral  (denoted by $u$)  satisfies the Carleson measure condition
\begin{equation}\label{eqn:Carleson}
	\sup _{B} \frac{1}{|B|} \int_0^{r_B} \int_{B}|t \nabla u(x, t)|^2 \d x \frac{\d t}{t}<\infty,
\end{equation}
where $\nabla=(\nabla_x,\partial_t)$. This is the key to proving the dual of the Hardy space $H^1(\rn)$ is ${\rm BMO}(\rn)$.
 Furthermore, let ${\rm HMO}(\mathbb R_+^{n+1})$ be the space of all harmonic functions on $\mathbb R_+^{n+1}$ satisfying \eqref{eqn:Carleson}.
The subsequent work  by Fabes--Johnson--Neri \cite{FJN1976}
showed that any $u\in {\rm HMO}(\mathbb R_+^{n+1})$ can be recovered from a trace in ${\rm BMO}(\mathbb R^n)$ via its Poisson integral.
% is representable by the Poisson integral of certain function in  ${\rm BMO}(\mathbb R^n)$.
 Therefore, their result can be regarded as
$$
     {\rm HMO}(\mathbb R_+^{n+1})=e^{-t\sqrt{-\Delta}}{\rm BMO}(\mathbb R^n).
$$

For general $V\in RH_q(\rn)$ with $q>n/2$, we say a locally integrable function $f$   belongs to  ${\rm BMO}_{L}(\mathbb R^n)$,
the space of BMO associated to $L=-\Delta+V$ (see \cite{DGMTZ2005}),
if  $f$ satisfies
$$     \|f\|_{{\rm BMO}_{L}(\mathbb{R}^n)}=\sup_{B: r_B<\rho(x_B)}\frac{1}{|B|} \int_{B}|f(y)-f_{B}|\d y
      +\sup_{B: r_B\geq \rho(x_B)} \frac{1}{|B|} \int_{B}|f(y)| \d y <\infty,$$
where %the  auxiliary  function $\rho(x)$ is defined by
$$
 \rho(x)=\sup \left\{r>0: \frac{1}{r^{n-2}} \int_{B(x, r)} V(y) \d y \leq 1\right\}.
$$
This ${\rm BMO}_{L}(\rn)$ is a proper subspace of ${\rm BMO}(\mathbb R^n)$. When $V= 1$, it is the ${\rm bmo}$ space  introduced by Goldberg \cite{Go1}.  Similarly, let ${\rm HMO}_L(\mathbb R_+^{n+1})$ be the space of all solutions to $-\partial_t^2 u(x,t)+ Lu(x,t)=0$ on $\mathbb R_{+}^{n+1}$ satisfying \eqref{eqn:Carleson}.

As a straightforward consequence of the above Liouville theorem associated to  $L$, it was  proven in \cite{DYZ2014} that when $V\in  RH_q(\rn)$ for some $q>(n+1)/2$,
%there exists a function $f$ belongs to ${\rm BMO}_L(\mathbb R^n)$,  which is the trace of the solution of $-u_{tt}(x,t)+ Lu(x,t)=0$ in the upper half-space $\mathbb R_{+}^{n+1}$ with $u(x,0)=f(x)$, where $u$ is assumed to satisfy the Carleson condition
%
%here $\nabla=(\nabla_x,\partial_t)$. That is,
 for any  $u\in {\rm HMO}_L(\mathbb R_+^{n+1})$ with $C^1$-regularity,
 there exists some $f\in {\rm BMO}_L(\mathbb R^n)$ such that $u(x,t)=e^{-t\sqrt{L}}f(x)$, which
 extends the result of Fabes--Johnson--Neri \cite{DYZ2014}.
 Recently, the above Liouville theorem and its application to the Poisson integral representation were extended
 to manifolds as well as general metric spaces by Jiang--Li in \cite{JL2022}.
Conversely, for any $f\in {\rm BMO}_{L}(\mathbb R^n)$,
the Poisson integral $u=e^{-t\sqrt{L}}f$  belongs to $ {\rm HMO}_{L}(\mathbb R_+^{n+1})$ whenever $V\in  RH_q(\rn)$ with $q>n$ (see \cite{DYZ2014}), which was later  improved into  $q> (n+1)/2$ (see \cite{JL2022}). Eventually, it was shown that, if $V\in  RH_q(\rn)$ for some $q>(n+1)/2$, then
\begin{equation}\label{eqn:BMOL-Dirichlet-22}
       {\rm HMO}_L(\mathbb R_+^{n+1})=e^{-t\sqrt{L}}{\rm BMO}_L(\mathbb R^n).
\end{equation}

We clarify the occurrence of $q> (n+1)/2 $ in the  above result.
To recover $u\in {\rm HMO}_L(\mathbb R_+^{n+1})$ by a trace in ${\rm BMO}_{L}(\rn)$,  one key intermediate step is to demonstrate that  for every $k\in\mathbb N$, $u(x,t+1/k)$ can be represented by $e^{-t\sqrt{L}}(u(\cdot, 1/k))(x)$ for $(x,t)\in \mathbb R_+^{n+1}$ (note that when $V= 0$, this fact was obtained in \cite{FJN1976} by the maximal principle for harmonic functions, bypassing the Liouville theorem),  where  the index $q>(n+1)/2$ %and the applicability of the Liouville theorem above
 arises   from the following observations:
\begin{itemize}
	\item [(i)]
  The function	$w(x,t)=u(x,t+1/k)-e^{-t\sqrt{L}}(u(\cdot, 1/k))(x)$ satisfies $-\partial_t^2w(x,t)+ Lw(x,t)=0$
  in the upper half-space $\mathbb R_{+}^{n+1}$ with zero trace,
  which can be extended canonically  to $\overline w$
  such that $ (-\partial_{t}^2-\Delta) \overline w(x,t)  +V(x)\overline w(x,t)=(-\Delta_{\mathbb R^{n+1}}+ V(x)) \overline w(x,t)=0$
  in $\mathbb R^{n+1}$, and one can verify that
	$\overline w\in L^1   (\mathbb R^{n+1}, (1+|x|)^{4(n+1)}\d x)$.
	
	%\item[(ii)] The condition $V\in RH_ q(\mathbb R^n)$ can be also regarded as a potential $V(x,t)=V(x)$ for each $t\in \mathbb R$, which belongs to ${RH}_q(\mathbb R^{n+1})$.

\item[(ii)] The natural extension $V(x,t)=V(x)$ guarantees that it also belongs to  $RH_ q(\mathbb R^{n+1})$ for some $q>(n+1)/2$.

\end{itemize}
%Therefore, the Liouville theorem is the key ingredient for the uniqueness of the solution satisfying \eqref{eqn:Carleson},
Hence,  applying  the Liouville theorem above (adapted for $\mathbb R^{n+1}$) to the  Schr\"odinger operator $-\partial_{t}^2+L$  on $\mathbb R^{n+1}$, it follows from (i) that  $\overline w= 0$ whenever  $q> (n+1)/2$, as desired.
%It should be emphasized that Duong--Yan--Zhang (\cite{DYZ2014}) assumed the solution to the Schr\"odinger equation admits the $C^1$-regularity.
% For more details, one may refer to the proof of part (1) of Theorem 1 in \cite{DYZ2014}.

Remarkably,  as shown in Shen's work \cite{Sh1995},
the reverse H\"older index $n/2$ of $V$ is critical for the regularity theory of the operator $L=-\Delta +V$ on $\mathbb R^n$.
From this perspective, the index $q> (n+1)/2$ seems to be optimal for ensuring the H\"{o}lder continuity
and the Liouville property of solutions to $-\partial_t^2u+L u=0$ in $\mathbb R_+^{n+1}$,
as these are the only aspects that are sensitive to the reverse H\"older index $q$.
 However, the strategies in \cite{DYZ2014, JL2022} primarily involve transferring regularity results for $L=-\Delta+V$ in $\mathbb R^n$
 to the Schr\"odinger equation $-\partial_t^2u+Lu=0$ in the higher dimensional space $\mathbb R_+^{n+1}$.
 Consequently, the condition $q>(n+1)/2$ exists only for technical reasons, and hence may not be sharp.

A natural question, inspired by the Dirichlet problem in the upper half-space, is whether
one can establish a suitable H\"{o}lder estimate and an improved  Liouville theorem for the Schr\"odinger equation
when the potential $V$ satisfying the critical reverse H\"older index $q> n/2$.

Obviously, approaches in  \cite{DYZ2014, JL2022, Sh1999}
%that rely on the regular estimates of the Poisson equation and the fundamental solution of $L$ on $\mathbb R^n$
are not applicable for such regularity results.
More precisely, Duong--Yan--Zhang, Jiang--Li, and Shen first established
some regularity results for the Schr\"{o}dinger equation on $X$/$\rn$, such as the H\"{o}lder continuity and the Liouville theorem.
When considering the Dirichlet problem on $\rr^{n+1}_+$, the natural extension $V(x,t)=V(x)$ for $(x,t)\in \rr^{n+1}_+$
 falls into $RH_q(\rr^{n+1})$ for some $q>(n+1)/2$ immediately
to ensure that these regularity results mentioned above are still valid.
However, this $(n+1)/2$ does not seem to be sharp, since the potential $V$ is in fact time independent.

%In this paper % We are determined not to invoke the extension technology because the dimension $n+1$ is inescapable in this process.
%To conquer this difficulty, we need to
%we
%establish a revised version of the H\"{o}lder continuity and the Liouville theorem on $\rr^{n+1}$ directly.
%Our approach primarily relies on the controllable growth of $V$ and the theory of the elliptic operator on $\mathbb{R}^{n+1}$,
%rather than the analogs on $\rn$, under the critical index setting; see Sections \ref{ho} and \ref{improved} for more details.

In this article, we address this problem in the context of a metric space $X$ as in \cite{JL2022}
to enlarge the applicable fields; see Section \ref{pre} for precise  definitions. We establish a revised version of the H\"{o}lder continuity and the Liouville theorem on $\rr^{n+1}$ directly.
Our approach primarily relies on the controllable growth of $V$ and the theory of the elliptic operator on $\mathbb{R}^{n+1}$,
rather than the analogs on $\rn$, under the critical index setting; see Sections \ref{ho} and \ref{improved} for more details.

Our first main result is the following H\"{o}lder regularity and Liouville theorem for the Schr\"odinger equation in the product space $X\times \mathbb R$,
where the potential satisfies the critical reverse H\"older index.

 \begin{thmA}  \label{thm:Liouville}
Let $(X,d,\mu,\mathscr{E})$ be a complete Dirichlet metric measure space
satisfying a doubling property with the upper/lower dimension $Q\ge n>1$, and admitting an $L^2$-Poincar\'{e} inequality.
%and $\LV=\mathcal L+V$ be a Schr\"odinger operator on $X$,
Assume that $u$ is a weak solution to the elliptic equation
$$
\mathbb{L}u=-\partial_{t}^2u+\LV u=-\partial_{t}^2u+\L u+V u=0
$$
in $X \times \rr$,
where $\mathcal L$ is a non-negative operator generalized by a Dirichlet form on $X$,
and $0\le V\in A_\infty(X)\cap RH_q(X)$ for some $q> \max\{Q/2,1\}$. Then we obtain the following.
%If there exist some $(x_0,t_0)\in X\times \rr$ and $\bz>0$ such that
%\begin{align}\label{global-control-new}
%\int_{X\times\rr} \frac{|u(x,t)|}{ \big  (1+d((x,t),(x_0,t_0))\big )^{\bz}\mu\big  (B((x_0,t_0),1+d((x,t),(x_0,t_0)))\big )}\d\mu(x)\d t\le C{(x_0,t_0,\bz)}<\infty,
%\end{align}
%then $u=0$ in $X\times \rr$.	
\begin{enumerate}
  \item[{(i)}] (H\"{o}lder regularity\footnote{To show the H\"{o}lder regularity of a weak solution to the elliptic equation $\mathbb{L}u=0$ in $X\times \rr$,
  we need to assume that the lower dimension $n\ge2$ and hence $\max\{Q/2,1\}=Q/2$ by $Q\ge n$. However, under the assumption $n>1$, another H\"{o}lder regularity is provided in Section \ref{sec4.3}.})
There exists a constant $0<\az<1$ such that for any $(x,t),(y,s)\in B_0=B((x_0,t_0),R)\subset X\times\rr$,
$$|u(x,t)-u(y,s)|\le C{\left(\frac{d((x,t),(y,s))}{R}\right)^\az}\left(1+R^2\fint_{B(x_0,R)} V\d\mu\right) \sup_{2B_0}|u|.$$
 \item[{(ii)}] (Liouville theorem) If there exist some $(x_0,t_0)\in X\times \rr$ and $\bz>0$ such that
\begin{align}\label{global-control-new}
\int_{X\times\rr} \frac{|u(x,t)|^2}{ \big  (1+d((x,t),(x_0,t_0))\big )^{\bz}\mu\big  (B((x_0,t_0),1+d((x,t),(x_0,t_0)))\big )}\d\mu(x)\d t\le C{(x_0,t_0,\bz)}<\infty,
\end{align}
then $u\equiv0$ in $X\times \rr$.
\end{enumerate}
 \end{thmA}

%
%\begin{theorem}
%Let $(X,d,\mu,\mathscr{E})$ be a complete Dirichlet metric measure space satisfying a doubling property $(D)$\footnote{For some technical reasons, the lower dimension $n$ is always assumed to be equal or greater than 2 in this result.}
%and admitting an $L^2$-Poincar\'{e} inequality $(P_2)$.
%Assume that $u$ is a weak solution to the elliptic equation
%$$-\partial^2_tu+{\mathscr{L}}u=-\partial^2_tu+\mathcal{L}u+Vu=0$$
%in  $2B_0\subset X\times \rr$ with $B_0=B((x_0,t_0),R)$, where  $0\le V\in A_\infty(X)\cap RH_q(X)$ for some $q>Q/2$.
%%Then there exists $C>0$ such that for any ball $B$ with $2B\subset \Omega$,
%%$$\|u\|_{L^\infty(B)}\le C\fint_{2B}|u|\d\mu.$$
%%Then $u$ is locally H\"older continuous in $2B_0$, and
%Then there exists a constant $0<\az<1$ such that for any $(x,t),(y,s)\in B_0$,
%$$|u(x,t)-u(y,s)|\le C{\left(\frac{d((x,t),(y,s))}{R}\right)^\az}\left(1+R^2\fint_{B(x_0,R)} V\d\mu\right) \sup_{2B_0}|u|.$$
%\end{theorem}

Note that when $X=\mathbb R^n$, $V\in {RH}_q(\mathbb R^n)$ implies that  $V\in A_\infty(\mathbb R^n)$ (see \cite{ST1989} for example).
For the classical Schr\"odinger operator $-\Delta +V$ on $\mathbb R^n$ equipped with the Lebegue measure, then $Q=n$ and the condition \eqref{global-control-new} is indeed the $(n+1)$-dimensional version of \eqref{eqn:Liouville-condition-DYZ},
hence our Theorem \ref{thm:Liouville} extends and improves the H\"{o}lder continuity and the Liouville theorem for the Schr\"{o}dinger operator $-\Delta_{\mathbb R^{n+1}}+V$,
weakening the reverse H\"older index requirement for the potential $V$  from   $q> (n+1)/2$ in \cite{DYZ2014} to the critical case $q> n/2$,
and relaxing the dimensional constraint from $n\geq 3$ to $n\geq 2$ (see Remarks \ref{rem:dim-relax1}
and \ref{rem:dim-relax2} for the discussion on the dimension).
%It is remarkable that Auscher--Ben Ali \cite{AB2007} achieve a surprising result:
%for $n\ge1$, if $0\le V\in RH_q(\rn)$ for some $q>1$, then
%$$\|\Delta f\|_{L^q(\rn)}+\|V f\|_{L^q(\rn)}\le C\|(-\Delta +V)f\|_{L^q(\rn)}$$
%holds for all smooth functions $f$ with compact support.
%This result extends the one of Shen obtained in \cite{Sh1995} under the restriction $q>n/2$.
%It is somewhat surprising that they can go below the order $n/2$
%which is critical for the regularity theory of the elliptic operator $-\Delta+V$.
%One might wonder if there is any possibility of breaking the requirement $q> \max\{Q/2,1\}$ in Theorem \ref{thm:Liouville} to $q>1$.
%This is an interesting problem to be solved.
Additionally, Theorem \ref{thm:Liouville} is also applicable when  $\mathcal L$ is a uniformly elliptic operator $-{\rm div}A\nabla$, where $A=A(x)$ is an $n\times n$ matrix of real symmetric, bounded measurable coefficients and satisfies the uniformly elliptic condition.

Utilizing some basic  properties of the product space $X\times\rr$,
our argument for Theorem \ref{thm:Liouville} fully exploits
the H\"older continuity of the Green function on $X\times \rr$ (not $X$) and the reverse H\"{o}lder property of $V$.
This allows us to derive important ingredients that any $\mathbb{L}$-harmonic function is H\"older continuous (see Theorem \ref{holder} and its remarks)
and satisfies an appropriate upper bound with exponential decay (see Lemma \ref{mean}).
Unlike previous methods that rely on the regular estimates of the Poisson equation and the fundamental solution of $\LV$ on $X$,
here our strategy addresses these issues directly on $X\times \mathbb R$, avoiding the constrained increase of the
 reverse H\"older index by $1/2$.
%to derive the Liouville property of the $\mathbb{L}$-harmonic function (see Proposition \ref{Liouville} and its remark below).

With the help the regularity results in the critical sense,
the second part of this article is dedicated to applications in Dirichlet problem
for the Schr\"odinger equation in the upper half-space $X\times \mathbb R_+$,
where the potential  $V$ is also critical that belongs to $A_\infty(X)\cap RH_{q}(X)$ for some $q>\max\{ Q/2,1\}$.

\smallskip

We start by considering the ${\rm BMO}_{\LV}$-Dirichlet problem.
Concretely, %let $\LV=\L+V$ and
define the harmonic mean oscillation space ${\rm HMO}_{\LV}(X\times \mathbb R_+)$ as the class of all the $\mathbb{L}$-harmonic functions  on $X\times \mathbb R_+$
(i.e., weak solutions to $\mathbb{L}u=-\partial_{t}^2 u+\L u+Vu=0$ in $X\times \mathbb R_+$) satisfying
\begin{equation}\label{eqn:HMO-LV}
    \|u\|_{{\rm HMO}_{\LV}}=\sup_{B}\mathcal C_{u,B}
    = \sup_{B}\left(\frac{1}{\mu(B)}   \int_0^{r_B} \int_{B} |t\nabla u|^2 \d \mu \frac{\d t}{t}\right)^{1/2}<\infty.
\end{equation}
%Similarly, we  define $\rho(x):=\sup\left\{r>0:\, \frac{r^2}{\mu(B(x,r))} \int_{B(x,r)}  V(y)d\mu(y)\leq 1\right\}$ for $x\in X$.

 \begin{thmA}  \label{thm:BMO-Dirichlet}
 Let all assumptions and notation be the same as in Theorem \ref{thm:Liouville}. % and let $\LV=\L+V$.
Then $u\in \mathrm{HMO}_{\LV}(X\times \rr_+)$ if and only if $u=e^{-t\sqrt{\LV}}f$ for some $f\in \BMO_\LV(X)$.
Moreover, it holds that % there exists a constant $C>0$ such that
$$ \|u\|_{\HMO_{\LV}(X\times \rr_+)} \approx\|f\|_{\BMO_\LV(X)}.$$
%\begin{itemize}
%	\item [(i)] If $u\in {\rm HMO}_{\LV}(X\times \mathbb R_+)$, then there exists a function $f\in {\rm BMO}_{\LV}(X)$ such that $u(x,t)=e^{-t\sqrt{\LV}} f(x)$ with
%	$$
%	    \|f\|_{{\rm BMO}_{\LV}}\leq C\|u\|_{{\rm HMO}_{\LV}},
%	$$
%	where $C>0$ is independent of $u$.
%	
%	\item[(ii)] Conversely, if $f\in {\rm BMO}_{\LV}(X)$, then $u(x,t):=e^{-t\sqrt{\LV}}f(x)\in {\rm HMO}_{\LV}(X\times \mathbb R+)$, and
%	$$
%	    \|u\|_{{\rm HMO}_{\LV}}\leq C  \|f\|_{{\rm BMO}_{\LV}},
%	$$
%	where $C>0$ is independent of $f$.
%\end{itemize}
 \end{thmA}

This result improves the characterization \eqref{eqn:BMOL-Dirichlet-22} on the ${\rm BMO}_{\LV}$-Dirichlet problem to the critical index $q> \max\{ Q/2,1\}$.
Theorem \ref{thm:BMO-Dirichlet}  is the ultimate analogue of the classical result ${\rm HMO}(\mathbb R_+^{n+1})=e^{-t\sqrt{-\Delta}}{\rm BMO}(\mathbb R^n)$
for the Laplacian by Fabes-Johnson-Neri, after previous works  \cite{DYZ2014, JL2022} in this topic.

Furthermore, when $X=\mathbb R^n$,
Song--Wu \cite{SW2022} studied the CMO-Dirichlet for the Laplace equation with a potential $V$ from ${RH}_q(\rn)$ for some $q> (n+1)/2$.
Here the trace space ${\rm CMO}_L(\rn)$, the space of vanishing mean oscillation associated to $L=-\Delta +V$,
is a proper subspace of ${\rm BMO}_L(\rn)$ and  is the pre-dual space of the Hardy space $H_L^1(\mathbb R^n)$.
%Song and the third author \cite{SW2022} studied the ${\rm CMO}_L$-Dirichlet problem with $V\in {RH}_q(\rn)$ for $q> (n+1)/2$,
%where  ${\rm CMO}_L(\rn)$ is the space of vanishing mean oscillation associated to $-\Delta +V$,
%a proper subspace of ${\rm BMO}_{-\Delta+V}$ and  is the pre-dual space of the Hardy space $H_{-\Delta +V}^1(\mathbb R^n)$.
Based on Theorem \ref{thm:BMO-Dirichlet} with   $X=\mathbb R^n$, we can extend with the same method in \cite{SW2022}
(relying heavily on two new characterizations of ${\rm CMO}_{L}(\mathbb R^n)$) to characterize the ${\rm CMO}_L$-Dirichlet problem in $\mathbb R_+^{n+1}$
under the critical index $q> n/2$.
For general $X$ and the Schr\"odinger operator $\LV=\L+V$, we investigate the ${\rm CMO}_{\LV}$-Dirichlet problem
for the Schr\"odinger equation $-\partial_{t}^2u+\LV u=0$ in $X\times \mathbb R_+$, which is more complicated due to  the underlying geometric and metric structures. Based on some new observations, we conquer this difficulty by using the Poincar\'{e} inequality and variants of the Calder\'{o}n reproducing formula.
This approach allows us to move beyond the common function space theory associated to operators, making it suitable for the setting of  metric measure spaces.
Specifically,
 denote
$$
    {\rm HCMO}_{\LV}(X\times \mathbb R_+)=\left\{u\in {\rm HMO}_{\LV}(X\times \mathbb R_+):\       \lim _{a \rightarrow 0}\sup _{B: r_{B} \leq a}  \mathcal{C}_{u,B}
     =  \lim _{a \rightarrow \infty}\sup _{B: r_{B} \geq a}  \mathcal{C}_{u,B}
     =   \lim _{a \rightarrow \infty}\sup _{B: B \subseteq \left(B(x_0, a)\right)^c}  \mathcal{C}_{u,B}=0\right\},
$$
where $\mathcal{C}_{u,B}$ is  defined in \eqref{eqn:HMO-LV}. We then establish the following result.

\begin{thmA}  \label{thm:CMO-Dirichlet}
Let all assumptions and notation be the same as in Theorem \ref{thm:Liouville}. %and let $\LV=\L+V$.
Then
Then $u\in \mathrm{HCMO}_{\LV}(X\times \rr_+)$ if and only if $u=e^{-t\sqrt{\LV}}f$ for some $f\in \mathrm{CMO}_\LV(X)$.
Moreover, it holds that % there exists a constant $C>0$ such that
$$ \|u\|_{\HMO_{\LV}(X\times \rr_+)} \approx\|f\|_{\BMO_\LV(X)}.$$
\end{thmA}

The $\rm BMO$ space is a natural substitution of $ L^\infty$ from the perspective of singular integrals,
thus the above characterizations on  Dirichlet problems with traces $f$ in ${\rm BMO}_{\LV}(X)$ and its subspace ${\rm CMO}_{\LV}(X)$
are both endpoint cases in this context.  For non-endpoint situations and as an application of Theorem \ref{thm:Liouville},
we turn to explore the real-variable characterization of the Dirichlet problem for Schr\"odinger equation in $X\times \mathbb R_+$ with weighted Morrey trace.
The Morrey space, as generalized Lebesgue spaces, introduced by Morrey in 1938 to study the regularity properties of solutions to PDEs, plays an important role at the interface between harmonic analysis, PDEs and potential theory;  see \cite{LLSY2024,SL2022,STY2018,YYZ2010-1} for example.
Here we introduce a variety of weighted harmonic Morrey functions (see Definition \ref{def2}),
which provide a more general solution space for the Schr\"odinger equation.
This work is inspired by the classical result  in Lebesgue traces by Stein and Weiss \cite{SW1971},
and contains the characterization  \cite{STY2018}  for the classical Morrey-Dirichlet problem under the Euclidean setting as a special case.
%{\color{red}(RMK. To many notation to describe the Morrey result.)}

The article is organized as follows.
In Section \ref{pre}, we recall some relevant definitions, including the doubling property, the Poincar\'{e} inequality, the reverse H\"{o}lder class.
Some estimates for the heat/Poisson/Green kernel are given in Section \ref{estimate-kernel},
and Sections  \ref{ho} and \ref{improved} are devoted to proving Theorem \ref{thm:Liouville} with the critical reverse H\"{o}lder index  $\max\{Q/2,1\}$.
 As applications, we characterize all $\mathbb{L}$-harmonic functions with BMO/CMO/Morrey traces in Sections \ref{s4}--\ref{s6}, respectively.

Throughout this article, we  denote by $C$ a positive constant which is independent of the main parameters, but it may vary from line to line.

\section{Preliminaries}\label{pre}

\subsection{Underlying space}
\hskip\parindent
In this subsection, we first briefly describe our Dirichlet metric measure space settings;
see \cite{BD1959,fukushima1994dirichlet,St1994,St1996} for more details.
Suppose that $X$ is a separable, connected, locally compact and metrisable space.
Let $\mu$ be a Borel measure that is strictly positive on non-empty open sets and finite on compact sets.
We consider a strongly local, closed, and regular Dirichlet form $\mathscr{E}$ on $L^{2}(X, \mu)$
with dense domain $\mathscr{D} \subset L^{2}(X, \mu)$ (see \cite{fukushima1994dirichlet} for an accurate definition).
Suppose that $\mathscr{E}$ admits a ``\textsl{carr\'{e} du champ}'',
which means that for all $f, g \in \mathscr{D}$, $\Gamma(f, g)$ is absolutely continuous with respect to the measure $\mu$.
In what follows, for simplicity of notation,
let $\langle\nabla_{x} f, \nabla_{x} g\rangle$ denote the energy density ${\d \Gamma(f, g)}/{\d \mu}$
and $|\nabla_{x} f|$ denote the square root of ${\d \Gamma(f, f)}/{\d \mu}$.
Assume the space $(X, \mu, \mathscr{E})$ is endowed with the intrinsic (pseudo-)distance on $X$ related to $\mathscr{E}$, which is defined by setting
$$
d(x, y)=\sup \left\{f(x)-f(y): f \in \mathscr{D}_{\text {loc}} \cap C(X),\ \left|\nabla_{x} f\right| \leq 1 \text { a.e.}\right\},
$$
where $C(X)$ is the space of continuous functions on $X$.
Suppose $d$ is indeed a distance and induces a topology equivalent to the original topology on $X$.
As a summary of the above situation, we will say that $(X, d, \mu, \mathscr{E})$
is a complete Dirichlet metric measure space.

Let the domain $\mathscr{D}$ be equipped with the norm $[\|f\|_{2}^{2}+\mathscr{E}(f, f)]^{1/2}$.
We can easily see that it is a Hilbert space and denote it by $W^{1,2}(X)$.
Given an open set $U \subset X$, we define the Sobolev spaces $W^{1, p}(U)$ and $W_{0}^{1, p}(U)$ in the usual sense
(see~\cite{St1996}).
With respect to the Dirichlet form, there exists an operator $\mathcal{L}$ with
dense domain $\mathscr{D}(\mathcal{L})$ in $L^{2}(X, \mu), \mathscr{D}(\mathcal{L}) \subset W^{1,2}(X)$,
such that
$$
\int_{X} \mathcal{L} f(x)  g(x) \d \mu(x)=\mathscr{E}(f, g),
$$
for all $f \in \mathscr{D}(\mathcal{L})$ and each $g \in W^{1,2}(X)$.

\subsection{Doubling property and Poincar\'{e}'s inequality}
\hskip\parindent
We denote  by $B(x_B,r_B)$ the open ball centered at $x_B$ of radius $r_B$
and set $\lambda B(x_B,r_B)=B(x_B,\lz r_B)$ for each $\lz>0$.
We suppose that $\mu$ is doubling, i.e., there exists a constant $C_D >0$ such that for every ball $B\subset X$,
$$\mu(2B) \leq C_D \mu(B) < \infty.\eqno(D)$$
Note that the doubling property of $\mu$ implies
there exists a constant $Q>1$ such that for any $x \in X$ and $0<r<R<\infty$,
$$
\mu(B(x,R)) \leq C_D \left(\frac{R}{r}\right)^Q \mu(B(x,r)),
$$
and the reverse doubling property holds on a connected space (cf. \cite[Remark 8.1.15]{HKNT15}),
i.e., there exists a constant $0<n \leq Q$ such that for any $x \in X$ and $0<r<R<\infty$,
$$
\mu(B(x,r)) \leq C \left(\frac rR\right)^n \mu(B(x,R)).\eqno(RD)
$$
In general, we refer to $Q$ and $n$ as the upper dimension and lower dimension of the underlying space $X$, respectively.
For some technical reasons, the lower dimension $n$ is always assumed to be %equal or
greater than 1 in the whole article.

Suppose that $(X, d, \mu, \mathscr{E})$
admits an $L^{2}$-Poincar\'{e}
inequality, namely, there exists a constant $C_{P}>0$ such that for each $f\in W^{1,2}(B)$ with $B\subset X$,
$$
\left(\fint_{B}\left|f-f_{B}\right|^{2} \d \mu\right)^{1 / 2}
	\leq C_{P} r_B\left(\fint_{B}\left|\nabla_{x} f\right|^{2} \d \mu\right)^{1 / 2}, \eqno(P_2)
$$
where $f_B$ denotes the mean (or average) of $f$ over $B$, i.e.,
$$
f_B=\fint_Bf\d\mu=\frac{1}{\mu(B)}\int_Bf\d\mu.
$$

Under the validity of a doubling property and an $L^2$-Poincar\'{e} inequality on $(X,d,\mu)$,
Sturm \cite{St1996} established two sides Gaussian bounds for the heat kernel associated to $\L$ (or say Li-Yau's estimates; see \cite{LY1986}),
i.e., there
exist $C_1,c_2,c_3>0$ such that for all $x,y\in X$ and $t>0$, it holds
$$
 \frac{1}{C_1\mu(B(x, \sqrt{t}))} \exp \left(-\frac{d(x, y)^{2}}{c_3 t}\right)\le h^{\mathcal{L}}_t(x, y)
  \le \frac{C_1}{\mu(B(x, \sqrt{t}))} \exp \left(-\frac{d(x, y)^{2}}{c_2 t}\right).\eqno(LY)
$$

\subsection{Reverse H\"{o}lder class}
\hskip\parindent
In the whole article,  we shall consider the Schr\"{o}dinger operator
$$
	\mathscr{L}=\mathcal{L}+V,
$$
%where $0\le V \in A_{\infty}(X) \cap RH_{q}(X)$ for some $q> Q/2$.
where the non-trivial potential $V$ is a Muckenhoupt weight and admits certain reverse H\"{o}lder inequality.
Here the Muckenhoupt class  and the reverse H\"{o}lder class are defined as follows (see for example, \cite{Mu1972,ST1989}).
\begin{definition}\label{RH}
\begin{enumerate}
 \item[\rm{(i)}]A non-negative function $V$ on $X$ is said to be in the {{Muckenhoupt class}} $A_\infty(X)$,
if there exists a constant $C>0$ such that for any ball $B\subset X$,
$$\lf(\fint_{B}V\d \mu\r)\exp\lf(\fint_{B}\log V^{-1}\d \mu\r)\le C.$$
 \item[\rm{(ii)}]A non-negative function $V$ on $X$  is said to be in the {reverse H\"{o}lder class $RH_q(X)$}
with $1< q\le\fz$,
if there exists a constant $C>0$ such that for any ball $B\subset X$,
$$
\lf(\fint_{B}V^q\d \mu\r)^{1/q}\le C\fint_{B}V\d \mu,
$$
with the usual modification when $q=\fz$.

%if there exists some $1\le p<\fz$
%such that
%$$\sup_{B}\fint_{B}V\d\mu\left(\fint_B V^{\frac{1}{1-p}}\d\mu\right)^{p-1}\le C$$
%when $p>1$, and when $p=1$
%$$\sup_{B}\fint_{B}V\d\mu \left(\inf_{x\in B}V \right)^{-1}\le C.$$
%Above the infimum is understood as the essential infimum.
\end{enumerate}
\end{definition}

In the Euclidean setting $X=\rr^n$, it is well known that $RH_q(\rn)\subset A_\infty(\rn)$,
and every $A_\infty$-weight is in the reverse H\"older class $RH_q(\rn)$ for some $q>1$.
This is however not known in general metric measure space. %; see \cite[Chapter 1]{ST1989}.
Nevertheless, for any $V\in A_\infty(X)\cap RH_q(X)$, the induced measure $V\d\mu$ is also doubling (see \cite[Chapter 1]{ST1989}).
One remarkable feature about the reverse H\"older class is the self-improvement property, i.e.,
$V\in RH_q(X)$ implies $V\in RH_{q+\varepsilon}(X)$ for some small $\varepsilon>0$; see \cite{ST1989} for example.
%see \cite{ge73} for the Euclidean case and \cite[Chapter 3]{bb10} (see also \cite{ST1989}) for general cases.
This in particular implies $V\in L^{q+\varepsilon}_{\mathrm{loc}}(X)$. % for some $q'$ strictly greater than $q$.
However, in general, the potential $V$ can be unbounded and does not belong to $L^p(X)$ for any $1\le p\le\fz$. As a model example, we could take $V(x)=|x|^2$ on $\rn$ with $n\ge3$.

Let us recall {the definition of} the critical function $\rho$ associated with the potential $V$ (see \cite[Definition 1.3]{Sh1995}).
For all $x \in X$, let
$$\rho(x)=\sup \left\{r>0: \frac{r^2}{\mu(B(x, r))}\int_{B(x, r)} V \d \mu \leq 1\right\}.$$
If $0\le V \in  A_\infty(X)\cap RH_{q}(X) $ for some $q> \max\{ Q/2,1\}$, then
$0 < \rho(x) < \infty$ for every $x \in X$, and
$$\frac{\rho(x)^2}{\mu(B(x,\rho(x)))}\int_{B(x,\rho(x))}V\d\mu\approx1.$$
%It follows from $V$ is a non-trivial potential that $0 < \rho(x) < \infty$ for any $x \in X$.
Additionally, by \cite[Lemma 2.1 \& Proposition 2.1]{YZ2011},
the critical function $\rho$ satisfies the following property:
%if 	$0\le V \in A_{\infty}(X) \cap RH_{q}(X)$ for some $q> \max\{ Q/2,1\}$, then
there exists a  constant $k_{0}>0$ such that for all $x,y\in X$,
\begin{align} \label{critical}
	C^{-1} \rho(x)\left(1+\frac{d(x, y)}{\rho(x)}\right)^{-k_{0}} \leq \rho(y) \leq C \rho(x)\left(1+\frac{d(x, y)}{\rho(x)}\right)^{\frac{k_0}{k_0+1}}.
\end{align}

\section{Estimates for kernels}\label{estimate-kernel}
\hskip\parindent
In this section, we provide some estimates for the heat/Poisson kernel on $X$,
and the heat/Green kernel on $X\times\rr$.
For a non-negative self-adjoint operator $T$, its heat/Green kernel is denoted by $h_t^T(\cdot,\cdot)$/$G_T(\cdot,\cdot)$.
%Throughout this article, we denote by $e^{-t\sqrt{\L}}$, $e^{-t\sqrt{\LV}}$ the Poisson semigroups,
%and $p^\L_t(x,y)$, $p^\LV_t(x,y)$ their integral kernels associated with $\L$ and $\LV$, respectively.
\subsection{Estimates for the heat/Poisson kernel on $X$}
\begin{lemma}\label{lem-1}
Let $(X,d,\mu,\mathscr{E})$ be a complete Dirichlet metric measure space satisfying a doubling property $(D)$ and admitting an $L^2$-Poincar\'{e} inequality $(P_2)$.
Suppose that $0\le V\in A_\infty(X)\cap RH_q(X)$ for some $q>\max\{ Q/2,1\}$.
The following statements are valid.
\begin{itemize}
 \item[{(i)}]
 For any $k\in\{0\}\cup\nn$ and all $N>0$, the heat kernel $h^\LV_t(x,y)$ admits the Gaussian upper bound
 $$
 \left|t^k\partial_t^kh_{t}^\LV(x, y)\right|\le  \frac{C}{\mu(B(x,\sqrt{t}))} \exp \left(-\frac{d(x, y)^{2}}{c t}\right)\left(1+\frac{\sqrt{t}}{\rho(x)}\right)^{-N},\eqno(GUB)
$$
and hence the Poisson kernel $p^\LV_t(x,y)$ admits the Poisson upper bound
$$
 \left|t^k\partial_t^kp_{t}^\LV(x, y)\right|\le C \frac{t}{t+d(x, y)} \frac{1}{\mu(B(x,t+d(x, y)))}
\left(1+\frac{t+d(x,y)}{\rho(x)}\right)^{-N}.\eqno(PUB)
$$

 \item[{(ii)}]
 There exists a constant $0<\az<1$ such that, if $d(x,y)<\sqrt{t}$, then
 $$
 \left|h^\LV_t(x,z)-h^\LV_t(y,z)\right|\le C\lf(\frac{d(x,y)}{\sqrt{t}}\r)^\az  \frac{1}{\mu(B(x,\sqrt{t}))}\exp \left(-\frac{d(x,z)^{2}}{ct}\right).
 $$
% whenever .

 \item[{(iii)}]  The perturbation kernel $q_t(x,y)=p^\L_t(x,y)-p^\LV_t(x,y)$ satisfies
 $$0\le q_t(x,y)=p^\L_t(x,y)-p^\LV_t(x,y)\le  C\lf(\frac{{t}}{\rho(x)}\r)^{\min\{1/2,2-Q/q\}}\frac{t}{t+d(x, y)} \frac{1}{\mu(B(x,t+d(x, y)))}.$$

%$$
%\begin{cases}
%\dis |e^{-t\LV}1(x)-1|\le C\lf(\frac{\sqrt{t}}{\rho(x)}\r)^{\min\{1/2,2-Q/q\}},\\
%\dis |e^{-t\sqrt{\LV}}1(x)-1|\le C\lf(\frac{{t}}{\rho(x)}\r)^{\min\{1/2,2-Q/q\}}.
%\end{cases}
%$$

\item[{(iv)}] For any $x\in B$, the derivative of the Poisson kernel satisfies
$$
\int_{5B\setminus 4B} \left|r_B\nabla_yp^\LV_{t}(x,y)\right|^2\d\mu(y)
%&\le \lim_{\epsilon\to0}\int_{2\epsilon}^{r^2}\int_{5B\setminus 4B}|r\nabla_y h^v_{s}(x,y)|^2{\d\mu(y,s)} \\
\le C\int_{6B\setminus 3B}\left(p^\LV_{t}(x,y) \left|r_B^2\partial_t^2 p^\LV_{t}(x,y)\right| +p^\LV_{t}(x,y)^2\right){\d\mu(y)}.
$$

  \item[{(v)}]
For any $N>0$, the derivative of the Poisson semigroup satisfies
$$
\begin{cases}
\dis \left|t\partial_t e^{-t\sqrt{\LV}}1(x)\right| \le C\lf(\frac{t}{\rho(x)}\r)^{\dz}\lf(1+\frac{t}{\rho(x)}\r)^{-N},\\[8pt]
\dis \lf(\int_0^{r_B}\fint_{B} \left|t\nabla_xe^{-t\sqrt{\LV}}1\right|^2\d \mu\frac{\d t}{t}\r)^{1/2}\le C\min\lf\{\lf(\frac{r_B}{\rho(x_B)}\r)^{\dz},1\r\},
\end{cases}
$$
where $0<\dz<\min\{1/2,2-Q/q\}$.
\end{itemize}
\end{lemma}

\begin{proof}
Properties (i) and (ii) follows from \cite[Section 3]{JL2022}.

\smallskip

(iii) By the perturbation estimate of the heat kernel
$$0\le h^\L_s(x,y)-h^\LV_s(x,y)\le  C\lf(\frac{\sqrt{s}}{\sqrt{s}+\rho(x)}\r)^{2-Q/q}\frac{1}{\mu(B(x,\sqrt{s}))} \exp \left(-\frac{d(x, y)^{2}}{c s}\right)$$
from \cite[Proposition 7.13]{BDL2018} (see also \cite[Corollary 5.3]{LI}), and the Bochner subordination formula, we arrive at
\begin{align*}
0\le q_t(x,y)
&=\frac{1}{\sqrt{\pi}}\int_0^\fz\lf(\frac{t^2}{4s}\r)^{1/2}\exp\lf(-\frac{t^2}{4s}\r) \left(h^\L_s(x,y)-h^\LV_s(x,y)\right)\frac{\d s}{s} \\
& \le C\lf(\frac{{t}}{\rho(x)}\r)^{\min\{1/2,2-Q/q\}}\frac{t}{t+d(x, y)} \frac{1}{\mu(B(x,t+d(x, y)))},
\end{align*}
where the last inequality is due to a classical argument from \cite[pages 1145--1146]{JL2022}.

\smallskip

(iv) A classical Caccioppoli argument tells us that
\begin{align*}
&\int_{6B\setminus 3B} \left|\vz(y)\nabla_yp^\LV_{t}(x,y)\right|^2\d\mu(y)  \\
%&\ =\int_{6B\setminus 3B} \langle\vz(y)^2\nabla_yp^\LV_{t}(x,y),\nabla_yp^\LV_{t}(x,y)\rangle\d\mu(y) \\
&\ =\int_{6B\setminus 3B} \left\langle\nabla_y(\vz(y)^2p^\LV_{t}(x,y)),\nabla_yp^\LV_{t}(x,y)\right\rangle\d\mu(y)
 -\int_{6B\setminus 3B} 2 \left\langle p^\LV_{t}(x,y)\nabla_y\vz(y),\vz(y)\nabla_yp^\LV_{t}(x,y)    \right\rangle \d\mu(y) \\
%&\ =\int_{6B\setminus 3B} \vz(y)^2p^\LV_{t}(x,y)\L p^\LV_{t}(x,y)\d\mu(y)
% -\int_{6B\setminus 3B} 2\langle p^\LV_{t}(x,y)\nabla_y\vz(y),\vz(y)\nabla_yp^\LV_{t}(x,y)\d\mu(y) \\
%&\ \le \int_{6B\setminus 3B} \vz(y)^2p^\LV_{t}(x,y)\LV p^\LV_{t}(x,y)\d\mu(y)
% -\int_{6B\setminus 3B} 2\langle p^\LV_{t}(x,y)\nabla_y\vz(y),\vz(y)\nabla_yp^\LV_{t}(x,y)\d\mu(y) \\
%&\ = \int_{6B\setminus 3B} \vz(y)^2p^\LV_{t}(x,y) \partial_t^2 p^\LV_{t}(x,y)\d\mu(y)
% -\int_{6B\setminus 3B} 2\langle p^\LV_{t}(x,y)\nabla_y\vz(y),\vz(y)\nabla_yp^\LV_{t}(x,y)\d\mu(y) \\
&\ \le  \int_{6B\setminus 3B} \vz(y)^2p^\LV_{t}(x,y) \left|\partial_t^2 p^\LV_{t}(x,y)\right|\d\mu(y)
 +2\int_{6B\setminus 3B}  p^\LV_{t}(x,y) \left|\nabla_y\vz(y)\right|^2 \d\mu(y) +\frac12  \int_{6B\setminus 3B} \left|\vz(y)\nabla_yp^\LV_{t}(x,y)\right|^2\d\mu(y).%\mathrm{LHS}.
\end{align*}
If we pick a Lipschitz function $\vz$ supported in $6B\setminus3B$ such that $\vz=1$ on $5B\setminus4B$ and $|\nabla_x\vz|\le C/r_B$,
then the desired result follows readily.

\smallskip

(v) The argument of the last property may be found in \cite[Section 5]{JL2022}. %; see \cite{LI} for the heat semigroup case.
This completes the proof.
%\begin{align*}
% \int_{X}|p^v_t(x,y)-p_t(x,y)|\d\mu(y)|u_s(x)|
%&\le C \int_0^\fz\frac{t}{s^{1/2}}\exp\lf(-\frac{t^2}{4s}\r)| e^{-s\LV}1(x)-1|\frac{\d s}{s}|u_s(x)| \\
%%&\le C \int_0^\fz\frac{t}{s^{1/2}}\exp\lf\{-\frac{t^2}{4s}\r\}\lf(\frac{\sqrt{s}}{\rho(x)}\r)^{\min\{1/2,2-Q/q\}}\frac{\d s}{s}|u_s(x)| \\
%&\le C\lf(\frac{t}{\rho(x)}\r)^{\min\{1/2,2-Q/q\}}|u_s(x)|\to0,\quad t\to0.
%\end{align*}
\end{proof}
%
%\subsection{Estimates for the Poisson kernel}
%\hskip\parindent
%Due to the perturbation of $V$, the Poisson kernel associated to $\LV$ and its time derivatives
%admit the Poisson upper bound with an additional polynomial decay (see~\cite{JL2022}),
%namely, for any $k\in\{0\}\cup\nn$ and all $N>0$,  there exists a constant $C=C(k,N)>0$ such that
%$$
% |t^k\partial_t^kp_{t}^\LV(x, y)|\le C \frac{t}{t+d(x, y)} \frac{1}{\mu(B(x,t+d(x, y)))}
%\left(1+\frac{t+d(x,y)}{\rho(x)}\right)^{-N}.\eqno(PUB)
%$$
%%We refer the reader to \cite{CLYZ2022,CDLY2017,LTZ,YYZ2009,YYZ2010-1}
%%and references therein for more related results about the Schr\"odinger operator and its application such as heat kernel \cite{Ku2000,LL2011,WY2016},
%%function space \cite{CDLSY2017,DGMTZ2005,JJY2011,JYY2012,YYZ2010-2}, boundary value problem \cite{LSTW2023,STY2018,SW2022}.
%
%
%
%
%
%
\subsection{Estimates for the heat/Green kernel on $X\times\rr$}
\hskip\parindent
To show the regularity of the $\mathbb{L}$-harmonic function on $X\times\rr$, let us introduce some notation on the product space.
We endow the product space $X\times\rr$ with the product metric
$$\sqrt{d(x_1,y_1)^2+|x_2-y_2|^2},$$
and the product measure $\d\mu(x_1,x_2)=\d\mu(x_1)\d x_2$.
Since $\d\mu(x_1)$ is doubling, the measure $\d\mu(x_1, x_2)$ is also doubling.
%Moreover, as the Poincar\'{e} inequality holds on $(X,d,\mu)$
%$$
%\left(\fint_{B(x,r)}\left|f-f_{B(x,r)}\right|^{2} \d \mu\right)^{1 / 2}
%	\leq C_{P} r\left(\fint_{B(x,r)}\left|\nabla_{x} f\right|^{2} \d \mu\right)^{1 / 2},
%$$
%{\color{red}we also have a Poincar\'{e} inequality on the product space $(X\times\rr,d,\mu)$, i.e., it holds}
%$$
%\left(\fint_{s}^{s+r}\fint_{B(x,r)}\left|g-g_{B(x,r)\times(s,s+r)}\right|^{2} \d \mu\d t\right)^{1 / 2}
%	\leq C_{P} r\left(\fint_{s}^{s+r}\fint_{B(x,r)}\left|\nabla g\right|^{2} \d \mu\d t\right)^{1 / 2}
%$$
%for all $g\in W^{1,2}({B(x,r)\times(s,s+r)})$, where $\nabla=(\nabla_x,\partial_t)$.
Let $h^\L_t(x_1,y_1)$ and $h^{\mathbb{L}-\LV}_t(x_2,y_2)$ denote the heat kernel on $X$ and $\rr$, respectively,
and it is well known that the heat kernel $h_t^{\mathbb{L}-V}((x_1,x_2),(y_1,y_2))$ on $X\times\rr$ is given by (see \cite{LZ2018} for example)
\begin{align}\label{pdct}
h_t^{\mathbb{L}-V}((x_1,x_2),(y_1,y_2))=h^{\L}_t(x_1,y_1)h^{\mathbb{L}-\LV}_t(x_2,y_2).
\end{align}
Notice that, if both $h^\L_t(x_1,y_1)$ and $h^{\mathbb{L}-\LV}_t(x_2,y_2)$ admits $(LY)$,
so dose the direct product heat kernel $h_t^{\mathbb{L}-V}((x_1,x_2),(y_1,y_2))$ by \eqref{pdct}.
The fundamental solution (or say Green function) of the elliptic operator $-\partial^2_t+\mathcal{L}$ (or $-\partial^2_t+\LV$) in $X\times \rr$
can be obtained from the heat kernel through the spectral theory
%$$
%\begin{cases}
%\displaystyle \Gamma_{\mathbb{L}-V}((x_1,x_2),(y_1,y_2))=\int_0^\fz h_t^{\mathbb{L}-V}((x_1,x_2),(y_1,y_2))\d t;\\
%\displaystyle \Gamma_{\mathbb{L}}((x_1,x_2),(y_1,y_2))=\int_0^\fz h_t^{\mathbb{L}}((x_1,x_2),(y_1,y_2))\d t;
%\end{cases}
%$$
$$
\begin{cases}
\displaystyle G_{\mathbb{L}-V}((x_1,x_2),(y_1,y_2))=\int_0^\fz h_t^{\mathbb{L}-V}((x_1,x_2),(y_1,y_2))\d t;\\[8pt]
\displaystyle G_{\mathbb{L}}((x_1,x_2),(y_1,y_2))=\int_0^\fz h_t^{\mathbb{L}}((x_1,x_2),(y_1,y_2))\d t.
\end{cases}
$$

In what follows, without causing confusion, we still use the notation $d$, $B$, $\mu$, $h_t$ and $G$
to denote the metric, ball, measure, heat kernel and Green function on the produce space for simplicity, namely
$$\displaystyle
\begin{cases}
\displaystyle d((x_1,x_2),(y_1,y_2))=\sqrt{d(x_1,y_1)^2+|x_2-y_2|^2};\\[4pt]
\displaystyle B((x_1,x_2),R)=\big\{(y_1,y_2)\in X\times\rr: d((y_1,y_2),(x_1,x_2))<R\big\}; \\[4pt]
\displaystyle \mu(B((x_1,x_2),R))=\int_{B((x_1,x_2),R)}\d\mu(y_1, y_2);\\[4pt]
\displaystyle h_t^{\mathbb{L}-V}((x_1,x_2),(y_1,y_2))=h^{\L}_t(x_1,y_1)h^{\mathbb{L}-\LV}_t(x_2,y_2);\\[4pt]
\displaystyle G_{\mathbb{L}-V}((x_1,x_2),(y_1,y_2))=\int_0^\fz h_t^{\mathbb{L}-V}((x_1,x_2),(y_1,y_2))\d t;\\[4pt]
\displaystyle G_{\mathbb{L}}((x_1,x_2),(y_1,y_2))=\int_0^\fz h_t^{\mathbb{L}}((x_1,x_2),(y_1,y_2))\d t.
\end{cases}
$$
The reader should distinguish these notation in the context carefully.
\begin{lemma}\label{lem-2}
Let $(X,d,\mu,\mathscr{E})$ be a complete Dirichlet metric measure space satisfying a doubling property $(D)$ and admitting an $L^2$-Poincar\'{e} inequality $(P_2)$.
Suppose that $0\le V\in A_\infty(X)\cap RH_q(X)$ for some $q>\max\{ Q/2,1\}$.
The following statements are valid.
\begin{itemize}
 \item[{(i)}]
 For any $N>0$, the heat kernel $h^{\mathbb{L}}_t((x_1,x_2),(y_1,y_2))$ admits the Gaussian upper bound
 $$
 h^{\mathbb{L}}_t((x_1,x_2),(y_1,y_2))\le  \frac{C}{\mu(B((x_1,x_2),\sqrt{t}))} \exp \left(-\frac{d((x_1,x_2),(y_1,y_2))^{2}}{c t}\right)
  \left(1+\frac{\sqrt{t}}{\rho(x_1)}\right)^{-N},
$$
and hence the Green function $\Gamma_{\mathbb{L}}((x_1,x_2),(y_1,y_2))$ admits the upper bound
$$
 G_{\mathbb{L}}((x_1,x_2),(y_1,y_2))\le C\frac{d((x_1,x_2),(y_1,y_2))^2}{\mu(B((x_1,x_2),d((x_1,x_2),(y_1,y_2))))}
  \left(1+\frac{d((x_1,x_2),(y_1,y_2))}{\rho(x_1)}\right)^{-N}.
$$

 \item[{(ii)}]
 There exists a constant $0<\az<1$ such that, if $d((x_1,x_2),(y_1,y_2))<\sqrt{t}$, then
 \begin{align*}
 &\left|h^{\mathbb{L}}_t((x_1,x_2),(z_1,z_2))-h^{\mathbb{L}}_t((y_1,y_2),(z_1,z_2))\right| \\
 &\ \le C\lf(\frac{d((x_1,x_2),(y_1,y_2))}{\sqrt{t}}\r)^\az \frac{1}{\mu(B((z_1,z_2),\sqrt{t}))}
 \lf[\exp \left(-\frac{d((x_1,x_2),(z_1,z_2))^{2}}{ct}\right) + \exp \left(-\frac{d((y_1,y_2),(z_1,z_2))^{2}}{ct}\right)\r],
 \end{align*}
  and hence the Green function $G_{\mathbb{L}}((x_1,x_2),(y_1,y_2))$ admits
 $$
\big|G_{\mathbb{L}}((x_1,x_2),(z_1,z_2))-G_{\mathbb{L}}((y_1,y_2),(z_1,z_2))\big|
\le C \left(\frac{d((x_1,x_2),(y_1,y_2))}{d((x_1,x_2),(z_1,z_2))}\right)^\az \frac{d((x_1,x_2),(z_1,z_2))^2}{\mu(B((x_1,x_2),d((x_1,x_2),(z_1,z_2))))},
$$
whenever $d((x_1,x_2),(y_1,y_2))<d((x_1,x_2),(z_1,z_2))/2$.
%
%
% \item[{(iii)}]  The perturbation kernel $q_t(x,y)=p^\L_t(x,y)-p^\LV_t(x,y)$ satisfies
% $$0\le q_t(x,y)=p^\L_t(x,y)-p^\LV_t(x,y)\le  C\lf(\frac{{t}}{\rho(x)}\r)^{\min\{1/2,2-Q/q\}}\frac{t}{t+d(x, y)} \frac{1}{\mu(B(x,t+d(x, y)))}.$$
%
%%$$
%%\begin{cases}
%%\dis |e^{-t\LV}1(x)-1|\le C\lf(\frac{\sqrt{t}}{\rho(x)}\r)^{\min\{1/2,2-Q/q\}},\\
%%\dis |e^{-t\sqrt{\LV}}1(x)-1|\le C\lf(\frac{{t}}{\rho(x)}\r)^{\min\{1/2,2-Q/q\}}.
%%\end{cases}
%%$$
%
%\item[{(iv)}] For any $x\in B$, the derivative of the Poisson semigroup satisfies
%$$
%\int_{5B\setminus 4B} |r_B\nabla_yp^\LV_{t}(x,y)|^2\d\mu(y)
%%&\le \lim_{\epsilon\to0}\int_{2\epsilon}^{r^2}\int_{5B\setminus 4B}|r\nabla_y h^v_{s}(x,y)|^2{\d\mu(y,s)} \\
%\le C\int_{6B\setminus 3B}\left(p^\LV_{t}(x,y)|r_B^2\partial_t^2 \nabla_yp^\LV_{t}(x,y)| +p^\LV_{t}(x,y)^2\right){\d\mu(y)}.
%$$
%
%  \item[{(v)}]
%For any $N>0$, the derivative of the Poisson semigroup satisfies
%$$
%\begin{cases}
%\dis |t\partial_t e^{-t\sqrt{\LV}}1(x)| \le C\lf(\frac{t}{\rho(x)}\r)^{\dz}\lf(1+\frac{t}{\rho(x)}\r)^{-N},\\
%\dis \lf(\int_0^{r_B}\fint_{B}|t\nabla_xe^{-t\sqrt{\LV}}1|^2\d \mu\frac{\d t}{t}\r)^{1/2}\le C\min\lf\{\lf(\frac{r_B}{\rho(x_B)}\r)^{\dz},1\r\},
%\end{cases}
%$$
%where $0<\dz<\min\{1/2,2-Q/q\}$.
\end{itemize}
\end{lemma}

\begin{proof}
(i) Note that
 $$
 h^{\mathbb{L}}_t((x_1,x_2),(y_1,y_2))=h^{\LV}_t(x_1,y_1)h^{\mathbb{L}-\LV}_t(x_2,y_2),
$$
and hence by the doubling property
\begin{align*}
h^{\mathbb{L}}_t((x_1,x_2),(y_1,y_2))
&\le  \frac{C}{\mu(B(x_1,\sqrt{t}))} \exp \left(-\frac{d(x_1,y_1)^{2}}{c t}\right)
  \left(1+\frac{\sqrt{t}}{\rho(x_1)}\right)^{-N}\frac{1}{\sqrt{4\pi t}}\exp \left(-\frac{d(x_2,y_2)^{2}}{4t}\right) \\
&\le \frac{C}{\mu(B((x_1,x_2),\sqrt{t}))} \exp \left(-\frac{d((x_1,x_2),(y_1,y_2))^{2}}{c t}\right)   \left(1+\frac{\sqrt{t}}{\rho(x_1)}\right)^{-N}.
\end{align*}
The upper bound of the Green function $G_{\mathbb{L}}((x_1,x_2),(y_1,y_2))$ follows from
$$G_{\mathbb{L}}((x_1,x_2),(y_1,y_2))=\int_0^\fz h_t^{\mathbb{L}}((x_1,x_2),(y_1,y_2))\d t$$
and the reverse doubling property with $n>1$.

\smallskip

(ii) From Lemma \ref{lem-1}, we deduce that for any $d((x_1,x_2),(y_1,y_2))<\sqrt{t}$,
\begin{align*}
&\left|h^{\mathbb{L}}_t((x_1,x_2),(z_1,z_2))-h^{\mathbb{L}}_t((y_1,y_2),(z_1,z_2))\right| \\
&\ \le \left  |h^\LV_t(x_1,z_1)-h^\LV_t(y_1,z_1)\right|h^{\mathbb{L}-\LV}_t(x_2,z_2) +h^\LV_t(y_1,z_1) \left |h^{\mathbb{L}-\LV}_t(x_2,z_2)-h^{\mathbb{L}-\LV}_t(y_2,z_2)\right| \\
&\ \le C\lf(\frac{d(x_1,y_1)}{\sqrt{t}}\r)^\az  \frac{1}{\mu(B(z_1,\sqrt{t}))\sqrt{t}}\exp \left(-\frac{d(x_1,z_1)^{2}}{ct}\right)\exp \left(-\frac{|x_2-z_2|^{2}}{4t}\right)\\
&\ \ +C\lf(\frac{|x_2-y_2|}{\sqrt{t}}\r)^\az  \frac{1}{\mu(B(z_1,\sqrt{t}))\sqrt{t}}\exp \left(-\frac{d(y_1,z_1)^{2}}{ct}\right)\exp \left(-\frac{|y_2-z_2|^{2}}{ct}\right)\\
&\ \le C\lf(\frac{d((x_1,x_2),(y_1,y_2))}{\sqrt{t}}\r)^\az \frac{1}{\mu(B((z_1,z_2),\sqrt{t}))}
 \lf[\exp \left(-\frac{d((x_1,x_2),(z_1,z_2))^{2}}{ct}\right) + \exp \left(-\frac{d((y_1,y_2),(z_1,z_2))^{2}}{ct}\right)\r].
 \end{align*}

It remains to show the H\"{o}lder continuity of the Green function.
To this end, one writes
\begin{align*}
\big|G_\mathbb{L}((x_1,x_2),(z_1,z_2))-G_\mathbb{L}((y_1,y_2),(z_1,z_2))\big |
&\le \int_0^\fz  \left|h^\mathbb{L}_t((x_1,x_2),(z_1,z_2))-h^\mathbb{L}_t((y_1,y_2),(z_1,z_2))\right |\d t \\
&= \lf\{\int_0^{4d((x_1,x_2),(y_1,y_2))^2}+\int_{4d((x_1,x_2),(y_1,y_2))^2}^{d((x_1,x_2),(z_1,z_2))^2}+\int_{d((x_1,x_2),(z_1,z_2))^2}^\fz\r\}\cdots \d t\\
&=:I_1+I_2+I_3.
\end{align*}
For the first term $I_1$, we arrive at
\begin{align*}
I_1
&\le \int_0^{4d((x_1,x_2),(y_1,y_2))^2} \frac{C}{\mu(B((z_1,z_2),\sqrt{t}))} \exp \left(-\frac{d((x_1,x_2),(z_1,z_2))^{2}}{ct}\right)\d t \\
&\le  C\int_0^{4d((x_1,x_2),(y_1,y_2))^2} \left(\frac{t}{d((x_1,x_2),(z_1,z_2))}\right)^\az \frac{\d t}{\mu(B((z_1,z_2),d((x_1,x_2),(z_1,z_2))))}  \\
&\le  C \left(\frac{d((x_1,x_2),(y_1,y_2))}{d((x_1,x_2),(z_1,z_2))}\right)^\az \frac{d((x_1,x_2),(z_1,z_2))^2}{\mu(B((x_1,x_2),d((x_1,x_2),(z_1,z_2))))}.
\end{align*}
To estimate the second term $I_2$, it holds by the H\"{o}lder estimate on the heat kernel
\begin{align*}
I_2
&\le C\int_{4d((x_1,x_2),(y_1,y_2))^2}^{d((x_1,x_2),(z_1,z_2))^2} \lf(\frac{d((x_1,x_2),(y_1,y_2))}{\sqrt{t}}\r)^\az \frac{1}{\mu(B((z_1,z_2),\sqrt{t}))}
 \exp \left(-\frac{d((x_1,x_2),(z_1,z_2))^{2}}{ct}\right)\d t \\
&\le  C \left(\frac{d((x_1,x_2),(y_1,y_2))}{d((x_1,x_2),(z_1,z_2))}\right)^\az \frac{d((x_1,x_2),(z_1,z_2))^2}{\mu(B((x_1,x_2),d((x_1,x_2),(z_1,z_2))))}.
\end{align*}
For the last term $I_3$, we deduce from the H\"{o}lder estimate again and $(RD)$ with $n>1$ that
\begin{align*}
I_3
%&\le C\int_{d((x_1,x_2),(z_1,z_2))^2}^\fz \lf(\frac{d((x_1,x_2),(y_1,y_2))}{\sqrt{t}}\r)^\az \frac{1}{\mu(B((z_1,z_2),\sqrt{t}))}
% \exp \left(-\frac{d((x_1,x_2),(z_1,z_2))^{2}}{ct}\right)\d t \\
&\le C\int_{d((x_1,x_2),(z_1,z_2))^2}^\fz \lf(\frac{d((x_1,x_2),(y_1,y_2))}{\sqrt{t}}\r)^\az \frac{1}{\mu(B((z_1,z_2),\sqrt{t}))} \d t  \\
&\le C\int_{d((x_1,x_2),(z_1,z_2))^2}^\fz \lf(\frac{d((x_1,x_2),(y_1,y_2))}{\sqrt{t}}\r)^\az \lf(\frac{d((x_1,x_2),(z_1,z_2))}{\sqrt{t}}\r)^{n+1}
\frac{\d t}{\mu(B((x_1,x_2),d((x_1,x_2),(z_1,z_2))))}  \\
&\le  C \left(\frac{d((x_1,x_2),(y_1,y_2))}{d((x_1,x_2),(z_1,z_2))}\right)^\az \frac{d((x_1,x_2),(z_1,z_2))^2}{\mu(B((x_1,x_2),d((x_1,x_2),(z_1,z_2))))}.
\end{align*}

Collecting all estimates above lead to the desired result.
This completes the proof.
\end{proof}

\section{H\"{o}lder regularity under the critical reverse H\"{o}lder index}\label{ho}
\hskip\parindent
%In this section, two H\"{o}lder estimates of the $\mathbb{L}$-harmonic function are established.
The H\"{o}lder regularity of the $\mathbb{L}$-harmonic function on $X\times \rr$  is established in this section.
\subsection{Statement of the main result}
\begin{theorem}\label{holder}
Let $(X,d,\mu,\mathscr{E})$ be a complete Dirichlet metric measure space satisfying a doubling property $(D)$\footnote{For some technical reasons, the lower dimension $n$ is always assumed to be equal or greater than 2 in this result. Moreover, since $Q\ge n\ge2$, we have $\max\{Q/2,1\}=Q/2$.}
and admitting an $L^2$-Poincar\'{e} inequality $(P_2)$.
Assume that $u$ is a weak solution to the elliptic equation
$$\mathbb{L}u=-\partial^2_tu+{\mathscr{L}}u=-\partial^2_tu+\mathcal{L}u+Vu=0$$
in  $2B_0\subset X\times \rr$ with $B_0=B((x_0,t_0),R)$, where  $0\le V\in A_\infty(X)\cap RH_q(X)$ for some $q>Q/2$.
%Then there exists $C>0$ such that for any ball $B$ with $2B\subset \Omega$,
%$$\|u\|_{L^\infty(B)}\le C\fint_{2B}|u|\d\mu.$$
%Then $u$ is locally H\"older continuous in $2B_0$, and
Then there exists a constant $0<\az<1$ such that for any $(x,t),(y,s)\in B_0$,
$$|u(x,t)-u(y,s)|\le C{\left(\frac{d((x,t),(y,s))}{R}\right)^\az}\left(1+R^2\fint_{B(x_0,R)} V\d\mu\right) \sup_{2B_0}|u|.$$
\end{theorem}

\begin{remark}
In \cite{JL2022}, the authors proved that all $\mathscr{L}$-harmonic functions on $X$ are locally H\"{o}lder continuous
provided $0\le V\in A_\infty(X)\cap RH_q(X)$ for some $q>\max\{ Q/2,1\}$.
When applying this conclusion to the Schr\"{o}dinger equation
$$-\partial_t^2u+\mathcal{L}u+Vu=0$$
in $X\times \rr$,
we see that the natural extension $V(\cdot,t) = V(\cdot)$ forces the critical reverse H\"{o}lder order $q$ to be at least $(Q+1)/2$.
However, with the help of the Green function on $X\times \rr$, our Theorem \ref{holder} states that, even if $0\le V\in A_\infty(X)\cap RH_q(X)$ for some $q>Q/2$,
the solution to the Schr\"{o}dinger equation in $X\times \rr$ is still locally H\"older continuous.
\end{remark}

\subsection{Proof of Theorem \ref{holder}}

\begin{lemma}\label{K}
Let $(X,d,\mu,\mathscr{E})$ be a complete Dirichlet metric measure space satisfying a doubling property $(D)$.
Suppose that  $0\le V\in A_\infty(X)\cap RH_q(X)$ for some $q>Q/2$.
Then there exists a constant $C>0$ such that
$$
\int_{B(x,r)}\frac{d(y,x)^2}{\mu(B(x,d(y,x)))}V(y)\d\mu(y)\le C\frac{r^2}{\mu(B(x,r))}\int_{B(x,r)}V(y)\d\mu(y).
$$
\end{lemma}
\begin{proof}
%We first claim that the potential $V$ admits the Kato condition, namely, for any $0<r<R<\fz$, it holds that
It is easy to see that the potential $V$ admits the Kato condition, namely, for any $0<r<R<\fz$, it holds that
$$
 \frac{r^2}{\mu(B(x,r))}\int_{B(x,r)}V\d\mu\le C\lf(\frac{r}{R}\r)^{2-Q/q}\frac{R^2}{\mu(B(x,R))}\int_{B(x,R)}V \d\mu. \eqno(K)
$$
%In fact, apply the H\"{o}lder inequality and the fact $V\in RH_q(X)$ to obtain
%\begin{align*}
%\frac{r^2}{\mu(B(x,r))}\int_{B(x,r)}V\d\mu
%&\le r^2\lf(\frac{\mu(B(x,R))}{\mu(B(x,r))}\r)^{1/q}\lf(\fint_{B(x,R)}V^q\d\mu\r)^{1/q} \\
%&\le Cr^2\lf(\frac{\mu(B(x,R))}{\mu(B(x,r))}\r)^{1/q}\fint_{B(x,R)}V\d\mu \\
%&\le  C\lf(\frac{r}{R}\r)^{2-Q/q}\frac{R^2}{\mu(B(x,R))}\int_{B(x,R)}V\d\mu,
%\end{align*}
%where the last inequality is due to $(D)$.
With this Kato condition in hand, the rest proof follows from a classical annuli argument, and is left to the interested reader.
%
% we deduce that
%\begin{align*}
%\int_{B(x,r)}\frac{d(y,x)^2}{\mu(B(x,d(y,x)))}V(y)\d\mu(y)
%&= \sum_{j=0}^\fz\int_{2^{-j-1}r\le d(y,x)<2^{-j}r}\frac{d(y,x)^2}{\mu(B(x,d(y,x)))}V(y)\d\mu(y) \\
%&\le C\sum_{j=0}^\fz\frac{(2^{-j}r)^2}{\mu(B(x,2^{-j}r))}\int_{ B(x,2^{-j}r)}V(y)\d\mu(y) \\
%&\le C\frac{r^2}{\mu(B(x,r))}\int_{ B(x,r)}V(y)\d\mu(y)\sum_{j=0}^\fz2^{-j(2-Q/q)}\\
%&\le C\frac{r^2}{\mu(B(x,r))}\int_{ B(x,r)}V(y)\d\mu(y),
%\end{align*}
%which completes the proof.
\end{proof}

\begin{proof}[Proof of Theorem \ref{holder}]
{\bf Step 1:} the identity for weak solutions.
Choose a Lipschitz function $\phi$ supported in $3B_0/2$ such that $\phi=1$ on $4B_0/3$ and $|\nabla_{x,t}\phi|\le C/R$.
Recall that $G_{\mathbb{L}-V}$ is the fundamental solution of the elliptic operator $-\partial^2_t+\mathcal{L}$ in $X\times \rr$.
Then we have
\begin{align*}
u\phi(x,t)
&=\int_{X\times\rr}  \left\langle\nabla_{z,r}G_{\mathbb{L}-V}((x,t),(z,r)),\nabla_{z,r}(u\phi)(z,r)\right\rangle{\d} \mu(z,r)\\
%&=\int_{X\times\rr}\langle\nabla_{z,r}\Gamma_0((x,t),(z,r)),\nabla_{z,r} \phi(z,r)\rangle u(z,r){\d} \mu(z,r)\\
%&\ +\int_{X\times\rr}\langle\phi(z,r)\nabla_{z,r}\Gamma_0((x,t),(z,r)),\nabla_{z,r} u(z,r)\rangle{\d} \mu(z,r) \\
%&=\int_{X\times\rr}\langle\nabla_{z,r}\Gamma_0((x,t),(z,r)),\nabla_{z,r} \phi(z,r)\rangle u(z,r){\d} \mu(z,r)\\
%&\ +\int_{X\times\rr}\langle\nabla_{z,r}(\phi\Gamma_0)((x,t),(z,r)),\nabla_{z,r} u(z,r)\rangle{\d} \mu(z,r) \\
%&\ -\int_{X\times\rr}\langle\Gamma_0((x,t),(z,r))\nabla_{z,r}\phi(z,r),\nabla_{z,r} u(z,r)\rangle{\d} \mu(z,r) \\
&=\int_{X\times\rr}  \left \langle\nabla_{z,r}G_{\mathbb{L}-V}((x,t),(z,r)),\nabla_{z,r} \phi(z,r)\right \rangle u(z,r){\d} \mu(z,r)  -\int_{X\times\rr}G_{\mathbb{L}-V}((x,t),(z,r)) \left\langle\nabla_{z,r}\phi(z,r),\nabla_{z,r} u(z,r)\right\rangle{\d} \mu(z,r)\\
&\quad
 -\int_{X\times\rr}\phi(z,r)G_{\mathbb{L}-V}((x,t),(z,r)) u(z,r)V(z){\d} \mu(z,r)
\end{align*}
and hence, for any $(x,t),(y,s)\in B_0$,
\begin{align*}
u(x,t)-u(y,s)
&=\int_{X\times\rr}\left\langle\nabla_{z,r}G_{\mathbb{L}-V}((x,t),(z,r))-\nabla_{z,r}G_{\mathbb{L}-V}((y,s),(z,r)),\nabla_{z,r} \phi(z,r)\right\rangle u(z,r){\d} \mu(z,r) \\
&\ -\int_{X\times\rr}  \left[G_{\mathbb{L}-V}((x,t),(z,r))-G_{\mathbb{L}-V}((y,s),(z,r))\right]  \left\langle\nabla_{z,r}\phi(z,r),\nabla_{z,r} u(z,r)\right\rangle{\d} \mu(z,r)  \\
&\ -\int_{X\times\rr}\phi(z,r)\left[G_{\mathbb{L}-V}((x,t),(z,r))-G_{\mathbb{L}-V}((y,s),(z,r))\right] u(z,r)V(z){\d} \mu(z,r)\\
&=:I_1+I_2+I_3.
\end{align*}

{\bf Step 2:} estimate the term $I_1$.
It follows from the H\"{o}lder inequality and the Caccioppoli inequality that
\begin{align*}
|I_1|
&\le \frac{C}{R}\lf(\int_{(3B_0/2)\setminus(4B_0/3)} \left|\nabla_{z,r}(G_{\mathbb{L}-V}((x,t),(z,r))-G_{\mathbb{L}-V}((y,s),(z,r)))\right|^2\d\mu(z,r)\r)^{1/2}
 \lf(\int_{3B_0/2}|u|^2{\d}\mu \r)^{1/2} \\
&\le \frac{C}{R^2}\lf(\int_{(2B_0)\setminus(5B_0/4)} \big|G_{\mathbb{L}-V}((x,t),(z,r))-G_{\mathbb{L}-V}((y,s),(z,r))\big|^2\d\mu(z,r)\r)^{1/2}
 \lf(\int_{3B_0/2}|u|^2{\d}\mu \r)^{1/2} \\
%&\le \frac{C}{R^2}\lf(\int_{(2B_0)\setminus(5B_0/4)}\lf(\frac{d((x,t),(y,s))}{R}\r)^{2\az_1}\sup_{(\xi,\bz)\in 6B_0/5}\Gamma_0((\xi,\bz),(z,r))^2\d\mu(z,r)\r)^{1/2}
% \lf(\int_{3B_0/2}|u|^2{\d}\mu \d r\r)^{1/2} \\
&\le \frac{C}{R^2}\lf(\int_{(2B_0)\setminus(5B_0/4)}\lf(\frac{d((x,t),(y,s))}{R}\r)^{2\az_1}\frac{d((z,r),(x_0,r_0))^4\d\mu(z,r)}{\mu(B((x_0,r_0),d((z,r),(x_0,r_0))))^2}\r)^{1/2}
 \lf(\int_{3B_0/2}|u|^2{\d}\mu \r)^{1/2} \\
&\le C\lf(\frac{d((x,t),(y,s))}{R}\r)^{\az_1}\sup_{3B_0/2}|u|,
\end{align*}
where in the third step we invoke the H\"{o}lder estimate on $(-\partial^2_t+\mathcal{L})$-harmonic function (see \cite[Corollary 3.3]{St1996}),
and the size condition for the fundamental solution of the elliptic operator $-\partial^2_t+\mathcal{L}$ in $X\times \rr$ (see \cite[Theorem 6.1]{CGL2021}),
i.e., there exists a constant $0<\az_1<1$ such that for any $(x,t),(y,s)\in B_0$,
\begin{align}\label{green-holder}
\big|G_{\mathbb{L}-V}((x,t),(z,r))-G_{\mathbb{L}-V}((y,s),(z,r))\big|
&\le C\lf(\frac{d((x,t),(y,s))}{R}\r)^{\az_1}\sup_{(\xi,\bz)\in 6B_0/5}G_{\mathbb{L}-V}((\xi,\bz),(z,r))  \nonumber\\
&\le C\lf(\frac{d((x,t),(y,s))}{R}\r)^{\az_1}\frac{d((z,r),(x_0,r_0))^2}{\mu(B((x_0,r_0),d((z,r),(x_0,r_0))))^2}.
\end{align}
%
%and in the fourth step we used
%$$
%\Gamma_0((x,t),(y,s))\approx\frac{d((x,t),(y,s))^2}{\mu(B((x,t),d((x,t),(y,s))))};
%$$

{\bf Step 3:} estimate the term $I_2$.  There holds from \eqref{green-holder} and the Caccioppoli inequality again that
\begin{align*}
|I_2|
%&\le \int_{(3B_0/2)\setminus(4B_0/3)}|G_{\mathbb{L}-V}((x,t),(z,r))-G_{\mathbb{L}-V}((y,s),(z,r))||\nabla_{z,r}\phi(z,r)||\nabla_{z,r} u(z,r)|{\d}\mu(z,r) \\
&\le\frac{C}{R}\int_{(3B_0/2)\setminus(4B_0/3)}\lf(\frac{d((x,t),(y,s))}{R}\r)^{\az_1}
 \frac{d((z,r),(x_0,r_0))^2}{\mu(B((x_0,r_0),d((z,r),(x_0,r_0))))}|\nabla_{z,r} u(z,r)|{\d}\mu(z,r) \\
%&\le C\lf(\frac{d((x,t),(y,s))}{R}\r)^{\az_1}\fint_{3B_0/2}|R\nabla  u|{\d}\mu  \d r \\
&\le C\lf(\frac{d((x,t),(y,s))}{R}\r)^{\az_1}\lf(\fint_{3B_0/2}|R\nabla_{z,r}  u|^2{\d}\mu \r)^{1/2}
%&\le C\lf(\frac{d((x,t),(y,s))}{R}\r)^{\az_1}\lf(\fint_{2B_0}| u|^2{\d}\mu \d r\r)^{1/2} \\
\le C\lf(\frac{d((x,t),(y,s))}{R}\r)^{\az_1}\sup_{2B_0}|u|.
\end{align*}

{\bf Step 4:} estimate the term $I_3$.
Finally, rewrite the term $I_3$ as
\begin{align*}
|I_3|
%&\le \int_{3B_0/2}|\Gamma_0((x,t),(z,r))-\Gamma_0((y,s),(z,r))|V(z)\d\mu(z)\d r \sup_{3B_0/2}|u|\\
&\le \int_{d((z,r),(y,s))<5R/2} \big|G_{\mathbb{L}-V}((x,t),(z,r))-G_{\mathbb{L}-V}((y,s),(z,r))\big |V(z)\d\mu(z,r) \sup_{3B_0/2}|u|\\
&\le \lf\{\int_{d((z,r),(y,s))<Kd((x,t),(y,s))}+\int_{Kd((x,t),(y,s))\le d((z,r),(y,s))<5R/2 }\r\}\cdots \d\mu(z,r) \sup_{3B_0/2}|u|\\
&=:I_{31}+I_{32},
%&\le \lf\{\int_{|z-y|<K|x-y|}+\int_{K|x-y|\le |z-y|<5r/2 }\r\}\cdots\d\pi(y)\sup_{3B/2}|u|=I_{41}+I_{42},
\end{align*}
where $K=8R^{1/2}d((x,t),(y,s))^{-1/2}>4\sqrt{2}$.

{\bf Substep 4.1:} estimate the term $I_{31}$.
For the local part, one may use the size condition for $G_{\mathbb{L}-V}$, and $(RD)$ to deduce
\begin{align*}
I_{31}
&\le \int_{d((z,r),(y,s))<Kd((x,t),(y,s))}  \big|G_{\mathbb{L}-V}((x,t),(z,r))-G_{\mathbb{L}-V}((y,s),(z,r))\big|V(z) \d\mu(z,r) \sup_{3B_0/2}|u|\\
%&\le C\int_{d((z,r),(x,t))<2Kd((x,t),(y,s))}\Gamma_0((x,t),(z,r)) V(z) \d\mu(z,r)\sup_{3B_0/2}|u|\\
%&\ +C\int_{d((z,r),(y,s))<Kd((x,t),(y,s))}\Gamma_0((y,s),(z,r))V(z) \d\mu(z,r) \sup_{3B_0/2}|u|\\
%&\le C\sup_{(\xi,\bz)\in B_0}\int_{d((z,r),(\xi,\bz))<2Kd((x,t),(y,s))}\frac{d((z,r),(\xi,\bz))^2}{\mu(B((\xi,\bz),d((z,r),(\xi,\bz))))} V(z) \d\mu(z,r)\sup_{3B_0/2}|u|\\
%&\le C\sup_{(\xi,\bz)\in B_0}\int_{|r-\bz|<2Kd((x,t),(y,s))}\int_{d(z,\xi)<2Kd((x,t),(y,s))}
% \frac{d((z,r),(\xi,\bz))}{\mu(B(\xi,d((z,r),(\xi,\bz))))} V(z) \d\mu(z,r)\sup_{3B_0/2}|u|\\
%&\le C\sup_{(\xi,\bz)\in B_0}\int_{|r-\bz|<2Kd((x,t),(y,s))}\int_{d(z,\xi)<2Kd((x,t),(y,s))}
% \frac{d(z,\xi)+|r-\bz|}{\mu(B(\xi,d(z,\xi)+|r-\bz|))} V(z) \d\mu(z,r)\sup_{3B_0/2}|u|\\
&\le C\sup_{(\xi,\bz)\in B_0}\int_0^{2Kd((x,t),(y,s))}\int_{d(z,\xi)<2Kd((x,t),(y,s))}
 \frac{d(z,\xi)+r}{\mu(B(\xi,d(z,\xi)+r))} V(z) \d\mu(z,r)\sup_{3B_0/2}|u|\\
&\le C\sup_{\xi\in B(x_0,R)}\int_{d(z,\xi)<2Kd((x,t),(y,s))}
 \lf[\int_0^{2Kd((x,t),(y,s))}\lf(\frac{d(z,\xi)}{d(z,\xi)+r}\r)^{n-1}\d r\r]\frac{d(z,\xi)V(z) \d\mu(z)}{\mu(B(\xi,d(z,\xi)))} \sup_{3B_0/2}|u|.
\end{align*}
Since $n\ge2$, the inner integral can be calculated as
$$
\int_0^{2Kd((x,t),(y,s))}\lf(\frac{d(z,\xi)}{d(z,\xi)+r}\r)^{n-1}\d r
\le C(n)d(z,\xi)\log\lf(\frac{4Kd((x,t),(y,s))}{d(z,\xi)}\r).
$$
Therefore we deduce from the Kato condition $(K)$ that
\begin{align*}
I_{31}
&\le C\sup_{\xi\in B(x_0,R)}\int_{d(z,\xi)<2Kd((x,t),(y,s))}
 \log\lf(\frac{4Kd((x,t),(y,s))}{d(z,\xi)}\r)\frac{d(z,\xi)^2}{\mu(B(\xi,d(z,\xi)))}V(z) \d\mu(z) \sup_{3B_0/2}|u|\\
%&\le  C\sup_{\xi\in B(x_0,R)}\sum_{j=-1}^\fz(j+3)\int_{2^{-j-1}Kd((x,t),(y,s))\le d(z,\xi)<2^{-j}Kd((x,t),(y,s))}
%  \frac{d(z,\xi)^2}{\mu(B(\xi,d(z,\xi)))}V(z) \d\mu(z) \sup_{3B_0/2}|u|\\
&\le  C\sup_{\xi\in B(x_0,R)}\sum_{j=-1}^\fz (j+3)\frac{[2^{-j}Kd((x,t),(y,s))]^2}{\mu(B(\xi,2^{-j}Kd((x,t),(y,s))))}\int_{d(z,\xi)<2^{-j}Kd((x,t),(y,s))}
  V(z) \d\mu(z) \sup_{3B_0/2}|u| \\
&\le  C\sup_{\xi\in B(x_0,R)}\sum_{j=-1}^\fz \frac{j+3}{2^{j(2-Q/q)}}\frac{[2Kd((x,t),(y,s))]^2}{\mu(B(\xi,2Kd((x,t),(y,s))))}\int_{d(z,\xi)<2Kd((x,t),(y,s))}
  V(z) \d\mu(z) \sup_{3B_0/2}|u| \\
%&\le C\sup_{\xi\in B(x_0,R)}\frac{[2Kd((x,t),(y,s))]^2}{\mu(B(\xi,2Kd((x,t),(y,s))))}\int_{d(z,\xi)<2Kd((x,t),(y,s))} V(z) \d\mu(z)\sup_{3B_0/2}|u|\\
&\le C\lf(\frac{2Kd((x,t),(y,s))}{32R}\r)^{2-Q/q}\sup_{\xi\in B(x_0,R)}\frac{(32R)^2}{\mu(B(\xi,32R))}\int_{d(z,\xi)<32R} V(z) \d\mu(z)\sup_{3B_0/2}|u|\\
&\le C\lf(\frac{d((x,t),(y,s))}{R}\r)^{1-Q/2q}\frac{R^2}{\mu(B(x_0,R))}\int_{B(x_0,R)} V \d\mu\sup_{3B_0/2}|u|.
\end{align*}
%{\color{red}Otherwise $1\le n<2$,} it holds that
%\begin{align*}
%\int_0^{2Kd((x,t),(y,s))}\lf(\frac{d(z,\xi)}{d(z,\xi)+r}\r)^{n-1}\d r
%&\le C(n)d(z,\xi)\lf(\frac{2Kd((x,t),(y,s))}{d(z,\xi)}\r)^{2-n},
%\end{align*}
%which together with $q>Q/n$ implies that
%\begin{align*}
%I_{31}
%&\le C\sup_{\xi\in B(x_0,R)}\int_{d(z,\xi)<2Kd((x,t),(y,s))}
% \lf(\frac{2Kd((x,t),(y,s))}{d(z,\xi)}\r)^{2-n}\frac{d(z,\xi)^2}{\mu(B(\xi,d(z,\xi)))}V(z) \d\mu(z) \sup_{3B_0/2}|u|\\
%&\le  C\sup_{\xi\in B(x_0,R)}\sum_{j=0}^\fz 2^{j(2-n)}\int_{2^{-j}Kd((x,t),(y,s))\le d(z,\xi)<2^{-j+1}Kd((x,t),(y,s))}
%  \frac{d(z,\xi)^2}{\mu(B(\xi,d(z,\xi)))}V(z) \d\mu(z) \sup_{3B_0/2}|u|\\
%&\le  C\sup_{\xi\in B(x_0,R)}\sum_{j=0}^\fz 2^{j(2-n)}\frac{[2^{-j+1}Kd((x,t),(y,s))]^2}{\mu(B(\xi,2^{-j+1}Kd((x,t),(y,s))))}\int_{d(z,\xi)<2^{-j+1}Kd((x,t),(y,s))}
%  V(z) \d\mu(z) \sup_{3B_0/2}|u| \\
%&\le  C\sup_{\xi\in B(x_0,R)}\sum_{j=0}^\fz 2^{-j(n-Q/q)}\frac{[2Kd((x,t),(y,s))]^2}{\mu(B(\xi,2Kd((x,t),(y,s))))}\int_{d(z,\xi)<2Kd((x,t),(y,s))}
%  V(z) \d\mu(z) \sup_{3B_0/2}|u| \\
%&\le C\sup_{\xi\in B(x_0,R)}\frac{[2Kd((x,t),(y,s))]^2}{\mu(B(\xi,2Kd((x,t),(y,s))))}\int_{d(z,\xi)<2Kd((x,t),(y,s))} V(z) \d\mu(z)\sup_{3B_0/2}|u|\\
%&\le C\lf(\frac{d((x,t),(y,s))}{R}\r)^{1-Q/2q}\frac{R^2}{\mu(B(x_0,R))}\int_{B(x_0,R)} V \d\mu\sup_{3B_0/2}|u|.
%\end{align*}

{\bf Substep 4.2:} estimate the term $I_{32}$.
For the annular part, it follows from \eqref{green-holder} %, the size condition for $\Gamma_0$, %\eqref{q3},
and an argument similar to the term $I_{31}$ that
\begin{align*}
I_{32}
&\le \int_{Kd((x,t),(y,s))\le d((z,r),(y,s))<5R/2}\big|G_{\mathbb{L}-V}((x,t),(z,r))-G_{\mathbb{L}-V}((y,s),(z,r))\big|V(z) \d\mu(z,r) \sup_{3B_0/2}|u|  \\
&\le C\int_{Kd((x,t),(y,s))\le d((z,r),(y,s))<5R/2}\lf(\frac{{d((x,t),(y,s))}}{K{d((x,t),(y,s))}/4}\r)^{\az_1}
 \frac{d((z,r),(y,s))^2}{\mu(B((y,s),d((z,r),(y,s))))}V(z) \d\mu(z,r) \sup_{3B_0/2}|u| \\
%&\ \ \ \ \ \ \ \  \ \  \times \sup_{(\xi,\bz)\in B((y,s),Kd((x,t),(y,s))/2)}\Gamma_0((\xi,\bz),(z,r))V(z) \d\mu(z,r) \sup_{3B_0/2}|u|  \\
%&\le C\lf(\frac{d((x,t),(y,s))}{R}\r)^{\az_1/2} \int_{d((z,r),(y,s))<3R}
%  \frac{d((z,r),(y,s))^2}{\mu(B((y,s),d((z,r),(y,s))))}V(z) \d\mu(z,r) \sup_{3B_0/2}|u|  \\
%&\le C\lf(\frac{d((x,t),(y,s))}{R}\r)^{\az_1/2} \frac{(3R)^2}{\mu(B(y,3R)))}\int_{B(y,3R)}   V(z) \d\mu(z) \sup_{3B_0/2}|u|  \\
&\le C\lf(\frac{d((x,t),(y,s))}{R}\r)^{\az_1/2}\frac{R^2}{\mu(B(x_0,R))}\int_{B(x_0,R)} V \d\mu\sup_{3B_0/2}|u|.
\end{align*}
%\begin{align*}
%I_{32}
%&\le \int_{Kd((x,t),(y,s))\le d((z,r),(y,s))<5R/2}|\Gamma_0((x,t),(z,r))-\Gamma_0((y,s),(z,r))|V(z) \d\mu(z)\d r \sup_{3B_0/2}|u|  \\
%&\le C\int_{Kd((x,t),(y,s))\le d((z,r),(y,s))<5R/2}\lf(\frac{{d((x,t),(y,s))}}{K{d((x,t),(y,s))}/4}\r)^{\az_1}
% \sup_{B((y,s),Kd((x,t),(y,s))/2)}\Gamma_0((\cdot,\cdot),(z,r))V(z) \d\mu(z)\d r \sup_{3B_0/2}|u| \\
%%&\ \ \ \ \ \ \ \  \ \  \times \sup_{(\xi,\bz)\in B((y,s),Kd((x,t),(y,s))/2)}\Gamma_0((\xi,\bz),(z,r))V(z) \d\mu(z)\d r \sup_{3B_0/2}|u|  \\
%&\le C\lf(\frac{d((x,t),(y,s))}{R}\r)^{\az_1/2} \int_{d((z,r),(y,s))<3R}
%  \frac{d((z,r),(y,s))^2}{\mu(B((y,s),d((z,r),(y,s))))}V(z) \d\mu(z)\d r \sup_{3B_0/2}|u|  \\
%%&\le C\lf(\frac{d((x,t),(y,s))}{R}\r)^{\az_1/2} \frac{(3R)^2}{\mu(B(y,3R)))}\int_{B(y,3R)}   V(z) \d\mu(z) \sup_{3B_0/2}|u|  \\
%&\le C\lf(\frac{d((x,t),(y,s))}{R}\r)^{\az_1/2}\frac{R^2}{\mu(B(x_0,R))}\int_{B(x_0,R)} V \d\mu\sup_{3B_0/2}|u|.
%\end{align*}

{\bf Step 5:} completion of the proof.
Combining these estimates on $I_1$, $I_2$, $I_{31}$ and $I_{32}$, we arrive at
$$
|u(x,t)-u(y,s)|\le C\lf(\frac{d((x,t),(y,s))}{R}\r)^{\min\{1-Q/2q,\az_1/2\}}\lf(1+\frac{R^2}{\mu(B(x_0,R))}\int_{B(x_0,R)} V \d\mu\r)\sup_{2B_0}|u|,
$$
which completes the proof of Theorem \ref{holder} by letting $\az={\min\{1-Q/2q,\az_1/2\}}$.
\end{proof}

\subsection{Another H\"{o}lder continuity}\label{sec4.3}
\hskip\parindent
In the proof of Theorem \ref{holder}, the Green function $G_{\mathbb{L}-V}$ and the reverse H\"{o}lder property of $V$ play a key role.
In fact, we can employ the Green function related to $V$ to show the H\"{o}lder continuity of the $\mathbb{L}$-harmonic function as follows.

\begin{theorem}\label{holder1}
Let $(X,d,\mu,\mathscr{E})$ be a complete Dirichlet metric measure space satisfying a doubling property $(D)$ and admitting an $L^2$-Poincar\'{e} inequality $(P_2)$.
Assume that $u$ is a weak solution to the elliptic equation
$$\mathbb{L}u=-\partial^2_tu+{\mathscr{L}}u=-\partial^2_tu+\mathcal{L}u+Vu=0$$
in  $8B_0\subset X\times \rr$ with $B_0=B((x_0,t_0),R)$, where  $0\le V\in A_\infty(X)\cap RH_q(X)$ for some $q>\max\{ Q/2,1\}$.\footnote{Since the Kato condition $(K)$ is not used in Theorem \ref{holder1}, we can relax the critical reverse H\"older index from $Q/2$ to $\max\{Q/2,1\}$ here.}
%Then there exists $C>0$ such that for any ball $B$ with $2B\subset \Omega$,
%$$\|u\|_{L^\infty(B)}\le C\fint_{2B}|u|\d\mu.$$
%Then $u$ is locally H\"older continuous in $2B_0$, and
Then there exists a constant $0<\az<1$ such that for any $(x,t),(y,s)\in B_0$,
$$|u(x,t)-u(y,s)|\le C{\left(\frac{d((x,t),(y,s))}{R}\right)^\az} \sup_{8B_0}|u|.$$
\end{theorem}

\begin{proof}
Take a smooth function $\phi$ supported in $7B_0$ such that $\phi=1$ on $6B_0$ and $|\nabla\phi|\le C/R$.
It follows from the $\mathbb{L}$-harmonic of $u$ that
\begin{align*}
u\phi(x,t)
&=\int_{X\times\rr} \left\langle\nabla_{z,h}G_\mathbb{L}((x,t),(z,h)),\nabla_{z,h}(u\phi)(z,h)\right\rangle\d\mu(z,h)
 +\int_{X\times\rr}G_\mathbb{L}((x,t),(z,h))(u\phi)(z,h)V(y)\d\mu(z,h)\\
%&=\int_{X\times\rr}\langle\nabla_{z,h}\Gamma_V((x,t),(z,h)),u\nabla_{z,h}\phi(z,h)\rangle\d\mu(z,h) \\
%&\ + \int_{X\times\rr}\langle\nabla_{z,h}\Gamma_V((x,t),(z,h)),\phi\nabla_{z,h}u(z,h)\rangle\d\mu(z,h)\\
%&\ +\int_{X\times\rr}\Gamma_V((x,t),(z,h))(u\phi)(z,h)V(y)\d\mu(z,h) \\
%&=\int_{X\times\rr}\langle\nabla_{z,h}\Gamma_V((x,t),(z,h)),u\nabla_{z,h}\phi(z,h)\rangle\d\mu(z,h) \\
%&\ - \int_{X\times\rr}\Gamma_V((x,t),(z,h))\langle\nabla_{z,h}\phi(z,h),\nabla_{z,h}u(z,h)\rangle\d\mu(z,h)\\
%&\ +\int_{X\times\rr}\langle\nabla_{z,h}(\phi\Gamma_V)((x,t),(z,h)),\nabla_{z,h}u(z,h)\rangle\d\mu(z,h) \\
%&\ +\int_{X\times\rr}\phi(z,h)\Gamma_V((x,t),(z,h))u(z,h)V(y)\d\mu(z,h) \\
&=\int_{X\times\rr} \left\langle\nabla_{z,h}G_\mathbb{L}((x,t),(z,h)),u\nabla_{z,h}\phi(z,h)\right\rangle\d\mu(z,h)
 - \int_{X\times\rr}G_\mathbb{L}((x,t),(z,h))\left\langle\nabla_{z,h}\phi(z,h),\nabla_{z,h}u(z,h)\right\rangle \d\mu(z,h).
\end{align*}
Therefore, invoking the Caccioppoli inequality, and the H\"{o}lder estimate on $G_\mathbb{L}$ from Section \ref{estimate-kernel}, we deduce that for any $(x,t),(y,s)\in B_0$,
\begin{align*}
|u(x,t)-u(y,s)|
%&\le \frac CR\int_{7B_0\setminus 6B_0}|\nabla_{z,h}G_\mathbb{L}((x,t),(z,h))-\nabla_{z,h}G_\mathbb{L}((y,s),(z,h))||u(z,h)|\d\mu(z,h) \\
%&\ + \frac CR\int_{7B_0\setminus 6B_0}|G_\mathbb{L}((x,t),(z,h))-G_\mathbb{L}((y,s),(z,h))||\nabla_{z,h}u(z,h)|\d\mu(z,h) \\
&\le \frac CR\lf(\int_{7B_0\setminus 6B_0} \left|\nabla_{z,h}G_\mathbb{L}((x,t),(z,h))-\nabla_{z,h}G_\mathbb{L}((y,s),(z,h))\right|^2\d\mu(z,h)\r)^{1/2}
 \lf(\int_{7B_0\setminus 6B_0}|u(z,h)|^2\d\mu(z,h)\r)^{1/2} \\
&\ + \frac CR\lf(\int_{7B_0\setminus 6B_0}\big|G_\mathbb{L}((x,t),(z,h))-G_\mathbb{L}((y,s),(z,h))\big|^2\d\mu(z,h)\r)^{1/2}
 \lf(\int_{7B_0\setminus 6B_0}|\nabla_{z,h}u(z,h)|^2\d\mu(z,h)\r)^{1/2} \\
&\le \frac{C}{R^2}\lf(\int_{8B_0\setminus 5B_0} \big|G_\mathbb{L}((x,t),(z,h))-G_\mathbb{L}((y,s),(z,h))\big|^2\d\mu(z,h)\r)^{1/2}  \lf(\int_{8B_0}|u|^2\d\mu\r)^{1/2} \\
&\le \frac{C}{R^2}\lf(\int_{8B_0\setminus 5B_0}\left(\frac{d((x,t),(y,s))}{d((x,t),(z,h))}\right)^{2\az} \frac{d((x,t),(z,h))^4}{\mu(B((x,t),d((x,t),(z,h))))^2}\d\mu(z,h)\r)^{1/2}
  \lf(\int_{8B_0}|u|^2\d\mu\r)^{1/2} \\
&\le C\left(\frac{d((x,t),(y,s))}{R}\right)^{\az} \sup_{8B_0}|u|,
\end{align*}
which completes the proof.
\end{proof}

\begin{remark}
On the one hand, different from Theorem \ref{holder}, we do not need the assumption on the lower dimension $n\ge2$ in Theorem \ref{holder1}.
The lower dimension $n\ge2$ in Theorem \ref{holder} is to control the the integral of $V$ over a ball with respect to the Green measure.
This integral does not appear in Theorem \ref{holder1}, and hence we can relax the lower dimension $n$ to be greater than 1.
On the other hand, comparing Theorem \ref{holder} with Theorem \ref{holder1},
note  that the critical term
$$\left(1+R^2\fint_{B(x_0,R)} V\d\mu\right)$$
occurs in Theorem \ref{holder} but not in Theorem \ref{holder1}.
In fact, this term can be absorbed by $\sup_{2B_0}|u|$
(because of an additional polynomial decay; see Lemma \ref{mean} below) at the expense of a lager ball; see \cite[Lemma 6.1]{Sh1999} for the Laplacian case.
\end{remark}

\section{Liouville theorem under the critical reverse H\"{o}lder index}\label{improved}
%\medskip
%\noindent
%{\color{red}RMK. Should we prove the Liouville theorem, i.e., Theorem \ref{thm:Liouville} in this section, and show Theorems \ref{thm:BMO-Dirichlet} and \ref{thm:CMO-Dirichlet} in the next section. Finally, we prove the Morrey traces in the last section.
%
%BMO and CMO can be organized together as the endpoint results;  Morrey can be regarded as the non-endpoint case.
%
%
%}
%
%\medskip
\subsection{Statement of the main result}
\hskip\parindent
%Before stating the main result, we first recall the definition of $\mathbb{L}$-harmonic functions on the upper half-space.
%We say that $u\in W^{1,2}(X\times\rr_+)$ is an $\mathbb{L}$-harmonic function
%in $X\times\rr_+$, if
%$$\int_0^\fz\int_X \partial_t u\partial_t \phi \d \mu\d t+\int_0^\fz\int_X\langle \nabla_x u, \nabla_x \phi\rangle \d \mu\d t+\int_0^\fz\int_X Vu\phi \d \mu\d t=0$$
%holds for all Lipschitz functions $\phi$ with compact support in $X\times\rr_+$.
%
%
The following Theorem \ref{Liouville} is the main result of this section, whose proof is given in Section \ref{sec3.3}.
\begin{theorem}\label{Liouville}
Let $(X,d,\mu,\mathscr{E})$ be a complete Dirichlet metric measure space satisfying a doubling property $(D)$ and admitting an $L^2$-Poincar\'{e} inequality $(P_2)$.
Assume that $u$ is a weak solution to the elliptic equation
$$\mathbb{L}u=-\partial_t^2+{\mathscr{L}} u=-\partial_t^2u+\L u+V u=0$$
in $X \times \rr$, where $0\le V\in A_\infty(X)\cap RH_{q}(X)$ for some $q>\max\{Q/2,1\}$.
If there exist some $(x_0,t_0)\in X\times \rr$ and $\bz>0$ such that
\begin{align}\label{global-control}
\int_{X\times\rr} \frac{|u(x,t)|^2}{(1+d((x,t),(x_0,t_0)))^{\bz}\mu(B((x_0,t_0),1+d((x,t),(x_0,t_0))))}\d\mu(x,t)\le C{(x_0,t_0,\bz)}<\infty,
\end{align}
then $u=0$ in $X\times \rr$.
\end{theorem}

%\newpage
%
%\begin{remark}
%In \cite{DYZ2014} and \cite{JL2022}, when the critical reverse H\"{o}lder order of $V$ is one half of the upper dimension of the underlying space,
%Duong-Yan-Zhang and Jiang--Li established the Liouville property of the $\mathbb{L}$-harmonic function on $\rn$ and on $X$, respectively.
%The former utilizes the Green function, and the latter utilizes the heat kernel.
%If we consider the Schr\"{o}dinger equation
%$$\mathbb{L}u=-\partial_t^2u+\mathcal{L}u+Vu=0,$$
%on $X\times \rr_+$,
%the dimension of the underlying space increases by 1 automatically.
%However, noticing that the non-negative potential $V(\cdot)$ is time independent,
%we believe that the critical reverse H\"{o}lder order of $V$ is also time independent.
%To remove the time dimension 1, our main tool is the theory of the elliptic operator on the product space $X\times \rr$.
%With the help of the upper bound for $\Gamma_V$ (see Proposition \ref{upperbound}),
%we can follow the argument of Duong-Yan-Zhang in \cite[Lemma 2.7]{DYZ2014} to prove the $\mathbb{L}$-harmonic function on $X\times \rr$
%enjoys the same Liouville property under the requirement $q>Q/2$ as in the $X$ case; see the proof of Proposition \ref{Liouville} below.
%%
%%
%%
%%we establish the upper bound with an additional exponential decay for the fundamental solution of the Schr\"{o}dinger operator $-\partial^2_t+\L+V$ in $X\times \rr$;
%%see Proposition \ref{upperbound} above.
%%Then following the argument as in \cite[Lemma 2.7]{DYZ2014},
%%we can directly prove the Schr\"{o}dinger function on $X\times \rr$ admits
%\end{remark}

\begin{remark}
In \cite{DYZ2014} and \cite{JL2022}, when the critical reverse H\"{o}lder order of $V$ is one half of the upper dimension of the underlying space,
Duong-Yan-Zhang and Jiang--Li established the Liouville property of the $\mathbb{L}$-harmonic function on $\rn$ and on $X$, respectively.
The former utilizes the Green function, and the latter utilizes the heat kernel.
If we consider the Schr\"{o}dinger equation
$$-\partial_t^2u+\mathcal{L}u+Vu=0,$$
in $X\times \rr_+$,
the dimension of the underlying space increases by 1 automatically.
However, noticing that the non-negative potential $V(\cdot)$ is time independent,
we believe that the critical reverse H\"{o}lder order of $V$ is also time independent in the product space setting.
To remove the time dimension 1, our main tool is the theory of the elliptic operator on the product space $X\times \rr$.
With the help of the H\"{o}lder continuity and a variety of mean value property for the $\mathbb{L}$-harmonic function (see Theorem \ref{holder} and Lemma \ref{mean}),
we can prove the $\mathbb{L}$-harmonic function on $X\times \rr$
enjoys the same Liouville property under the requirement $q>\max\{Q/2,1\}$ as in the $X$ case,
whose proof is different from that of Duong-Yan-Zhang \cite[Lemma 2.7]{DYZ2014} and Jiang--Li in \cite[Proposition 4.4]{JL2022};
see the proof of Proposition \ref{Liouville} below.
%
%
%
%we establish the upper bound with an additional exponential decay for the fundamental solution of the Schr\"{o}dinger operator $-\partial^2_t+\L+V$ in $X\times \rr$;
%see Proposition \ref{upperbound} above.
%Then following the argument as in \cite[Lemma 2.7]{DYZ2014},
%we can directly prove the Schr\"{o}dinger function on $X\times \rr$ admits
\end{remark}

\begin{remark}
Here it is remarkable that, in Proposition \ref{Liouville}, we prove the Liouville theorem for the solution to the Schr\"{o}dinger equation in $X\times \rr$.
However, the relationship between the BMO function and the Carleson measure has to do with the upper half-space $X\times \rr_+$.
We can establish another Liouville property for the $\mathbb{L}$-harmonic function on $X\times \rr_+$ with zero trace directly; see \cite{YZ2021} for the heat equation.
More precisely, if an $\mathbb{L}$-harmonic function on $X\times \rr_+$ with zero trace satisfies
$$\int_{X\times\rr_+} \frac{|u(x,t)|}{(1+d((x,t),(x_0,t_0)))^{\bz}\mu(B((x_0,t_0),1+d((x,t),(x_0,t_0))))}\d\mu(x,t)\le C{(x_0,t_0,\bz)}<\infty$$
for some $(x_0,t_0)\in X\times \rr$ and $\bz>0$, then $u=0$  in $X\times \rr_+$,
where $0\le V\in A_\infty(X)\cap RH_{q}(X)$ for some $q>\max\{Q/2,1\}$;
see Theorem \ref{Liouville1} below for more details.
\end{remark}

%\subsection{H\"{o}lder continuity}\label{ho}

\subsection{Proof of Theorem \ref{Liouville}}\label{sec3.3}
\hskip\parindent
To show the Liouville theorem, we first establish a variety of mean value property for the $\mathbb{L}$-harmonic function on $X\times\rr$.

\begin{lemma}\label{mean}
Let $(X,d,\mu,\mathscr{E})$ be a complete Dirichlet metric measure space satisfying a doubling property $(D)$ and admitting an $L^2$-Poincar\'{e} inequality $(P_2)$.
Assume that $u$ is a weak solution to the elliptic equation
$$\mathbb{L}u=-\partial^2_tu+\mathscr{L}u=-\partial^2_tu+\mathcal{L}u+Vu=0$$
in $2B_0\subset X\times\rr$ with $B_0=B((x_0,t_0),R)$, where $0\le V\in A_\infty(X)\cap RH_q(X)$ for some $q>\max\{Q/2,1\}$.
For any $N>0$, then there exists a constant $C>0$ independent of $u$ and $B_0$ such that for any $(x,t)\in B_0$,
$$
|u(x,t)|\le C\fint_{2B_0}|u|\d\mu \lf(1+\frac{R}{\rho(x)}\r)^{-N}.
$$
\end{lemma}

\begin{proof}
%{\bf Step 1:} the identity for weak solutions.
%Fix a ball $B_0=B((x_0,t_0),R)\subset X\times \rr$, and
Pick a Lipschitz function $\phi$ supported in $3B_0/2$ such that $\phi=1$ on $4B_0/3$ and $|\nabla_{x,t}\phi|\le C/R$.
Then we have
\begin{align*}
u\phi(x,t)
&=\int_{X\times\rr}\left\langle\nabla_{y,s}G_\mathbb{L}((x,t),(y,s)),\nabla_{y,s}(u\phi)(y,s)\right\rangle\d\mu(y,s)
 +\int_{X\times\rr}G_\mathbb{L}((x,t),(y,s))(u\phi)(y,s)V(y)\d\mu(y,s)\\
%&=\int_{X\times\rr}\langle\nabla_{y,s}G_\mathbb{L}((x,t),(y,s)),u\nabla_{y,s}\phi(y,s)\rangle\d\mu(y,s) \\
%&\ + \int_{X\times\rr}\langle\nabla_{y,s}G_\mathbb{L}((x,t),(y,s)),\phi\nabla_{y,s}u(y,s)\rangle\d\mu(y,s)\\
%&\ +\int_{X\times\rr}G_\mathbb{L}((x,t),(y,s))(u\phi)(y,s)V(y)\d\mu(y,s) \\
%&=\int_{X\times\rr}\langle\nabla_{y,s}G_\mathbb{L}((x,t),(y,s)),u\nabla_{y,s}\phi(y,s)\rangle\d\mu(y,s) \\
%&\ - \int_{X\times\rr}G_\mathbb{L}((x,t),(y,s))\langle\nabla_{y,s}\phi(y,s),\nabla_{y,s}u(y,s)\rangle\d\mu(y,s)\\
%&\ +\int_{X\times\rr}\langle\nabla_{y,s}(\phiG_\mathbb{L})((x,t),(y,s)),\nabla_{y,s}u(y,s)\rangle\d\mu(y,s) \\
%&\ +\int_{X\times\rr}\phi(y,s)G_\mathbb{L}((x,t),(y,s))u(y,s)V(y)\d\mu(y,s) \\
&=\int_{X\times\rr} \left\langle\nabla_{y,s}G_\mathbb{L}((x,t),(y,s)),u\nabla_{y,s}\phi(y,s)\right\rangle\d\mu(y,s)
 - \int_{X\times\rr}G_\mathbb{L}((x,t),(y,s)) \left\langle\nabla_{y,s}\phi(y,s),\nabla_{y,s}u(y,s)\right\rangle\d\mu(y,s).
\end{align*}
%{\bf Step 2:} the pointwise bound for $u$.
%One the one hand, by H\"{o}lder's inequality, we obtain that for any $(x,t)\in B_0$,
By Caccioppoli's inequality and the size condition for $G_\mathbb{L}$ from Section \ref{estimate-kernel}, it follows that for any $(x,t)\in B_0$,
\begin{align*}
|u(x,t)|
%&\le \int_{X\times\rr}|\nabla_{y,s}G_\mathbb{L}((x,t),(y,s))||u\nabla_{y,s}\phi(y,s)|\d\mu(y,s) \\
%&\ +\int_{X\times\rr}G_\mathbb{L}((x,t),(y,s))|\nabla_{y,s}\phi(y,s)||\nabla_{y,s}u(y,s)|\d\mu(y,s) \\
%&\le \frac{C}{R}\int_{5B_0\setminus4B_0}|\nabla_{y,s}G_\mathbb{L}((x,t),(y,s))||u(y,s)|\d\mu(y,s) \\
%&\ +\frac{C}{R}\int_{5B_0\setminus4B_0}G_\mathbb{L}((x,t),(y,s))|\nabla_{y,s}u(y,s)|\d\mu(y,s) \\
&\le \frac{C}{R}\lf(\int_{(4B_0/3)\setminus (3B_0/2)} \left|\nabla_{y,s}G_\mathbb{L}((x,t),(y,s))\right|^2\d\mu(y,s)\r)^{1/2}
 \lf(\int_{(4B_0/3)\setminus (3B_0/2)}|u(y,s)|^2\d\mu(y,s)\r)^{1/2} \\
&\ +\frac{C}{R}\lf(\int_{(4B_0/3)\setminus (3B_0/2)}\big|G_\mathbb{L}((x,t),(y,s))\big|^2\d\mu(y,s)\r)^{1/2}\lf(\int_{(4B_0/3)\setminus (3B_0/2)} \left|\nabla_{y,s}u(y,s)\right|^2\d\mu(y,s)\r)^{1/2} \\
%&\le \frac{C}{R^2}\lf(\int_{6B_0\setminus3B_0}|G_\mathbb{L}((x,t),(y,s))|^2\d\mu(y,s)\r)^{1/2}\lf(\int_{6B_0\setminus3B_0}|u(y,s)|^2\d\mu(y,s)\r)^{1/2} \\
%&\ +\frac{C}{R^2}\lf(\int_{6B_0\setminus3B_0}|G_\mathbb{L}((x,t),(y,s))|^2\d\mu(y,s)\r)^{1/2}\lf(\int_{6B_0\setminus3B_0}|u(y,s)|^2\d\mu(y,s)\r)^{1/2}
&\le \frac{C}{R^2}\lf(\int_{(5B_0/3)\setminus (5B_0/4)}\big|G_\mathbb{L}((x,t),(y,s))\big|^2\d\mu(y,s)\r)^{1/2}
 \lf(\int_{2B_0}|u|^2\d\mu\r)^{1/2} \\
&\le C\fint_{2B_0}|u|\d\mu \lf(1+\frac{R}{\rho(x)}\r)^{-N},
\end{align*}
where the last line is a direct consequence of the mean value property.
%One the other hand, for any $N>0$, it holds by the estimate on the heat kernel $h_r^v(x,y)$ from \cite{JL2022}
%\begin{align*}
%\Gamma_V((x,t),(y,s))
%&= \int_0^\fz h_r^v(x,y)\frac{1}{(4\pi r)^{1/2}} \exp\lf(-\frac{|t-s|^2}{4r}\r)\d r \\
%&\le \int_0^\fz \frac{C}{\mu(B((x,t),\sqrt{r}))}\exp\lf(-\frac{d(x,y)^2+|t-s|^2}{cr}\r)\lf(1+\frac{\sqrt{r}}{\rho(x)}\r)^{-N}\d r \\
%&\le C\frac{d((x,t),(y,s))^2}{\mu(B((x,t),d((x,t),(y,s))))}\lf(1+\frac{d((x,t),(y,s))}{\rho(x)}\r)^{-N}.
%\end{align*}
%%Though this bound has the polynomial decay merely, it is sufficient to prove the Liouville theorem.
%Therefore we obtain by this estimate that
%$$
%|u(x,t)|
%\le \frac{C}{R^2}\lf(\int_{(5B_0/3)\setminus (5B_0/4)}\frac{R^4}{\mu(B_0)^2}\lf(1+\frac{R}{\rho(x)}\r)^{-2N} \d\mu(y,s)\r)^{1/2}
% \int_{2B_0}\frac{|u|}{\mu(B_0)^{1/2}}\d\mu\le C\fint_{2B_0}|u|\d\mu \lf(1+\frac{R}{\rho(x)}\r)^{-N},
%$$
%which completes the proof.
This completes the proof.
\end{proof}

\begin{remark}
In the proof of Lemma \ref{mean}, the Green function $G_\mathbb{L}$ plays a key role.
 Here another proof is provided but all pertinent results are stated without detail.
We first obtain by Fefferman-Phong's inequality and Caccioppoli's inequality that
$$\lf(\fint_{B_0}|u|^2\d\mu\r)^{1/2} \le C\lf(\fint_{2B_0}|u|^2\d\mu\r)^{1/2}\lf(1+\frac{R}{\rho(x)}\r)^{-\frac{1}{k_0+1}}.$$
Then the iterative technique shows that for any $N>0$,
$$\lf(\fint_{2B_0}|u|^2\d\mu\r)^{1/2} \le C\lf(\fint_{4B_0}|u|^2\d\mu\r)^{1/2}\lf(1+\frac{R}{\rho(x)}\r)^{-N}.$$
Finally the classical mean value property tells us for any $(x,t)\in B_0$,
$$
|u(x,t)|\le C\fint_{4B_0}|u|\d\mu \lf(1+\frac{R}{\rho(x)}\r)^{-N}.
$$
Though the domain on the RHS of the above estimate is $4B_0$, it is sufficient to prove the Liouville theorem.
\end{remark}

\begin{proof}[Proof of Theorem \ref{Liouville}]
We first claim that the assumption \eqref{global-control} is well-posedness.
Indeed, if this property holds for some $(x_0,t_0)\in X\times \rr$, then it holds that for all $(y,s)\in X\times \rr$
%\begin{align*}
%&\int_{X\times\rr} \frac{|u(x,t)|}{(1+d((x,t),(y,s)))^{\bz}\mu(B((y,s),1+d((x,t),(y,s))))}\d\mu(x)\d t\\
%&\ \le {\int_{X\times\rr\setminus B((y,s),2d((y,s),(x_0,t_0)))}} \frac{|u(x,t)|}{(1+d((x,t),(y,s)))^{\bz}\mu(B((y,s),1+d((x,t),(y,s))))}\d\mu(x)\d t \\
%&\ \ +C\|u\|_{L^\fz(B((y,s),2d((y,s),(x_0,t_0))))}\\
%&\ \le C{\int_{X\setminus B(z,2d(x,z))}} \frac{|u(x,t)|}{(1+d((x,t),(x_0,t_0)))^{\bz}\mu(B((x_0,t_0),1+d((x,t),(x_0,t_0))))}\d\mu(x)\d t \\
%&\ \ +C\|u\|_{L^\fz(B((y,s),2d((y,s),(x_0,t_0))))} <\infty.
%\end{align*}
%
\begin{align*}
&\int_{X\times\rr} \frac{|u(x,t)|^2}{(1+d((x,t),(y,s)))^{\bz}\mu(B((y,s),1+d((x,t),(y,s))))}\d\mu(x,t)\\
&\ \le \lf\{\int_{X\times\rr\setminus B((y,s),2d((y,s),(x_0,t_0)))}+\int_{B((y,s),2d((y,s),(x_0,t_0)))}\r\} \cdots\d\mu(x,t)\\
&\ \le C{\int_{X\times\rr\setminus B((y,s),2d((y,s),(x_0,t_0)))}} \frac{|u(x,t)|^2}{(1+d((x,t),(x_0,t_0)))^{\bz}\mu(B((x_0,t_0),1+d((x,t),(x_0,t_0))))}\d\mu(x,t) \\
&\ \ +C\|u\|_{L^\fz(B((y,s),2d((y,s),(x_0,t_0))))} <\infty
\end{align*}
by the locally boundedness of $u$.
Next, take a ball $B_0=B((x_0,t_0),R)\subset X\times \rr$ with $R>100(1+\rho(x_0))$,
and one may invoke Lemma \ref{mean} with $N=1+\bz$ to obtain that for all $x\in B(x_0,\rho(x_0))$ and $t_0-1<t<t_0+1$,
\begin{align*}
|u(x,t)|
%&\le C\fint_{2B_0}|u|\d\mu\d s \lf(1+\frac{R}{\rho(x)}\r)^{-N}\\
&\le C\lf( \int_{X\times\rr} \frac{|u(y,s)|^2}{(1+d((y,s),(x_0,t_0)))^{\bz}\mu(B((x_0,t_0),1+d((y,s),(x_0,t_0))))}\d\mu(y,s)\r)^{1/2}R^{\bz/2}\lf(\frac{\rho(x)}{R}\r)^\bz\\
&\le C(x_0,t_0,\bz)R^{-\bz/2}\to 0,\quad R\to\fz,
\end{align*}
where the last inequality is due to \eqref{critical}.
Since $(x_0,t_0)$ is arbitrary, we see that $u=0$ in the whole product space $X\times \rr$.
This completes the proof.
\end{proof}

\subsection{Another Liouville theorem}
\hskip\parindent
In Theorem \ref{Liouville}, we establish a Liouville theorem for the Schr\"{o}dinger equation in $X\times \rr$,
where the Green function associated with $V$ is used.
When we meet the upper half-space $X\times \rr_+$,
what does the Liouville theorem of Schr\"{o}dinger equation in $X\times \rr_+$  look like?
According to our knowledge, a basic analysis tool in $X\times \rr$ is the Green function, and the corresponding tool in $X\times\rr_+$ is the heat kernel or Poisson kernel.
In this subsection, we provide a Liouville theorem for the elliptic Schr\"{o}dinger equation in $X\times \rr_+$,
whose proof is inspired by the parabolic equation.

\begin{theorem}\label{Liouville1}
Let $(X,d,\mu,\mathscr{E})$ be a complete Dirichlet metric measure space satisfying a doubling property $(D)$ and admitting an $L^2$-Poincar\'{e} inequality $(P_2)$.
Assume that $u$ is a weak solution to the elliptic equation
$$\mathbb{L}u=-\partial_t^2u+{\mathscr{L}} u=-\partial_t^2u+\L u+V u=0$$
in $X \times \rr_+$ with zero trace, where $0\le V\in A_\infty(X)\cap RH_{q}(X)$ for some $q>\max\{Q/2,1\}$.
If there exist some $x_0\in X$ and $\bz>0$ such that
\begin{align*}
\int_{X\times\rr_+} \frac{|u(x,t)|^2}{(1+t+d(x,x_0))^{\bz}\mu(B(x_0,1+t+d(x,x_0)))}\d\mu(x,t)\le C{(x_0,\bz)}<\infty,
\end{align*}
then $u=0$ in $X\times \rr_+$.
\end{theorem}

\begin{proof}
{\bf Step 1:} the test functions.
For any $x_0\in X$, %let $Q:=B(x_B,r_B)$ with $r_B>>\max\{\rho(x_B),1\}$.
let $Q_0=B(x_0,R)\times (0,R)$ be a cylinder with $R>100(1+\rho(x_0))$.
We take a Lipschitz function $\varphi$ on $X$ with $\supp  \varphi \subset B(x_0,5R)$ such that $\varphi=1$ on $B(x_0,4R)$,
$|\nabla_x \varphi|\le C/R$,
and take a smooth function $\phi$ on $\rr$ with
$\supp  \phi\subset (-5R,5R)$ such that $\phi=1$ on $(-4R,4R)$, $|\partial_t\phi|<C/R$.
%Note that $w$ is locally H\"older continuous by Lemma \ref{mvp}.

\smallskip

{\bf Step 2:} the identity for weak solutions.
It follows from $-\partial_t^2u+\LV u=0$ with zero trace that for any $0<t<R$,
\begin{align*}
u\vz\phi(x,t)
%&=\int_0^t\int_X p_{t-s}^\LV(x,y)(\partial_s+\LV)(w\vz\phi)(y,s)\d\mu(y,s) \\
%&=\int_0^t\int_X p_{t-s}^\LV(x,y)(\partial_s+\LV)(w\vz\phi)(y,s)\d\mu(y,s) \\
&=-\int_0^t\int_X p_{t-s}^\LV(x,y)  \left[\partial^2_su(y,s)\varphi(y)\phi(s)+2\partial_su(y,s)\varphi(y)\partial_s\phi(s)+u(y,s)\varphi(y)\partial_s^2\phi(s)\right]\d\mu(y,s) \\
&\ +\int_0^t\int_X  \left\langle \nabla_yp_{t-s}^\LV(x,y), \varphi(y)\phi(s)\nabla_yu(y,s)+u(y,s)\phi(s)\nabla_y\varphi(y)\right\rangle\d\mu(y,s) \\
&\ +\int_0^t\int_X p_{t-s}^\LV(x,y)\varphi(y)\phi(s)V(y)u(y,s)\d\mu(y,s) \\
%&=\int_0^t\int_X p_{t-s}^\LV(x,y)\varphi(y)\phi(s)\partial_su(y,s)\d\mu(y,s) \\
%&\ +\int_0^t\int_X p_{t-s}^\LV(x,y)u(y,s)\vz(y)\partial_s\phi(s)\d\mu(y,s) \\
%&\ +\int_0^t\int_X \langle \nabla_y(p_{t-s}^\LV(x,y)\varphi(y)\phi(s)), \nabla_yu(y,s)\rangle\d\mu(y,s) \\
%&\ -\int_0^t\int_X p_{t-s}^\LV(x,y)\phi(s)\langle \nabla_y\varphi(y), \nabla_yu(y,s)\rangle\d\mu(y,s) \\
%&\ +\int_0^t\int_X \langle \nabla_yp_{t-s}^\LV(x,y), u(y,s)\phi(s)\nabla_y\vz(y)\rangle\d\mu(y,s) \\
%&\ +\int_0^t\int_X p_{t-s}^\LV(x,y)\varphi(y)\phi(s)V(y)u(y,s)\d\mu(y,s) \\
&=%\int_0^t\int_X p_{t-s}^\LV(x,y)u(y,s)\varphi(y)\partial_s\phi(s)\d\mu(y,s)
 -\int_0^t\int_X p_{t-s}^\LV(x,y)\phi(s)\left\langle \nabla_y\varphi(y), \nabla_yu(y,s)\right \rangle\d\mu(y,s)
 +\int_0^t\int_X \left\langle \nabla_yp_{t-s}^\LV(x,y), u(y,s)\phi(s)\nabla_y\vz(y)\right\rangle\d\mu(y,s),
\end{align*}
since $\supp  \partial_t\phi\subset (-5R,-4R)\cup(4R,5R)$.

\smallskip

{\bf Step 3:} the pointwise bound for $u$.
By H\"{o}lder's inequality and Caccioppoli's inequality for the Poisson kernel from Lemma \ref{lem-1}, we deduce that for any $(x,t)\in B(x_0,\rho(x_0))\times (0,R)$,
\begin{align*}
|u(x,t)|
&\le \frac{C}{R^2} \int_0^{t}\int_{B(x_0,5R)\setminus B(x_0,4R)} \left[p^\LV_{t-s}(x,y) \left|r\nabla_y u(y,s)\right|+ \left|r\nabla_yp^\LV_{t-s}(x,y)\right||u(y,s)|\right]\d\mu(y,s) \\
&\le \frac{C}{R^2} \lf(\int_0^{t}\int_{B(x_0,5R)\setminus B(x_0,4R)} \left|p^\LV_{s}(x,y)\right|^2\d\mu(y,s)\r)^{1/2}\lf(\int_{5Q_0} |R\nabla_y u(y,s)|^2\d\mu(y,s)\r)^{1/2} \\
&\ + \frac{C}{R^2} \lf(\int_0^{t}\int_{B(x_0,5R)\setminus B(x_0,4R)} \left|R\nabla_yp^\LV_{s}(x,y)\right|^2\d\mu(y,s)\r)^{1/2} \lf(\int_{6Q_0}|u(y,s)|^2\d\mu(y,s)\r)^{1/2} \\
&\le\frac{C}{R^2}\lf[\int_0^{R}\int_{B(x_0,6R)\setminus B(x_0,3R)}\left(p^\LV_{s}(x,y)  \left|R^2\partial_s^2 p^\LV_{s}(x,y)\right| +{p^\LV_{s}(x,y)^2}\right){\d\mu(y,s)}\r]^{1/2}
 \lf(\int_{6Q_0}|u(y,s)|^2\d\mu(y,s)\r)^{1/2}.
\end{align*}

\smallskip

{\bf Step 4:} control the Poisson kernel term.
Recall that $R>100(1+\rho(x_0))$ and $\rho(x)\approx \rho(x_0)$ for any $x\in B(x_0,\rho(x_0))$ (see \eqref{critical}).
From the estimate on the Poisson kernel in $(PUB)$, it follows that
%\begin{align*}
%\int_0^{36r^2}\int_{6B_0\setminus 3B_0}|h^v_{s}(x,y)|^2 {\d\mu(y)\d s}
%%&\ \le \int_0^{36r^2}\int_{6B\setminus 3B}\frac{C}{\mu(B(x,\sqrt{s}))^2}\exp\lf\{-\frac{d(x,y)^2}{cs}\r\}
%% \exp\lf\{-\epsilon\lf(1+\frac{d(x,y)\vee \sqrt{s}}{\rho(x)}\r)^{2 /(2 k_{0}+3)}\r\}{\d\mu(y)\d s}\\
%%&\ \le\int_{6Q_0\setminus 3 Q_0}\frac{C}{\mu(B(x,|t-s|^{1/2}))^2}\exp\lf\{-\frac{d(x,y)^2}{c|t-s|}\r\}\exp\lf\{-\epsilon\lf(1+\frac{r}{\rho(x)}\r)^{2 /(2 k_{0}+3)}\r\}{\d\mu(y)\d s}\\
%%&\ \le \int_0^{36r^2}\int_{6B(x_0,r)\setminus 3B(x_0,r)}\frac{C}{\mu(B(x,|t-s|^{1/2}))^2}\exp\lf\{-\frac{d(x,y)^2}{c|t-s|}\r\} \exp \lf\{-\epsilon\lf(1+\frac{r}{\rho(x)}\r)^{2 /(2 k_{0}+3)}\r\} {\d\mu(y)\d s} \\
%%&\ \le \int_0^{36r^2}\int_{6B(x_0,r)\setminus 3B(x_0,r)}\frac{C}{\mu(B(x,d(x,y)\vee |t-s|^{1/2}))^2}\exp\lf\{-\frac{d(x,y)^2}{c|t-s|}\r\}
%% \exp \lf\{-\epsilon\lf(1+\frac{r}{\rho(x)}\r)^{2 /(2 k_{0}+3)}\r\} {\d\mu(y)\d s} \\
%& \le \int_0^{36r^2}\int_{6B_0\setminus 3B_0}\frac{C}{\mu(B(x,r))^2} \exp \lf\{-\epsilon\lf(1+\frac{r}{\rho(x)}\r)^{2 /(2 k_{0}+3)}\r\} {\d\mu(y)\d s} \\
%& \le C\frac{r^2}{\mu(B(x_0,r))} \lf(\frac{\rho(x_0)}{r}\r)^{\kappa},
%\end{align*}
$$
\int_0^{R}\int_{B(x_0,6R)\setminus B(x_0,3R)}p^\LV_{s}(x,y)^2 {\d\mu(y,s)}
%&\ \le \int_0^{36r^2}\int_{6B\setminus 3B}\frac{C}{\mu(B(x,\sqrt{s}))^2}\exp\lf\{-\frac{d(x,y)^2}{cs}\r\}
% \exp\lf\{-\epsilon\lf(1+\frac{d(x,y)\vee \sqrt{s}}{\rho(x)}\r)^{2 /(2 k_{0}+3)}\r\}{\d\mu(y,s)}\\
%&\ \le\int_{6Q_0\setminus 3 Q_0}\frac{C}{\mu(B(x,|t-s|^{1/2}))^2}\exp\lf\{-\frac{d(x,y)^2}{c|t-s|}\r\}\exp\lf\{-\epsilon\lf(1+\frac{r}{\rho(x)}\r)^{2 /(2 k_{0}+3)}\r\}{\d\mu(y,s)}\\
%&\ \le \int_0^{36r^2}\int_{6B(x_0,r)\setminus 3B(x_0,r)}\frac{C}{\mu(B(x,|t-s|^{1/2}))^2}\exp\lf\{-\frac{d(x,y)^2}{c|t-s|}\r\} \exp \lf\{-\epsilon\lf(1+\frac{r}{\rho(x)}\r)^{2 /(2 k_{0}+3)}\r\} {\d\mu(y,s)} \\
%&\ \le \int_0^{36r^2}\int_{6B(x_0,r)\setminus 3B(x_0,r)}\frac{C}{\mu(B(x,d(x,y)\vee |t-s|^{1/2}))^2}\exp\lf\{-\frac{d(x,y)^2}{c|t-s|}\r\}
% \exp \lf\{-\epsilon\lf(1+\frac{r}{\rho(x)}\r)^{2 /(2 k_{0}+3)}\r\} {\d\mu(y,s)} \\
 \le \int_0^{R}\int_{B(x_0,6R)}\frac{C}{\mu(B(x_0,R))^2} \lf(\frac{\rho(x)}{R}\r)^{\bz} {\d\mu(y,s)}
 \le C\frac{R}{\mu(B(x_0,R))} \lf(\frac{\rho(x_0)}{R}\r)^{\bz},
$$
and similarly
$$
\int_0^{R}\int_{B(x_0,6R)\setminus B(x_0,3R)}p^\LV_{s}(x,y) \left|R^2\partial_s^2 p^\LV_{s}(x,y)\right|^2 {\d\mu(y,s)}
 %\le \int_0^{R}\int_{B(x_0,6R)}\frac{C}{\mu(B(x_0,R))^2} \lf(\frac{\rho(x)}{R}\r)^{2N} {\d\mu(y,s)}
 \le C\frac{R}{\mu(B(x_0,R))} \lf(\frac{\rho(x_0)}{R}\r)^{\bz}.
$$

\smallskip

{\bf Step 5:} completion of the proof.
By two estimates in Steps 3-4 above, we deduce that for any $(x,t)\in B(x_0,\rho(x_0))\times(0,R)$,
%\begin{align*}
%|u(x,t)|
%&\le \frac{C}{R^{3/2}} \lf(\frac{\rho(x_0)}{R}\r)^{\bz/2}R^{\bz/2}\left(\int_{6Q_0}\frac{|u(y,s)|^2}{R^\bz \mu(B(x_0,R))}\d\mu(y,s)\right)^{1/2} \\
%&\le C\frac{\rho(x_0)^{\kappa+1}}{r} \left(\int_{6Q_0}\frac{|w(y,s)|}{(1+d(y,x_0)\vee \sqrt{s})^{\kappa+2}\mu(B(x_0,1+d(y,x_0)\vee \sqrt{s}))}\d\mu(y,s)\right)
%\le C\frac{\rho(x_0)^{\kappa+1}}{r},
%\end{align*}
$$
|u(x,t)|
\le \frac{C}{R^{3/2}} \lf(\frac{\rho(x_0)}{R}\r)^{\bz/2}R^{\bz/2}\left(\int_{6Q_0}\frac{|u(y,s)|^2}{R^\bz \mu(B(x_0,R))}\d\mu(y,s)\right)^{1/2} \\
%\le C\frac{\rho(x_0)^{\kappa+1}}{r} \left(\int_{6Q_0}\frac{|w(y,s)|}{(1+d(y,x_0)\vee \sqrt{s})^{\kappa+2}\mu(B(x_0,1+d(y,x_0)\vee \sqrt{s}))}\d\mu(y,s)\right)
\le C(x_0,\bz)R^{-3/2},
$$
which converges to zero as $R\to\infty$.
Therefore $u= 0$ on $B(x_0,\rho(x_0))\times(0,\fz)$. As $x_0$ is arbitrary, we see that $u=0$ in $X\times\rr_+$.
This completes the proof.
\end{proof}

\begin{remark}
For different cases $X\times \rr$ and $X\times \rr_+$, we adopt the Green function and the Poisson kernel as tools.
In fact, with the help of another approach, we can derive the Liouville theorem of Schr\"{o}dinger equation in $X\times \rr_+$
from the conclusion on the product space $X\times \rr$.
More precisely, since the trace of $u$ is equal to zero, we define a $(-\partial_t^2+{\mathscr{L}})$-harmonic function on $X\times \rr$ as
$$
w(x,t)=
\begin{cases}
\displaystyle u(x,t),\quad &t>0; \\
\displaystyle 0,&t=0; \\
\displaystyle u(x,-t), &t<0.
\end{cases}
$$
Note that $w$ is a even function with respect to the time variable.
Then $w$ admits the weighted $L^2$-integrability on $X\times \rr$ becaues $u$ admits a similar property on $X\times \rr_+$.
By Theorem \ref{Liouville}, we arrive at $w=0$ in $X\times \rr$, and hence $u=0$ in $X\times \rr_+$.
Though this argument is more or less simple, we still provide the Poisson kernel approach to prove the Liouville theorem of Schr\"{o}dinger equation in $X\times \rr_+$
since we aim to emphasize the difference between the Green function and the Poisson kernel.
\end{remark}

\begin{remark}
It is remarkable that Auscher--Ben Ali \cite{AB2007} achieve the $L^p$ estimate:
for $n\ge1$, if $0\le V\in RH_q(\rn)$ for some $q>1$, then
$$
\|\Delta f\|_{L^p(\rn)}+\|V f\|_{L^p(\rn)}\le C\|(-\Delta +V)f\|_{L^p(\rn)}$$
holds for $1<p<q+\varepsilon$. %all smooth functions $f$ with compact support.
This result extends the one of Shen obtained in \cite{Sh1995} under the restriction $q>n/2$.
It is somewhat surprising that they can go below the order $n/2$
which is critical for the regularity theory of the elliptic operator $-\Delta+V$.
One might wonder if there is any possibility of breaking the requirement $q> \max\{Q/2,1\}$ for the H\"older continuity and Liouville theorem (i.e., Theorem \ref{thm:Liouville}) to $q>1$. This is an interesting problem, while to some extent is another story since it's unclear how to define the auxiliary function $\rho$ determined by $V$ whenever $V\in RH_q$ for some $q<n/2$. Under the framework with $\rho$ well-defined, the reverse H\"older index given in our results  in the previous and this sections is critical.
\end{remark}

\section{Application I: harmonic function with BMO trace}\label{s4}
\hskip\parindent
To state the main result of this section, we first introduce some function classes.

\begin{definition}
An $\mathbb{L}$-harmonic function $u$ defined on $X\times \rr_+$
is said to be in $\mathrm{HMO}_\LV=\mathrm{HMO}_\LV(X\times \rr_+)$,
the {space of functions of harmonic mean oscillation},
if $u$ satisfies
$$
\|u\|_{\HMO_\LV}=\sup_{B} \lf(\int_0^{r_B}\fint_B  \left|t\nabla_{x,t} u\right|^2\d \mu\frac{\d t}{t}\r)^{1/2}<\fz.
$$
\end{definition}

\begin{definition}
A locally integrable function $f$ defined on $X$ is said to be in $\mathrm{BMO}_\LV=\mathrm{BMO}_\LV(X)$,
the {space of functions of bounded mean oscillation}, if $f$ satisfies
$$
\|f\|_{\BMO_\LV}=\sup_{B:r_B< \rho(x_B)}\lf(\fint_B|f-f_B|^2\d \mu\r)^{1/2}+\sup_{B:r_B\ge \rho(x_B)}\lf(\fint_B|f|^2\d \mu\r)^{1/2}<\fz.
$$
\end{definition}

\begin{theorem}\label{thm1}
Let $(X,d,\mu,\mathscr{E})$ be a complete Dirichlet metric measure space satisfying a doubling property $(D)$ and admitting an $L^2$-Poincar\'{e} inequality $(P_2)$.
Suppose that $0\le V\in A_\infty(X)\cap RH_{q}(X)$ for some $q>\max\{Q/2,1\}$.
Then $u\in \mathrm{HMO}_{\LV}(X\times \rr_+)$ if and only if $u=e^{-t\sqrt{\LV}}f$ for some $f\in \BMO_\LV(X)$.
Moreover, it holds that % there exists a constant $C>0$ such that
$$ \|u\|_{\HMO_{\LV}} \approx\|f\|_{\BMO_\LV}.$$
%
%
%
%Then the following statements hold true.
%\begin{enumerate}
%   \item[\rm{(i)}] If $u\in \mathrm{HMO}_{\LV}(X\times \rr_+)$, then $u=\PV_tf$ for some $f\in \BMO_\LV(X)$.
%Moreover there exists a constant $C>1$ independent of $u$ such that
%$$\|f\|_{\BMO_\LV}\le C\|u\|_{\HMO_{\LV}}.$$
% \item[\rm{(ii)}]If $f\in \BMO_\LV(X)$, then $u=\PV_tf\in \HMO_{\LV}(X\times \rr_+)$,
%and there exists  a constant $C>1$ independent of $f$ such that
%$$\|u\|_{\HMO_{\LV}}\le C\|f\|_{\BMO_\LV}.$$
%\end{enumerate}
\end{theorem}

Two remarks are in order.

\begin{remark}\label{rem:dim-relax1}
Compared with \cite[Theorem 1.1]{JL2022}, Theorem \ref{thm1} improves the critical reverse H\"{o}lder order $q$ from $(Q+1)/2$ to $\max\{Q/2,1\}$
at the expense of a slight reverse doubling property on the underlying space (namely $n>1$).
The assumption $n>1$ appears naturally when we need to make use of the fundamental solution of the elliptic operator as tool,
and in the case of Euclidean space $\rn$,
the assumption $n>1$ is superfluous.
Therefore our Theorem \ref{thm1} extends the result of Duong-Yan-Zhang \cite{DYZ2014}
by means of improving the critical reverse H\"{o}lder order from $n$ to $n/2$,
relaxing the dimension of Euclidean space $\rn$ from $n\ge3$ to $n\ge2$,
and removing the $C^1$-regularity of the $\mathbb{L}$-harmonic function.
\end{remark}

\begin{remark}\label{rem:dim-relax2}
When $(X,d,\mu)=(\rn,|\cdot|,\d x)$ with $n\ge3$,
Shen \cite{Sh1995} studied the $L^p$-estimate for the Schr\"{o}dinger operator $-\Delta+V$,
where the non-negative potential $V$ is in the reverse H\"older class $RH_q(\rn)$ for some $q> n/2$.
However, the observant reader might notice that
why does our conclusion (or \cite[Corollary 1.2]{JL2022} for another result) only requires the dimension of the underlying space $n\ge 2$,
rather than $n\ge3$ as in \cite{Sh1995}?
Below we explain the reason.
On the one hand, in \cite{Sh1995}, Shen invoke repeatedly the Kato condition
$$
\int_{B(x,R)}\frac{V(y)}{|x-y|^{n-2}}\d y\le \frac{C}{R^{n-2}}\int_{B(x,R)}V(y)\d y.
$$
In fact, the term $|x-y|^{2-n}$ in the integrand above comes from the fundamental solution of $-\Delta$ provided $n\ge3$.
In the case of $n=2$, it should be the logarithmic function $\log|x-y|$.
On the other hand, in the article \cite{JL2022}, the authors employ the theory of the parabolic operator (namely heat kernel), rather than the elliptic operator,
to avoid meeting $(-\Delta)^{-1}$.
In this article, with the help of the Green function in the product space,
we need the dimension of the product space $\rr^{n+1}$ to be equal or greater than 3, and hence $n\ge2$ automatically.
Moreover, when $n=2>1$, since the non-negative potential $V$ on $\rr^2$ is in the reverse H\"{o}lder class $RH_q(\rr^2)$ for some $q>n/2=1$,
the measure $V\d y$ induced by $V$ satisfies the doubling property
$$\int_{B(x,2r)}V\d y\le C\int_{B(x,r)}V\d y$$
and the reverse doubling property
$$\int_{B(x,r)}V\d y\le C\lf(\frac{r}{R}\r)^{2-2/q}\int_{B(x,R)}V\d y$$
for every $x\in \rr^2$ and all $0<r<R<\fz$.
From this, we see that the critical function $\rho(x_1,x_2)$ associated to $V$ on $\rr^2$ admits the similar properties as $\rho(x_1,x_2,x_3)$ on $\rr^3$;
see \cite{Ku2000} for another case $n=2$.
Therefore the dimension of the underlying space in this article (or \cite{JL2022}) can be relaxed to be greater than 1.
\end{remark}

%\subsection{From $(Q+1)/2$ to $Q/2$}\label{sec3.4}
\begin{proof}[Proof of Theorem \ref{thm1}]
{\bf Step 1:} from BMO to HMO.
In \cite{JL2022}, under the requirement $q>\max\{Q/2,1\}$, Jiang--Li established some estimates on the Poisson semigroup such as $(PUB)$ and
$$
\begin{cases}
\dis \left|t\partial_t e^{-t\sqrt{\LV}}1(x)\right| \le C\lf(\frac{t}{\rho(x)}\r)^{\dz/2}\lf(1+\frac{t}{\rho(x)}\r)^{-N},\\[8pt]
\dis \lf(\int_0^{r_B}\fint_{B} \left|t\nabla_x e^{-t\sqrt{\LV}}1 \right|^2\d \mu\frac{\d t}{t}\r)^{1/2}\le C\min\lf\{\lf(\frac{r_B}{\rho(x_B)}\r)^{\dz/2},1\r\},
\end{cases}
$$
for some $0<\dz<\min\{1,2-Q/q\}$ and any $N>0$; see Lemma \ref{lem-1}.
In this article,
%we assume that $n\ge 2$ and hence $\max\{Q/2,1\}=Q/2$ by $Q\ge n$.
%Therefore, in our assumption $q>Q/2$ with $Q\ge n\ge 2$,
we can follow the argument in the proof of \cite[Theorem 5.3]{JL2022} verbatim to derive
$$
\|u\|_{\HMO_\LV}=\| e^{-t\sqrt{\LV}} f\|_{\HMO_\LV}=\sup_{B} \lf(\int_0^{r_B}\fint_B  \left|t\nabla_{x,t}  e^{-t\sqrt{\LV}}f\right|^2\d \mu\frac{\d t}{t}\r)^{1/2}\le C\|f\|_{\BMO_\LV}.
$$

\smallskip

{\bf Step 2:} from HMO to BMO.
In this step, we show that, under the assumption $q>\max\{Q/2,1\}$,
every ${\HMO_\LV}$-function on $X\times \rr$ can be represented as the Poisson integral with the trace in $\BMO_\LV(X)$.
Since it is intended solely as a brief review and not as a rigorous development, the pertinent results are stated without proof.

First, we establish the time regularity
$$|t\partial_t u(x,t)|\le C\|u\|_{\BMO_\LV}$$
and the weighted $L^2$-integrability
$$\int_X\frac{|u(y,s)|^2}{(t+d(x,y))\mu(B(x,t+d(x,y)))}\d\mu(y)
\le C\left[(1+t^{-1})\|u(\cdot,s)\|^2_{L^\fz(B(x,2))}+(1+\log^2 s )\|u\|^2_{\HMO_\LV} \right],$$
which indicates the Poisson integral $e^{-t\sqrt{\LV}}(u(\cdot,s))$ is well-defined.
Next, with the help of the H\"{o}lder continuity (Theorem \ref{holder1}) and the Liouville property (Theorem \ref{Liouville} or Theorem  \ref{Liouville1}),
we prove the identity
$$u(x,t+s)=e^{-t\sqrt{\LV}}(u(\cdot,s))(x).$$
This identity and the theory of tent space yield the uniform norm estimate
$$\sup_{s>0}\|u(\cdot,s)\|_{\BMO_\LV}\le C\|u\|_{\HMO_\LV}.$$
Finally, the Banach-Alaoglu theorem tells us that there exists a $\BMO_\LV$-function $f$ such that
$$\|f\|_{\BMO_\LV}\le C\|u\|_{\HMO_\LV},$$
and the Hardy-BMO duality theorem of Fefferman-Stein further implies the identity
$$u(x,t)=e^{-t\sqrt{\LV}}f(x)+h(x)$$
for some function $h(x)$ on $X$.
This function $h$ can be verified as vanish by the Liouville property from \cite[Proposition 4.4]{JL2022}.

Based on the above argument, we derive the desired result.
This completes the proof.
\end{proof}

\section{Application II: harmonic function with CMO trace}\label{s5}
%\hskip\parindent
\subsection{Statement of the main result}
\hskip\parindent
To state the main result of this section, we set a fixed reference point $x_0$ in $X$, and introduce some function classes.
%
%An $\LV_+$-harmonic function $u$ defined on $X\times \rr_+$
%is said to be in $\mathrm{HMO}_\LV=\mathrm{HMO}_\LV(X\times \rr_+)$,
%the {space of functions of harmonic mean oscillation},
%if $u$ satisfies
%$$
%\|u\|_{\HMO_\LV}=\sup_{B} \lf(\int_0^{r_B}\fint_B  |t\nabla u|^2\d \mu\frac{\d t}{t}\r)^{1/2}.
%$$

\begin{definition}
An $\mathbb{L}$-harmonic function $u$ defined on $X\times \rr_+$
is said to be in $\mathrm{HCMO}_\LV=\mathrm{HCMO}_\LV(X\times \rr_+)$,
the {space of functions of harmonic vanishing mean oscillation},
if $u$ is in $\HMO_\LV(X\times \rr_+)$ and satisfies three limiting conditions
$$\bz_1(u)=\bz_2(u)=\bz_3(u)=0,$$
where
$$
\begin{cases}
\dis \bz_1(u)=\lim_{a\to0}\sup_{B:r_B\le a} \lf(\int_0^{r_B}\fint_B  |t\nabla_{x,t} u|^2\d \mu\frac{\d t}{t}\r)^{1/2}; \\[8pt]
\dis \bz_2(u)=\lim_{a\to\fz}\sup_{B:r_B\ge a} \lf(\int_0^{r_B}\fint_B  |t\nabla_{x,t} u|^2\d \mu\frac{\d t}{t}\r)^{1/2}; \\[8pt]
\dis \bz_3(u)=\lim_{a\to\fz}\sup_{B:B\subset B(x_0,a)^\complement} \lf(\int_0^{r_B}\fint_B  |t\nabla_{x,t} u|^2\d \mu\frac{\d t}{t}\r)^{1/2}.
\end{cases}
$$
We endow $\mathrm{HCMO}_\LV(X\times \rr_+)$ with the norm of $\mathrm{HMO}_\LV(X\times \rr_+)$.
%where $\nabla=(\nabla_x,\partial_t)$ and $|\nabla u|^2=|\nabla_x u|^2+|\partial_tu|^2$.
\end{definition}

%A locally integrable function $f$ defined on $X$ is said to be in $\mathrm{BMO}_\LV=\mathrm{BMO}_\LV(X)$,
%the {space of functions of bounded mean oscillation}, if $f$ satisfies
%$$
%\|f\|_{\BMO_\LV}=\sup_{B}\lf(\fint_B|f-f_B|^2\d \mu\r)^{1/2}+\sup_{{B:r_B\ge \rho(x_B)}}\lf(\fint_B|f|^2\d \mu\r)^{1/2}.
%$$

\begin{definition}
A locally integrable function $f$ defined on $X$ is said to be in $\mathrm{CMO}_\LV=\mathrm{CMO}_\LV(X)$,
the {space of functions of vanishing mean oscillation}, if $f$ is in $\BMO_\LV(X)$ and satisfies four limiting conditions
$$\gz_1(f)=\gz_2(f)=\gz_3(f)=\gz_4(f)=0,$$
where
$$
\begin{cases}
&\dis  \gz_1(f)=\lim_{a\to0}\sup_{\substack{B:r_B< \rho(x_B)\\ r_B< a}} \lf(\fint_B|f-f_B|^2\d \mu\r)^{1/2}; \\
%&\dis \gz_2(f)=\lim_{a\to\fz}\sup_{B:r_B\ge a}\lf(\fint_B|f-f_B|^2\d \mu\r)^{1/2}; \\
&\dis \gz_2(f)=\lim_{a\to\fz}\sup_{\substack{B:r_B< \rho(x_B)\\ B\subset B(x_0,a)^\complement}}\lf(\fint_B|f-f_B|^2\d \mu\r)^{1/2};\\
&\dis \gz_3(f)=\lim_{a\to\fz}\sup_{\substack{B:r_B\ge \rho(x_B)\\r_B\ge a}}\lf(\fint_B|f|^2\d \mu\r)^{1/2}; \\
&\dis \gz_4(f)=\lim_{a\to\fz}\sup_{\substack{B: r_B\ge \rho(x_B)\\B\subset B(x_0,a)^\complement}}\lf(\fint_B|f|^2\d \mu\r)^{1/2}.
%&\gz_6(f)={\color{red}\lim_{a\to0}}\sup_{\substack{B:r_B\ge \rho(x_B)\\ {\color{red}r_B\le a}}}\lf(\fint_B|f|^2\d \mu\r)^{1/2}.
\end{cases}
$$
We endow $\mathrm{CMO}_\LV(X)$ with the norm of $\mathrm{BMO}_\LV(X)$.
\end{definition}
We remark that, Song--Wu \cite{SW2022} equip three independent limiting conditions
%{\color{red}(RMK. But these five conditions can be reduced into three independent conditions as shown in \cite{SW2022}. In the setting described in this article, can we also reduce the above four conditions?)}
for the $\mathrm{BMO}_\LV$-function with $X=\mathbb R^n$,
and our $\mathrm{CMO}_\LV$-function enjoys four limiting conditions.
However, this does not affect the definition of a $\mathrm{CMO}_\LV$-function.
For more details about the $\mathrm{CMO}_\LV$-function, we refer the interested reader to Remarks 7.1--7.3 of \cite{LI}.

The following theorem is the main result of this section.

\begin{theorem}\label{thm2}
Let $(X,d,\mu,\mathscr{E})$ be a complete Dirichlet metric measure space satisfying a doubling property $(D)$ and admitting an $L^2$-Poincar\'{e} inequality $(P_2)$.
Suppose that $0\le V\in A_\infty(X)\cap RH_{q}(X)$ for some $q> \max\{Q/2,1\}$.
Then $u\in \mathrm{HCMO}_{\LV}(X\times \rr_+)$ if and only if $u=e^{-t\sqrt{\LV}}f$ for some $f\in \CMO_\LV(X)$.
Moreover, it holds that % there exists a constant $C>0$ such that
$$\|u\|_{\mathrm{HMO}_{\LV}} \approx \|f\|_{\BMO_\LV}.$$
%
%
%
%
%
%Then the following statements hold true.
%\begin{enumerate}
%   \item[\rm{(i)}] If $u\in \mathrm{HCMO}_{\LV}(X\times \rr_+)$, then $u=\PV_tf$ for some $f\in \CMO_\LV(X)$.
%Moreover there exists a constant $C>1$ independent of $u$ such that
%$$\|f\|_{\BMO_\LV}\le C\|u\|_{\HMO_{\LV}}.$$
% \item[\rm{(ii)}]If $f\in \CMO_\LV(X)$, then $u=\PV_tf\in \HCMO_{\LV}(X\times \rr_+)$,
%and there exists  a constant $C>1$ independent of $f$ such that
%$$\|u\|_{\HMO_{\LV}}\le C\|f\|_{\BMO_\LV}.$$
%\end{enumerate}
 \end{theorem}

\begin{remark}
In Theorems \ref{thm1} and \ref{thm2}, we always assume that the non-negative potential $V$ is non-trivial.
However, it is worth mentioning that Theorem \ref{thm3} holds for $V=0$
since the H\"{o}lder continuity and the Liouville property are valid automatically.
In fact, the $V=0$ version of Theorems \ref{thm1} and \ref{thm2}) have been studied in \cite{JLS},
whose proof is far away from the argument \cite{DYZ2014,JL2022} due to different natures of Schr\"{o}dinger operator $\LV$ and Laplace operator $\L$.
See \cite{LLSY2024} for more results about Morrey-Dirichlet problem for the elliptic equation.
\end{remark}

To the best of our knowledge, Theorem \ref{thm2} is new even for $V=1$ on $X$.
In the case of the Euclidean space $\rn$ with $n\ge2$,
Theorem \ref{thm2} is also new for the uniformly elliptic operator $\L=-\mathrm{div} A\nabla$,
where $A=A(x)$ is an $n\times n$ matrix of real, symmetric, bounded measurable coefficients, defined on $\rn$, and satisfies the ellipticity condition, i,e.,
there exist constants $0<\lz\le\Lambda<\fz$ such that for all $\xi\in \rn$,
$$\lz |\xi|^2\le \langle A \xi, \xi\rangle\le \Lambda|\xi|^2.$$

%
%\begin{corollary}\label{cor1}
%Let $\L=-\mathrm{div} A\nabla$ be a uniformly elliptic operator on $\rn$.
%Then an $\LV_+$-harmonic function $u\in \HMO(\rr^{n+1}_+)$ if and only if there exists a function $f\in \BMO(\rn)$ such that $u(x,t)=P_tf(x)$.
%Moveover there exists a constant $C>1$ such that
%$$C^{-1}\|f\|_{\mathrm{BMO}(\rn)} \leq\|u\|_{\mathrm{HMO}(\rr^{n+1}_+)} \leq C\|f\|_{\mathrm{BMO}(\rn)}.$$
%\end{corollary}
%

\begin{corollary}
%Let $\LV=-\mathrm{div} A\nabla+V$ be a Schr\"{o}dinger operator on $\rn$ with $n\ge2$,
Let $\L=-\mathrm{div} A\nabla$ be an uniformly elliptic operator on $\rn$ with $n\ge2$.
Suppose that $0\le V\in RH_{q}(\rn)$ for some $q> n/2$.
Then $u\in \HCMO_\LV(\rr^{n+1}_+)$ if and only if $u=e^{-t\sqrt{\LV}}f$ for some $f\in \CMO_\LV(\rr^n)$.
Moreover, it holds that
$$
 \|u\|_{\mathrm{HMO}_\LV} \approx\|f\|_{\BMO_\LV}.
$$
\end{corollary}
Note that on $\rn$, the reverse H\"{o}lder property of $V\in RH_q(\rn)$ implies that $V\in A_\fz(\rn)$.
Therefore this corollary extends (in a sense) the result of Song-Wu \cite{SW2022}
from the Laplace operator $-\Delta$ to the uniformly elliptic operator $\L=-\mathrm{div} A\nabla$,
improve the reverse H\"{o}lder order $q$ from $(n+1)/2$ to $n/2$,
and relax the dimension of Euclidean space $\rn$ from $n\ge3$ to $n\ge2$.

Our  result  also applies to the elliptic operator $L=-\oz^{-1}\mathrm{div} A\nabla$,
where $\oz$ is a Muckenhoupt $A_2$-weight and $A=A(x)$ is an $n\times n$ matrix of real symmetric, bounded measurable coefficients satisfying the degenerate ellipticity condition,
i.e.,  there exist constants $0<\lz\le\Lambda<\fz$ such that
$$\lz \oz(x) |\xi|^2\le \langle A(x)\xi, \xi\rangle\le \Lambda \oz(x)|\xi|^2. $$ %\quad \forall\, \xi\in \rn.$$
The doubling property is a consequence of the $A_2$-weight,
and it was known from \cite{FKS82} that an $L^2$-Poincar\'e inequality holds.
Moreover, as we do not need any curvature condition,  our result can be applied to
different settings, which include Riemannian metric measure
spaces, sub-Riemannian manifolds;
see \cite[Introduction]{ACDH2004}, \cite[Section 6]{BDL2018}, \cite[Section 7]{CJKS2020} and \cite[Section 5]{YZ2011} for more examples.

\subsection{From solution to trace}
\hskip\parindent
In this subsection we seek the CMO trace of an HCMO function with desired norm controlled.
\begin{theorem}\label{hcmo}
Let $(X,d,\mu,\mathscr{E})$ be a complete Dirichlet metric measure space satisfying a doubling property $(D)$ and admitting an $L^2$-Poincar\'{e} inequality $(P_2)$.
Suppose that $0\le V\in A_\infty(X)\cap RH_{q}(X)$ for some $q> \max\{Q/2,1\}$.
If $u\in \mathrm{HCMO}_{\LV}(X\times \rr_+)$, then $u=e^{-t\sqrt{\LV}}f$ for some $f\in \CMO_\LV(X)$.
Moreover, it holds that % there exists a constant $C>0$ independent of $u$ such that
$$\|f\|_{\BMO_\LV}\le C\|u\|_{\mathrm{HMO}_{\LV}}.$$
 \end{theorem}

\begin{proof}
{\bf Step 1:} reduction.
If $u\in\mathrm{HCMO}_\LV(X\times\rr_+)$, then $u\in\mathrm{HMO}_\LV(X\times\rr_+)$.
By Theorem \ref{thm1}, there exists a function $f\in \mathrm{BMO}_\LV(X)$ such that $u=e^{-t\sqrt{\LV}}f$ with
$$\|f\|_{\mathrm{BMO}_\LV} \le C \|u\|_{\mathrm{HMO}_\LV} .$$
It remains to verity $f$ satisfies four limiting conditions
$$\gz_1(f)=\gz_2(f)=\gz_3(f)=\gz_4(f)=0.$$
%Due to the definition of the $\BMO_\LV$-function,
%we will consider two cases: the ball with arbitrary radius \& the ball with large radius.

{\bf Step 2:} control the mean oscillation of $f$.
For any ball $B=B(x_B,r_B)\subset X$  with $r_B<\rho(x_B)$, it holds that
\begin{align*}
\lf(\fint_B|f-f_B|^2\d\mu\r)^{1/2}
& \le \lf(\fint_{B}\fint_{B}|f(x)-f(y)|^2\d\mu(x)\d\mu(y)\r)^{1/2} \\
%&\ = \lf(\fint_{B}\fint_{B}|f(x)-u(x,r_B)+u(x,r_B)-u(y,r_B)+u(y,r_B)-f(y)|\d\mu(x)\d\mu(y)\r)^{1/2} \\
& \le 2\lf(\fint_{B}\left|f-e^{-r_B\sqrt{\LV}}f\right|^2\d\mu\r)^{1/2}+\lf(\fint_{B}\fint_{B}|u(x,r_B)-u(y,r_B)|^2\d\mu(x)\d\mu(y)\r)^{1/2}.
\end{align*}

Let us estimate the second term first.
One may employ $(P_2)$ to obtain
\begin{align*}
\lf(\fint_{B}\fint_{B}|u(x,r_B)-u(y,r_B)|^2\d\mu(x)\d\mu(y)\r)^{1/2}
%&=\lf(\fint_{r_B/2}^{r_B}\fint_{B}\fint_{B}|u(x,r_B)-u(y,r_B)|^2\d\mu(x)\d\mu(y)\d t\r)^{1/2} \\
&\le \lf(\fint_{r_B/2}^{r_B}\fint_{B}\fint_{B}|u(x,t)-u(y,t)|^2\d\mu(x)\d\mu(y)\d t\r)^{1/2} \\
&\ + 2\lf(\fint_{r_B/2}^{r_B}\fint_{B}|u(x,r_B)-u(x,t)|^2\d\mu(x)\d t\r)^{1/2}
%&\le C\lf(\fint_{r_B/2}^{r_B}\fint_{B}|r_B\nabla_xu(x,t)|^2\d\mu(x)\d t\r)^{1/2} \\
%&\ + 2\lf(\fint_{r_B/2}^{r_B}\fint_{B}\lf|\int_t^{r_B}\partial_su(x,s)\d s\r|^2\d\mu(x)\d t\r)^{1/2} \\
%&\le C\lf(\int_0^{r_B}\fint_{B}|t\nabla_x u|^2\d\mu\frac{\d t}{t}\r)^{1/2} \\
%&\ + 2\lf[\fint_{B}\lf(\int_{r_B/2}^{r_B}|s\partial_su(x,s)|\frac{\d s}{s}\r)^2\d\mu(x)\r]^{1/2} \\
\le C\lf(\int_0^{r_B}\fint_{B} \left|t \nabla_{x,t} u\right|^2\d\mu\frac{\d t}{t}\r)^{1/2}.
\end{align*}
%For the difference term $|u(x,r_B)-u(x,t)|$, by employing  over and over again, we deduce that
%\begin{align*}
%|u(x,r_B)-u(x,t)|
%&=\lf|\int_t^{r_B}\partial_su(x,s)\d s\r|\\
%&\le \int_t^{r_B}\lf|\lim_{h\to0^+}\frac{u(x,s+h)-u(x,s)}{h}\r|\d s\\
%&\le \int_{r_B/2}^{r_B}\lim_{h\to0^+}\lf|\frac{u(x,s+h)-u(x,s)}{h}\r|\d s\\
%&\le C\int_{r_B/2}^{r_B}\lim_{h\to0^+} \lf(\frac{1}{s^2}\int_0^{2s}\fint_{B(x,2s)}|r\partial_ru(y,r)|^2\d\mu(y)\frac{\d r}{r} \r)^{1/2}    \d s\\
%&\le C\lf(\int_0^{2r_B}\fint_{B(x,2r_B)} |r\partial_ru(y,r)|^2\d\mu(y)\frac{\d r}{r}\r)^{1/2} ,
%%&\le C\lf(\int_0^{4r_B}\fint_{4B} |r\partial_ru(y,r)|^2\d\mu(y)\frac{\d r}{r}\r)^{1/2},
%\end{align*}
%which yields
%\begin{align*}
%\lf(\fint_{r_B/2}^{r_B}\fint_{B}|u(x,r_B)-u(x,t)|^2\d\mu(x)\d t\r)^{1/2}
%\le C\lf(\int_0^{4r_B}\fint_{4B} |t\partial_tu|^2\d\mu\frac{\d t}{t}\r)^{1/2}.
%\end{align*}
%Therefore we know the second term can be estimate as
%\begin{align*}
%\lf(\fint_{B}\fint_{B}|u(x,r_B)-u(y,r_B)|^2\d\mu(x)\d\mu(y)\r)^{1/2}\le C\lf(\int_0^{4r_B}\fint_{4B}|t\nabla u|^2\d\mu\frac{\d t}{t}\r)^{1/2}.
%\end{align*}

For the first term, the $L^2$-duality argument and the Calder\'{o}n reproducing formula (see \cite{DY2005-1,DY2005-2,DYZ2014,JLS} for example) tell us that
\begin{align*}
\lf(\fint_{B} \left|f-e^{-r_B\sqrt{\LV}}f\right|^2\d\mu\r)^{1/2}
%&=[\mu(B)]^{-1/2}\sup_{\|g\|_{L^2(B)}\le 1}\lf|\int_Xg(I-\PV_{r_B})f \d\mu\r| \\
&=[\mu(B)]^{-1/2}\sup_{\|g\|_{L^2(B)}\le 1}\lf|\int_X f(I-e^{-r_B\sqrt{\LV}})g \d\mu\r| \\
&=[\mu(B)]^{-1/2}\sup_{\|g\|_{L^2(B)}\le 1}\lf|\int_0^\fz\int_Xt\partial_t e^{-t\sqrt{\LV}}f t\partial_t e^{-t\sqrt{\LV}}(I-e^{-r_B\sqrt{\LV}})g\d \mu\frac{\d t}{t}\r|.
\end{align*}
Moreover, the Step 1 in \cite[page 215]{JLS} tells us that the integral in the RHS of the identity above can be controlled by
$$
\lf|\int_0^\fz\int_Xt\partial_t e^{-t\sqrt{\LV}}f t\partial_t e^{-t\sqrt{\LV}}(I-e^{-r_B\sqrt{\LV}})g\d \mu\frac{\d t}{t}\r|
\le  C[\mu(B)]^{1/2}\sum_{k=1}^\fz k2^{-k}\lf(\int_0^{2^{k}r_B}\fint_{{2^{k}B}}|t\partial_tu|^2\d \mu\frac{\d t}{t}\r)^{1/2}.
%\le {C}\mu(B)^{1/2}\|g\|_{L^2(B)}\|P_tu_\varepsilon\|_{\TMO(X\times \rr_+)}
%&\le {C}[\mu(B)]^{1/2+\az}\|u\|_{\TMO^\az_\L}\|g\|_{L^2(B)}.
$$
Therefore we derive
$$
\lf(\fint_{B}\left|f-e^{-r_B\sqrt{\LV}}f\right|^2\d\mu\r)^{1/2}
%&=[\mu(B)]^{-1/2}\sup_{\|g\|_{L^2(B)}\le 1}\lf|\int_0^\fz\int_Xt\partial_t\PV_tf t\partial_t\PV_t(I-\PV_{r_B})g\d \mu\frac{\d t}{t}\r|\\
\le C\sum_{k=1}^\fz k2^{-k}\lf(\int_0^{2^{k}r_B}\fint_{{2^{k}B}}|t\partial_tu|^2\d \mu\frac{\d t}{t}\r)^{1/2}.
$$
%which, together the estimate on the second term, implies

Combining two estimates on the first term and the second term leads to
$$
\lf(\fint_{B}|f-f_{B}|^2\d\mu\r)^{1/2}
 \le C\sum_{k=0}^\fz k2^{-k}\lf(\int_0^{2^{k}r_B}\fint_{{2^{k}B}}|t\nabla_{x,t} u|^2\d \mu\frac{\d t}{t}\r)^{1/2}.
 % = C\sum_{k=0}^\fz{ 2^{-k/2}}\eta_k(u,B).
$$

\smallskip

{\bf Step 3:} verify $\gz_1(f)=\gz_2(f)=0$.
The proof of $\gz_1(f)=\gz_2(f)=0$ is similar to that of Step 1 in \cite[Theorem 4.1]{JLS}.
We leave the detail to the interested reader.

\smallskip

{\bf Step 4:} control the mean of $f$.
For any ball $B=B(x_B,r_B)\subset X$ with $r_B\ge \rho(x_B)$,
it holds by the $L^2$-duality argument and the Calder\'{o}n reproducing formula (see \cite{SL2022} for example) that
$$
\lf(\fint_{B}|f|^2\d\mu\r)^{1/2}
=[\mu(B)]^{-1/2}\sup_{\|g\|_{L^2(B)}\le 1}\lf|\int_X fg \d\mu\r|
=[\mu(B)]^{-1/2}\sup_{\|g\|_{L^2(B)}\le 1}\lf|\int_0^\fz\int_Xt\partial_te^{-t\sqrt{\LV}}f t\partial_t e^{-t\sqrt{\LV}}g\d \mu\frac{\d t}{t}\r|.
$$
One may invoke $(PUB)$ to obtain
$$
\lf|\int_0^\fz\int_Xt\partial_te^{-t\sqrt{\LV}}f t\partial_t e^{-t\sqrt{\LV}}g\d \mu\frac{\d t}{t}\r|
\le  C[\mu(B)]^{1/2}\sum_{k=1}^\fz 2^{-k}\lf(\int_0^{2^{k}r_B}\fint_{{2^{k}B}}|t\partial_tu|^2\d \mu\frac{\d t}{t}\r)^{1/2};
%\le {C}\mu(B)^{1/2}\|g\|_{L^2(B)}\|P_tu_\varepsilon\|_{\TMO(X\times \rr_+)}
%&\le {C}[\mu(B)]^{1/2+\az}\|u\|_{\TMO^\az_\L}\|g\|_{L^2(B)}.
$$
see the proof of \cite[Lemma 5]{SL2022} for example.
%Indeed, one writes
%\begin{align*}
%\lf|\int_0^\fz\int_Xt\partial_t\PV_tf t\partial_t\PV_tg\d \mu\frac{\d t}{t}\r|
%&=\lf\{\int_{T(2B)}+\sum_{k=2}^\fz\int_{T({2^{k}B})\setminus T({2^{k-1}B})}\r\}\cdots\d \mu\frac{\d t}{t} \\
%&=J_1+\sum_{k=2}^\fz J_k.
%\end{align*}
%Similar to $I_1$ above, we estimate the term $J_1$ as
%$$
%J_1\le C[\mu(B)]^{1/2}\lf(\int_0^{2r_B}\fint_{2B}|t\partial_tu|^2\d \mu\frac{\d t}{t}\r)^{1/2}.
%$$
%For the term $J_k$, by the upper bound on $p^v_t$, and the property of $\rho$ (see \eqref{critical}), we see the the integrand can be controlled by
%\begin{align*}
%|t\partial_t\PV_tg(x)|
%&=\lf|\int_Xt\partial_tp^v_t(x,y)g(y)\d\mu(y)\r| \\
%&\le C\int_B\frac{t}{t+d(x,y)}\frac{1}{\mu(B(x,t+d(x,y)))}\lf(1+\frac{{t}+d(x,y)}{\rho(y)}\r)^{-1}|g(y)|\d \mu(y)\\
%&\le C\int_B\frac{t\rho(y)}{(t+d(x,y))^2}\frac{|g(y)|}{\mu(B(x,t+d(x,y)))}\d \mu(y) \\
%&\le C\frac{t}{(2^kr_B)^2}\frac{1}{\mu(2^kB)}\int_B\rho(y)|g(y)|\d \mu(y) \\
%&\le C\frac{t}{(2^kr_B)^2}\frac{1}{\mu(2^kB)}\int_B\rho(x_B)\lf(1+\frac{d(y,x_B)}{\rho(x_B)}\r)^{k_0/(k_0+1)}|g(y)|\d \mu(y)\\
%&\le C2^{-k}\lf(\frac{t}{2^kr_B}\r)\frac{\|g\|_{L^1(B)}}{\mu(2^kB)}.
%\end{align*}
%where the second inequality from the bottom is due to the property of $\rho$.
%Hence the rest proof of this claim is done verbatim as in Case 1.
%Based on the above arguments, we derive that for any ball $B$ with $r_B\ge\rho(x_B)$,
Therefore  we have
$$
\lf(\fint_{B}|f|^2\d\mu\r)^{1/2}
 \le C\sum_{k=2}^\fz 2^{-k}\lf(\int_0^{2^{k}r_B}\fint_{{2^{k}B}}|t\partial_t u|^2\d \mu\frac{\d t}{t}\r)^{1/2}.
 % = C\sum_{k=2}^\fz{ 2^{-k}}\eta_k(u,B).
$$

\smallskip

{\bf Step 5:} verify $\gz_3(f)=\gz_4(f)=0$.
The proof of $\gz_3(f)=\gz_4(f)=0$ is similar to that of Step 1 in \cite[Theorem 4.1]{JLS}.
We leave the detail to the interested reader.

\smallskip

{\bf Step 6:} completion of the proof.
From Steps 3 and 5, we know that
$$\gz_1(f)=\gz_2(f)=\gz_3(f)=\gz_4(f)=0,$$
which completes the proof.
%
%It remains to show $f$ satisfies the two limiting conditions
%$$\gz_4(f)=\gz_5(f)=0.$$
%Note that for any fixed integer $k$, it holds from $\bz_2(u)=\bz_3(u)=0$
%\begin{align*}
%\sup_{\substack{B:r_B\ge a\\r_B\ge \rho(x_B)}}\eta_k(u,B)
%&\le\sup_{B: r_B\ge a}\lf(\int_0^{2^{k}r_B}\fint_{{2^{k}B}}|t\nabla u|^2\d \mu\frac{\d t}{t}\r)^{1/2} \\
%&=\sup_{B: r_B\ge 2^ka}\lf(\int_0^{r_B}\fint_{B}|t\nabla u|^2\d \mu\frac{\d t}{t}\r)^{1/2} \\
%&\le \sup_{B: r_B\ge a}\lf(\int_0^{r_B}\fint_{B}|t\nabla u|^2\d \mu\frac{\d t}{t}\r)^{1/2} \to0, \quad a\to\fz,
%\end{align*}
%and
%$$
%\sup_{\substack{B:B\subset B(x_0,a)^\complement\\ r_B\ge \rho(x_B)}}\eta_k(u,B)
% \le \sup_{B: B\subset B(x_0,a)^\complement}\lf(\int_0^{2^{k}r_B}\fint_{{2^{k}B}}|t\nabla u|^2\d \mu\frac{\d t}{t}\r)^{1/2}\to0, \quad a\to\fz.
%$$
%Therefore by repeating the arguments in Case 1 we arrive at
% $$\gz_4(f)=\gz_5(f)=0,$$
%which completes the proof.
\end{proof}

\subsection{From trace to solution}
\hskip\parindent
In this subsection we complete the proof of Theorem \ref{thm2} by showing that
every $\CMO_\LV$-function induces a vanishing Carleson measure $\big|t\nabla_{x,t}e^{-t\sqrt{\LV}}f\big|^2\d\mu\d t/t$.

\begin{theorem}\label{cmo}
Let $(X,d,\mu,\mathscr{E})$ be a complete Dirichlet metric measure space satisfying a doubling property $(D)$ and admitting an $L^2$-Poincar\'{e} inequality $(P_2)$.
Suppose that $0\le V\in A_\infty(X)\cap RH_{q}(X)$ for some $q> \max\{Q/2,1\}$.
If $f\in \CMO_\LV(X)$, then $u=e^{-t\sqrt{\LV}}f\in \HCMO_{\LV}(X\times \rr_+)$.
Moreover, it holds that % there exists a constant $C>0$ independent of $f$ such that
$$\|u\|_{\HMO_{\LV}}\le C\|f\|_{\BMO_\LV}.$$
 \end{theorem}

\begin{proof}
%(from CMO to HCMO)
{\bf Step 1:} reduction.
If $f\in\mathrm{CMO}_\LV(X)$, then $f\in\mathrm{BMO}_\LV(X)$.
By Theorem \ref{thm1}, we see  $u=e^{-t\sqrt{\LV}}f\in \mathrm{HMO}_\LV(X\times\rr_+)$ with
$$ \|u\|_{\mathrm{HMO}_\LV}  \le C\|f\|_{\mathrm{BMO}_\LV}.$$
It remains to verity $u$ satisfies the three limiting conditions
$$\bz_1(u)=\bz_2(u)=\bz_3(u)=0.$$
To this end, we split $f$ into
$$f=(f-f_B)+f_B.$$

\smallskip

{\bf Step 2:} the difference term satisfies three limiting conditions.
For the difference term, the classical argument from the proof of \cite[Proposition 5.3]{JL2022} implies that for any ball $B=B(x_B,r_B)\subset X$,
$$
\lf(\int_0^{r_B}\fint_{B} \left|t\nabla_{x,t} e^{-t\sqrt{\LV}}(f-f_B)\right|^2\d \mu\frac{\d t}{t}\r)^{1/2}\le C\sum_{k=1}^\fz4^{-k}\sum_{j=1}^k\lf(\fint_{4^{j}B}|f-f_{4^{j}B}|^2\d\mu\r)^{1/2}.
$$
Similar to Theorem \ref{hcmo}, it is easy to verify the term $f-f_B$ satisfies three limiting conditions, the details being omitted here.

\smallskip

{\bf Step 3:} estimate the constant term.
Thus the main difficulty of showing Theorem \ref{cmo} arises from the constant term $f_B$.
%The argument presented in the following partly follows the proof of \cite[Theorem A]{SW2022}.
%We consider two cases: the ball with small radius \& the ball with large radius.
%
%Case 1: the ball with small radius.
By \cite[(7.1)]{LI}, we see that for any ball $B=B({x_B,r_B})\subset X$ with $r_B<\rho(x_B)$, it holds that
$$|f_B|\le C\lf(\frac{\rho(x_B)}{r_B}\r)^{\dz/4}\|f\|^{1-\varepsilon}_{\BMO_\LV}|f|^\varepsilon_{B(x_B,\rho(x_B))}$$
for some positive constant $\varepsilon$ (which may be small sufficiently).
%Indeed, on the one hand, there holds by $(D)$ that
%$$
%|f_B|\le C\frac{\mu(B({x_B,\rho(x_B)}))}{\mu(B({x_B,r_B}))}\fint_{B({x_B,\rho(x_B)})}|f|\d\mu\le C\lf(\frac{\rho(x_B)}{r_B}\r)^Q|f|_{B(x_B,\rho(x_B))}.
%$$
%On the other hand, note that
%$$
%|f_B|\le |f_{B({x_B,r_B})}-f_{B(x_B,\rho(x_B))}|+|f_{B(x_B,\rho(x_B))}|\le C\lf(1+\log\frac{\rho(x_B)}{r_B}\r)\|f\|_{\BMO_\LV}.
%$$
%Hence we have
%\begin{align*}
%|f_B|
%&\le C\lf(1+\log\frac{\rho(x_B)}{r_B}\r)^{1-\varepsilon}\lf(\frac{\rho(x_B)}{r_B}\r)^{Q\varepsilon}\|f\|^{1-\varepsilon}_{\BMO_\LV}|f|^\varepsilon_{B(x_B,\rho(x_B))}\\
%&\le C\lf(\frac{\rho(x_B)}{r_B}\r)^{\dz/4}\|f\|^{1-\varepsilon}_{\BMO_\LV}|f|^\varepsilon_{B(x_B,\rho(x_B))}
%\end{align*}
%as claimed.
From this estimate and Lemma \ref{lem-1}, it holds
$$
\lf[\int_0^{r_B}\fint_{B}  \left |t\nabla_x e^{-t\sqrt{\LV}}f_B\right|^2\d\mu\frac{\d t}{t}\r]^{1/2}
\le
\begin{cases}
\displaystyle C\lf(\frac{r_B}{\rho(x_B)}\r)^{\dz/2}\|f\|^{1-\varepsilon}_{\BMO_\LV}|f|^\varepsilon_{B(x_B,\rho(x_B))},\quad &r_B<\rho(x_B), \\[8pt]
\displaystyle C|f|_{B},&r_B\ge\rho(x_B).
\end{cases}
$$

{\bf Step 4:} the constant term satisfies three limiting conditions.
Once the estimate of $f_B$ in hand, we can consider three limiting conditions at infinity and zero, respectively.
Hence the rest proof of this step is done verbatim as in \cite[Step 4]{LI} (see \cite{SW2022} for the Euclidean case),
and is left to the interested reader.
%We leave the detail to the interested reader.

{\bf Step 5:} completion of the proof.
From Steps 2 and 4, we know that the difference term and the constant term satisfy three limiting conditions, which implies
$$\bz_1(u)=\bz_2(u)=\bz_3(u)=0,$$
and hence completes the proof.
\end{proof}

\section{Application III: harmonic function with weighted Morrey trace}\label{s6}

\subsection{Preliminaries: Muckenhoupt class associated with $V$}
\hskip\parindent
As in \cite{BHS2011}, we say that a weight $\oz$ belongs to the Muckenhoupt class $A_p^{\rho,\theta}(X)$ with $1\le p<\fz$, if
$$
[\oz]_{A_p^{\rho,\theta}}=
\begin{cases}
\dis \sup_B\lf(\frac{1}{\Psi_\theta(B)}\fint_B \oz \d \mu\r)\lf(\sup_{B}\oz^{-1}\r),\quad &p=1,\\[8pt]
\dis \sup_B\lf(\frac{1}{\Psi_\theta(B)}\fint_B \oz \d \mu\r)\lf(\frac{1}{\Psi_\theta(B)}\fint_B \oz^{-\frac{1}{p-1}} \d \mu\r)^{p-1}, &1<p<\fz,
\end{cases}
$$
is finite, where $\Psi_\theta$ is an auxiliary function defined by
$$
\Psi_\theta(B)=\lf(1+\frac{r_B}{\rho(x_B)}\r)^{\theta}
$$
for some $0< \theta<\fz$.
When $\theta=0$, the class $A_p^{\rho,\theta}(X)$ goes back to the classical Muckenhoupt class; see \cite{BLYZ2008,CF1974,Li2019,Mu1972,ST1989} for example.
In fact, when $\theta\neq0$, if we pick $\oz(x)=(1+|x|)^n$ on $\rn$, then
it is not a classical Muckenhoupt $A_\fz$-weight, but
$$
\oz(x)=(1+|x|)^n\in A_1^{\rho,\theta}(\rn)
$$
provided that $V=1$ and $\Psi_\theta(B)=(1+r_B)^\theta$.
We denote $A_p^\rho(X)=\bigcup_{\theta>0}{A_p^{\rho,\theta}}(X)$ for each $1\le p<\fz$.

In the end of this subsection, we collect some properties of $A_p^\rho(X)$ with $1\le p<\fz$; see \cite{BHS2011,Ta2015,ZT2016} for the Euclidean case.

\begin{lemma}\label{lem0}
Let $1\le p<\fz$. The following statements are valid.
\begin{itemize}
 \item[{(i)}]  If $0< \theta_1\le \theta_2<\fz$, then $A_p^{\rho,\theta_1}(X)\subset A_p^{\rho,\theta_2}(X)$.

 \item[{(ii)}]  If $1\le p_1\le  p_2<\fz$, then $A_{p_1}^{\rho,\theta}(X)\subset A_{p_2}^{\rho,\theta}(X)$.

 \item[{(iii)}]  If $1/p+1/{q}=1$, then $\oz\in A_{p}^{\rho,\theta}(X)$ if and only if $\oz^{-\frac{1}{p-1}}\in A_{q}^{\rho,\theta}(X)$.

 \item[{(iv)}]  If $\oz\in A_p^{\rho,\theta}(X)$, then for any $0<r<R<\fz$,
 $$
 \frac{\oz(B(x,R))}{\oz(B(x,r))}\le C\lf(\frac{\mu(B(x,R))}{\mu(B(x,r))}\Psi_\theta(B(x,R))\r)^p.
 $$

 \item[{(v)}] If $\oz\in A_p^{\rho,\theta}(X)$, then there exist $\dz,\theta'>0$ such that
 $$\lf(\fint_B \oz^{1+\dz}\d\mu\r)^{\frac{1}{1+\dz}}\le C\lf(\fint_B\oz\d\mu\r)\lf(1+\frac{r_B}{\rho(x_B)}\r)^{\theta'}.$$

 \item[{(vi)}] If $\oz\in A_p^{\rho}(X)$, then $\oz^{1+\dz} (X)\in A_p^{\rho}(X)$ for some $\dz>0$.

 \item[{(vii)}] If $1< p<\fz$, then $\oz\in A_p^\rho(X)$ implies $\oz\in A_{p-\varepsilon}^\rho(X)$ for some small $\varepsilon>0$.
%\item[{(v)}]  If $1< p<\fz$, then $\oz\in A_p^\rho(X)$ if and only if $\oz=\oz_1\oz_2$, where $\oz_1,\oz_2\in A_1^\rho(X)$.
\end{itemize}
\end{lemma}

\begin{proof}
Properties (i)-(iii) are obvious by the definition of $A_p^{\rho,\theta}(X)$.
%For other properties (iv)-(xxxx), instead of giving a detailed proof, let us outline it
\smallskip

(iv) By checking the proof of \cite[Lemma 2.2 (iii)]{Ta2015}, we see that
\begin{align}\label{q2}
\frac{1}{\Psi_\theta(B)\mu(B)}\int_B |f| \d \mu \le [\oz]_{A_p^{\rho,\theta}}^{1/p} \lf(\frac{1}{\oz(B)}\int_B |f|^p\oz \d \mu\r)^{1/p}.
\end{align}
Letting $f=\mathbbm{1}_E$ with $E\subset B$ leads to the desired result.

\smallskip

(v) Following the argument in \cite[Theorem IV]{CF1974}, we obtain that
$$
\lf(\fint_B \oz^{1+\dz}\d\mu\r)^{\frac{1}{1+\dz}}\le C\lf(\fint_B\oz\d\mu\r)
$$
for every $B\subset X$ with $r_B<\rho(r_B)$.

It remains to consider the large ball case.
For any $B\subset X$ with $r_B\ge \rho(r_B)$, we can decompose it into the union of some critical balls (whose radiuses equal to the value of $\rho$ at centers)
by the the partition of unity of $X$ via critical balls; see \cite{BDL2018} for the partition of unity.
Therefore a similar argument employed in \cite[Lemma 5]{BHS2011} tells us that
 $$\lf(\fint_B \oz^{1+  \dz}\d\mu\r)^{\frac{1}{1+\dz}}\le C\lf(\fint_B\oz\d\mu\r)\lf(1+\frac{r_B}{\rho(x_B)}\r)^{\theta'},$$
where $\theta'=\frac{k_0Qp}{k_0+1}+(\frac{k_0}{k_0+1}+1)\theta p+\frac{2(k_0+1)Q\dz}{1+\dz}$.

\smallskip

(vi) This is a direct consequence of (iii) and (v).

\smallskip

(vii) We first note by the definition of $A_p^{\rho,\theta}(X)$ and H\"{o}lder's inequality that
$$[\oz^\gz]_{A_{\gz p+1-\gz}^{\rho,\theta}}\le [\oz]^\gz_{A_p^{\rho,\theta}}$$
for any $0<\gz<1$.
Then the rest proof is done verbatim as in \cite[Corollary 7.2.6]{GTM249}, and is left to the interested reader.
%We leave the detail to the interested reader.
\end{proof}

\subsection{Statement of the main result}
\hskip\parindent
To state the main result of this section, we first introduce some function classes.%
%Let $\az(\cdot)$ be a non-negative set function whose domain of definition is a class of sets in $X$, namely,
%\begin{align*}
%\az:\ &2^{X}\rightarrow\rr_+ \\
%&E\ \ \mapsto \az(E).
%\end{align*}
%For some technical reasons, we always assume that the set function $\az$ is ``decreasing'', i.e., there exists a constant $C>0$ such that,
%for any $E\subset F\subset X$,
%$$\az(F)\le C\az(E).\eqno(De)$$
%Moreover, a set function $\az$ is said to be translation invariant if it satisfies that for all $x\in X$ and  $E\subset X$,
%$$\az(x+E)=\az(E).\eqno(TI)$$
%For any ball $B=B(x_B,r_B)\subset X$, denote by $A(x_B,r_B)$ its ``volume'' with respect to the set function $\az$, namely, $A(x_B,r_B)=\az(B)$.

\begin{definition}\label{def1}
For any $1<p<\fz$, $0<\lz<1$, $\tau\in \rr$ and $\theta>0$,
%and $\lz$ is a non-negative set function.
a locally integrable function $f$ defined on $X$ is said to be in $L^{p,\lz}_{\tau}(\d\oz)=L^{p,\lz}_{\tau}(X,\d\oz)$ with $\oz\in A_p^{\rho,\theta}(X)$,
the weighted Morrey space related to $V$, if $f$ satisfies
$$
\|f\|_{L^{p,\lz}_{\tau}(\d\oz)}=\sup_{B}\lf(\frac{1}{\oz(B)^\lz}\int_B|f|^p\d \oz\r)^{1/p}\lf(1+\frac{r_B}{\rho(x_B)}\r)^{\tau}<\fz,
$$
where $\oz(B)$ stands for the weighted measure of $B$ as
$$\oz(B)=\int_B\d\oz=\int_B\oz\d\mu.$$
The union $\bigcup_{\tau\in\rr} L^{p,\lz}_{\tau}(X,\d\oz)$ is denoted by $L^{p,\lz}(X,\d\oz)$.
\end{definition}

\begin{remark}
Obviously, when $\oz=1$ and $\tau=0$,
this Morrey space is introduced by Morrey \cite{Mo1938} in the Euclidean setting. % $(X,d,\mu)=(\rn,|\cdot|,\d x)$.
Moreover, further letting $\lz=0$,
the Morrey space $L^{p,\lz}(\rn)$ goes back to the classical Lebesgue space $L^{p}(\rn)$. %  as in Definition \ref{def0}.
%When $w\equiv1$, the weighted Morrey space $L^{p,\az}_w(\rn)$ coincides with the generalized Morrey space $L^{p,\az}(\rn)$ introduced by Nakai \cite{Na1994}.
\end{remark}

\begin{definition}\label{def2}
For any $1<p<\fz$, $0<\lz<1$, $\tau\in \rr$ and $\theta>0$,
%Suppose that $\oz\in A_p^{\rho,\theta}(X)$, and $\lz$ is a non-negative set function.
an $\mathbb{L}$-harmonic function $u$ defined on $X\times \rr_+$
is said to be in $L^\fz(L^{p,\lz}_\tau)(\d\oz\d t)=L^\fz(L^{p,\lz}_\tau)(X\times \rr_+,\d\oz\d t)$ with $\oz\in A_p^{\rho,\theta}(X)$,
the weighted harmonic Morrey space related to $V$,
if $u$ satisfies
$$
\|u\|_{L^\fz(L^{p,\lz}_\tau)(\d\oz\d t)}=\sup_{t>0}\|u(\cdot,t)\|_{L^{p,\lz}_\tau(\d\oz)}<\fz.
$$
The union $\bigcup_{\tau\in\rr} L^\fz(L^{p,\lz}_\tau)(X\times \rr_+,\d\oz\d t)$ is denoted by $L^\fz(L^{p,\lz})(X\times \rr_+,\d\oz\d t)$.
\end{definition}

Under the Euclidean setting, Song--Tian--Yan \cite{STY2018} characterized the $\mathbb{L}$-harmonic function with Morrey trace ($\oz=1$ and $\tau=0$)
via the Carleson measure condition
provided $V\in RH_q(\rn)$ for some $q>n$.
%This result was further improved to the metric measure space and $V\in RH_{(Q+1)/2}(X)$ by Shen-Li \cite{SL2022}.
This result was further extended to the metric measure space and the critical reverse H\"{o}lder order improved to $(Q+1)/2$ by Shen--Li \cite{SL2022}.
%Here there are two points need to be mentioned.
As an application of Sections \ref{ho} and \ref{improved}, %the H\"{o}lder continuity (see Proposition \ref{holder}) and the Liouville property (see Proposition \ref{Liouville})
we can further improve the critical reverse H\"{o}lder order $(Q+1)/2$ of Shen--Li \cite{SL2022} to $\max\{Q/2,1\}$.
With the help of the weighted Morrey function related to $V$ in Definition \ref{def1}, we can derive another boundary behaviour of the $\mathbb{L}$-harmonic function as follows.

\begin{theorem}\label{thm3}
Let $(X,d,\mu,\mathscr{E})$ be a complete Dirichlet metric measure space satisfying a doubling property $(D)$ and admitting an $L^2$-Poincar\'{e} inequality $(P_2)$.
Suppose that $0\le V\in A_\infty(X)\cap RH_{q}(X)$ for some $q>\max\{Q/2,1\}$.
For any $1<p<\fz$, $0<\lz<1$ and $\theta>0$,
if $\oz\in A_p^{\rho,\theta}(X)$, then $u\in L^\fz(L^{p,\lz})(X\times \rr_+,\d\oz\d t)$ if and only if $u=e^{-t\sqrt{\LV}} f$ for some $f\in L^{p,\lz}(X,\d\oz)$.
Moreover, for any $\tau\in\rr$, it holds that
$$\|f\|_{L^{p,\lz}_\tau(\d\oz)}\le \|u\|_{L^\fz(L^{p,\lz}_\tau)(\d\oz\d t)}\le C\|f\|_{L^{p,\lz}_{\tau+\lz\theta}(\d\oz)}.$$
%
%
%
%
%Then the following statements hold true.
%\begin{enumerate}
%   \item[\rm{(i)}] If $u\in L^\fz(L^{p,\az})(X\times \rr_+)$, then there exists a function $f\in L^{p,\az}(X)$ such that $u(x,t)=\PV_tf(x)$.
%Moreover, it holds
%$$\|f\|_{L^{p,\az}}\le \|u\|_{L^\fz(L^{p,\az})}.$$
% \item[\rm{(ii)}]If $f\in L^{p,\az}(X)$, then $u(x,t)=\PV_tf(x)\in L^\fz(L^{p,\az})(X\times \rr_+)$,
%and there exists a constant $C>1$ independent of $f$ such that
%$$\|u\|_{L^\fz(L^{p,\az})}\le C\|f\|_{L^{p,\az}}.$$
%\end{enumerate}
\end{theorem}

\begin{remark}
When $(X,d,\mu)=(\rn,|\cdot|,\d x)$,  $V=0$, $\lz=\tau=0$ and $\oz=1$,
our Theorem \ref{thm3} goes back to the result of Stein--Weiss; see \cite[Chapter II]{SW1971} for more details.
%But other than that, this result is new.
It is remarkable that our method and technique in the proof
of Theorem \ref{thm3} are more or less different from that of Stein--Weiss \cite{SW1971}.
On the one hand, from the solution to trace, the sub-half-space
$$\rn\times\rr^{t_0}_+=\{(t,x)\in \rn\times\rr_+: t\ge t_0>0\}$$
plays a key role in the proof of Stein--Weiss.
Without this sub-half-space,
we directly control the upper bound of an $\mathbb{L}$-harmonic function
and its Poisson integral via the mean value property and the Poisson kernel.
With the help of these estimates, %and the reflection method,
we can derive the Poisson semigroup can be regarded as a translation transformation with respect to the time variable;
see Lemma \ref{lem1} and its proof for more details.
Therefore we can seek the trace of the $\mathbb{L}$-harmonic function
via the feature of the Morrey function and the locally Lebesgue function; see the proof of Theorem \ref{bddmo} for more details.
On the other hand, from trace to solution, the arguments of Stein--Weiss is very simple.
It relies heavily on the conservative property of the Poisson semigroup,
and the Young inequality by noting the facts the Poisson kernel is in $L^1(\rn)$ and the initial value is in $L^p(\rn)$.
Unfortunately, due to the fact that the non-negative potential $V$ is non-trivial, the conservative property no longer holds.
Moreover, the Young inequality may fail to the Morrey function as the following reason.
The concept of the integral average over a ball in the Morrey function is totally different from that of the Lebesgue function,
and hence the Young inequality is no longer applicable due to the complex structure of a Morrey function.
To overcome these difficulties, we consider the small time case and the large time case to estimate the Poisson integral of a weighted Morrey function
via the annulus technique, the variant Hardy--Littlewood maximal function, and the pointwise bound for the Poisson kernel associated to $\LV$;
see the proof of Theorem \ref{moo} for more details.
\end{remark}

\subsection{From solution to trace}
\hskip\parindent
In this subsection we seek the Morrey trace of a weighted harmonic Morrey function with desired norm controlled.
\begin{theorem}\label{bddmo}
Let $(X,d,\mu,\mathscr{E})$ be a complete Dirichlet metric measure space satisfying a doubling property $(D)$ and admitting an $L^2$-Poincar\'{e} inequality $(P_2)$.
Suppose that $0\le V\in A_\infty(X)\cap RH_{q}(X)$ for some $q>\max\{Q/2,1\}$.
%Assume that the non-negative set function $\lz$ is ``decreasing'' $(De)$ and  translation invariant $(TI)$.
For any $1<p<\fz$, $0<\lz<1$ and $\theta>0$,
if $u\in L^\fz(L^{p,\lz})(X\times \rr_+,\d\oz\d t)$ with $\oz\in A_p^{\rho,\theta}(X)$, then $u=e^{-t\sqrt{\LV}}f$ for some $f\in L^{p,\lz}(X,\d\oz)$.
Moreover, for any $\tau\in\rr$, it holds that
$$\|f\|_{L^{p,\lz}_\tau(\d\oz)}\le \|u\|_{L^\fz(L^{p,\lz}_\tau)(\d\oz\d t)}.$$
\end{theorem}

To prove this theorem, let us establish some properties of a weighted harmonic Morrey function.

\begin{lemma}\label{lem1}
Let $(X,d,\mu,\mathscr{E})$ be a complete Dirichlet metric measure space satisfying a doubling property $(D)$ and admitting an $L^2$-Poincar\'{e} inequality $(P_2)$.
Suppose that $0\le V\in  A_\infty(X)\cap RH_{q}(X)$ for some $q>\max\{Q/2,1\}$.
%Assume that the non-negative set function $\lz$ is ``decreasing'' $(De)$ and  translation invariant $(TI)$.
For any $1<p<\fz$, $0<\lz<1$, $\tau\in \rr$ and $\theta>0$,
if $u\in L^\fz(L^{p,\lz}_\tau)(X\times \rr_+,\d\oz\d t)$ with $\oz\in A_p^{\rho,\theta}(X)$,
then the following statements are valid.
\begin{itemize}
  \item[{(i)}] There exists a constant $C>0$ such that for any $x\in X$ and $t>0$,
$$|u(x,t)|\le C\oz(B(x,t))^{\frac{\lz-1}{p}} \|u\|_{L^\fz(L^{p,\lz}_\tau)(\d\oz\d t)}.$$

  \item[{(ii)}] The Poisson integral $e^{-t\sqrt{\LV}}u_s(x)$ of $u_s(\cdot)=u(\cdot,s)$ is well defined on $X\times \rr_+$,
  and satisfies
  $$\lim_{t\to0}e^{-t\sqrt{\LV}}u_s(x)=u_s(x).$$
Moreover, there exists a constant $C>0$ such that for any $x\in X$ and $t,s>0$,
$$
\left|e^{-t\sqrt{\LV}}u_s(x)\right|\le C\oz(B(x,s))^{\frac{\lz-1}{p}}\lf[\lf(1+\frac{s}{\rho(x)}\r)^{(k_0+1)\theta}+\lf(\frac{\rho(x)}{s}\r)^Q\r]
 \|u\|_{L^\fz(L^{p,\lz}_\tau)(\d\oz\d t)}.
 $$

 \item[{(iii)}]  For any $x\in X$ and $s,t>0$, it holds
$$
  u(x,t+s)=e^{-t\sqrt{\LV}} u_s(x).
$$
\end{itemize}
\end{lemma}
\begin{proof}
(i) %Note that the potential $V$ is non-negative.
%By the mean value property of the subharmonic function, and H\"{o}lder's inequality, we arrive at
By Lemma \ref{mean}, %and the mean value property of the subharmonic function,
we arrive at
\begin{align*}
|u(x,t)|
%&\le C\sup_{B(x,t/4)\times(3t/4,5t/4)}|u|\exp\lf\{-\varepsilon\lf(1+\frac{t}{\rho(x)}\r)^{\frac{1}{k_0+1}}\r\}\\
&\le C\fint_{t/2}^{3t/2}\fint_{B(x,t)}|u(y,s)|\d\mu(y)\d s\lf(1+\frac{t}{\rho(x)}\r)^{-N}\\
&\le C\fint_{t/2}^{3t/2}\lf(\fint_{B(x,t)}|u(y,s)|^p\d\oz(y)\r)^{1/p}\lf(1+\frac{t}{\rho(x)}\r)^{\theta}\d s \lf(1+\frac{t}{\rho(x)}\r)^{-N}\\
%&\le C\fint_{t/2}^{3t/2}|Q(x,t)|^{-\frac{1}{n}(\sum_{j=1}^n \frac{1}{p_j})} \|U(s,y)\mathbbm{1}_{Q(x,t)}(y)\|_{L_y^{\vec{p}}(\rn)}s^{1-2\sz}\d s\\
%&\le C\fint_{t/2}^{3t/2}|Q(x,t)|^{\az}|Q(x,t)|^{-\az-\frac{1}{n}(\sum_{j=1}^n \frac{1}{p_j})} \|U(s,y)\mathbbm{1}_{Q(x,t)}(y)\|_{L_y^{\vec{p}}(\rn)}s^{1-2\sz}\d s\\
%&\le C\fint_{t/2}^{3t/2}\oz(B(x,t))^{\frac{\lz-1}{p}}\|u\|_{L^\fz(L^{p,\lz}_\tau)(\d\oz\d t)}\d s \\
&\le C\oz(B(x,t))^{\frac{\lz-1}{p}}\|u\|_{L^\fz(L^{p,\lz}_\tau)(\d\oz\d t)},
\end{align*}
where the third inequality  is a direct consequence of \eqref{q2}.

%On the one hand, by an annulus argument, the $\sz$-Poisson upper bound and Lemma \ref{lem0}, we deduce
%\begin{align*}
%|U_s(t,x)|
%%&\le \int_\rn p_t(x,y)|u_s(y)|\d y \\
%%&\le \lf\{\int_{Q(x,t)}+\sum_{k=1}^\fz\int_{Q(x,2^kt)\setminus Q(x,2^{k-1}t)}\r\} p_t(x,y)|U(s,y)| \d y\\
%&\le C\lf\{\int_{Q(x,t)}+\sum_{k=1}^\fz\int_{Q(x,2^kt)\setminus Q(x,2^{k-1}t)}\r\} \lf(\frac{t}{t+|x-y|}\r)^{2\sz}\frac{1}{(t+|x-y|)^n}|U(s,y)| \d y\\
%&\le C\sum_{k=0}^\fz2^{-2\sz k}\fint_{Q(x,2^kt)}|U(s,y)|\d y \\
%%&\le C\sum_{k=0}^\fz2^{-2\sz k} |Q(x,2^kt)|^{-\frac{1}{n}(\sum_{j=1}^n \frac{1}{p_j})} \|U(s,y)\mathbbm{1}_{Q(x,2^kt)}(y)\|_{L_y^{\vec{p}}(\rn)} \\
%&\le C\sum_{k=0}^\fz2^{-2\sz k} |Q(x,2^kt)|^{\az}\|U(s,y)\|_{L^\fz_s(L^{\vec{p},\az}_y)(\rr_+\times\rr^{n})}\\
%&\le C  |Q(x,t)|^{\az} \|U(t,x)\|_{L^\fz_t(L^{\vec{p},\az}_x)(\rr_+\times\rr^{n})}.
%\end{align*}
(ii) {\bf Step 1:} the Poisson integral is well defined. It follows from Lemma \ref{lem0} and \eqref{critical} that
\begin{align*}
\left|e^{-t\sqrt{\LV}}u_s(x)\right|
%&\le \int_Xp^v_t(x,y)|u(y,s)|\d\mu(y)\\
%&\le C\int_X p^v_t(x,y)\oz(B(y,s))^{\frac{\lz-1}{p}}\|u\|_{L^\fz(L^{p,\lz}_\tau)(\d\oz\d t)}\d\mu(y)\\
%&\le C\int_X p^v_t(x,y)\lf(\frac{\oz(B(x,s))}{\oz(B(y,s))}\r)^{\frac{1-\lz}{p}}\d\mu(y) \oz(B(x,s))^{\frac{\lz-1}{p}}\|u\|_{L^\fz(L^{p,\lz}_\tau)(\d\oz\d t)} \\
&\le C\int_X p^\LV_t(x,y) \oz(B(x,s))^{\frac{\lz-1}{p}}\lf(\frac{\oz(B(y,s+d(x,y)))}{\oz(B(y,s))}\r)^{\frac{1-\lz}{p}}\d\mu(y)\|u\|_{L^\fz(L^{p,\lz}_\tau)(\d\oz\d t)} \\
%&\le C\int_X p^v_t(x,y)\lf(\frac{\oz(B(y,s+d(x,y)))}{\oz(B(y,s))}\r)^{1/p}\d\mu(y) \oz(B(x,s))^{\frac{\lz-1}{p}}\|u\|_{L^\fz(L^{p,\lz}_\tau)(\d\oz\d t)} \\
%&\le C\oz(B(x,s))^{\frac{\lz-1}{p}}\int_X p^v_t(x,y)\frac{\mu(B(y,s+d(x,y)))}{\mu(B(y,s))}\lf(1+\frac{s+d(x,y)}{\rho(y)}\r)^\theta\d\mu(y)
%  \|u\|_{L^\fz(L^{p,\lz}_\tau)(\d\oz\d t)} \\
&\le C \oz(B(x,s))^{\frac{\lz-1}{p}}\int_X p^\LV_t(x,y)\lf(1+\frac{d(x,y)}{s}\r)^Q\lf(1+\frac{s+d(x,y)}{\rho(x)}\r)^{(k_0+1)\theta}\d\mu(y)
  \|u\|_{L^\fz(L^{p,\lz}_\tau)(\d\oz\d t)} \\
&= C \oz(B(x,s))^{\frac{\lz-1}{p}}\lf\{\int_{d(y,x)<\min\{t,s\}}+\int_{s\le d(y,x)<t}+\int_{t\le d(y,x)<s}+\int_{d(y,x)\ge\max\{t,s\}}\r\} \cdots .
\end{align*}
We only consider the last global term since other terms can be handle in a similar way.
The last global term can be majorized by
%One may invoke $(PUB)$ to obtain
$$\int_{d(y,x)\ge\max\{t,s\}}\frac{t}{d(x, y)\mu(B(x,d(x, y)))}\lf(1+\frac{d(x,y)}{\rho(x)}\r)^{-N+(k_0+1)\theta}\lf(1+\frac{d(x,y)}{s}\r)^Q\d\mu(y)
\le C\left(\frac{\rho(x)}{s}\right)^{Q}.$$
%\begin{align*}
%&\int_{d(y,x)\ge\max\{t,s\}}p^v_t(x,y)\lf(1+\frac{d(x,y)}{s}\r)^Q\lf(1+\frac{s+d(x,y)}{\rho(x)}\r)^{(k_0+1)\theta}\d\mu(y) \\
%&\ \le C\int_{d(y,x)\ge\max\{t,s\}}
% \frac{t}{d(x, y)\mu(B(x,d(x, y)))} \left(\frac{\rho(x)}{\rho(x)+d(x,y)}\right)^{Q}\lf(\frac{d(x,y)}{s}\r)^Q\d\mu(y)\\
%%&\ \le C\int_{d(y,x)\ge t}
%% \frac{t}{d(x, y)} \frac{1}{\mu(B(x,d(x, y)))}\left(\frac{\rho(x)}{s}\right)^{Q}\lf(\frac{s+d(x,y)}{\rho(x)+d(x,y)}\r)^Q\d\mu(y) \\
%&\ \le C\left(\frac{\rho(x)}{s}\right)^{Q}.
%\end{align*}

%Apply (i), $(TI)$ and $(PUB)$ to obtain
%\begin{align*}
%|\PV_tu_s(x)|
%&\le \int_X p^v_t(x,y)|u_s(y)|\d\mu(y) \\
%&\le C\int_X p^v_t(x,y)A(y,s)\|u\|_{L^\fz(L^{p,\lz}_\tau)(\d\oz\d t)}\d\mu(y)\\
%&\le CA(x,s)\|u\|_{L^\fz(L^{p,\lz}_\tau)(\d\oz\d t)}.
%\end{align*}
%which completes the proof.
%Based on the above arguments, we arrive at
%\begin{align*}
%|U_s(t,x)|\le C|Q(x,t+s)|^\az \|U(t,x)\|_{L^\fz_t(L^{\vec{p},\az}_x)(\rr_+\times\rr^{n})}
%\end{align*}
%as desired. This completes the proof.

{\bf Step 2:} the limiting behavior.
It remains to show
$$\lim_{t\to0}e^{-t\sqrt{\LV}}u_s(x)=u_s(x).$$
To this end, note by the conservative property that
\begin{align*}
\left|e^{-t\sqrt{\LV}}u_s(x)-u_s(x)\right|
%&=\lf|\int_{X}p^v_t(x,y)u_s(y)\d\mu(y)-\int_{X}p_t(x,y)u_s(x)\d\mu(y)\r| \\
%&\le \int_{X}p^v_t(x,y)|u_s(y)-u_s(x)|\d\mu(y)+\int_{X}|p^v_t(x,y)-p_t(x,y)|\d\mu(y)|u_s(x)| \\
&\le \int_{B(x,s/32)}p^\LV_t(x,y)|u_s(y)-u_s(x)|\d\mu(y)+\int_{X\setminus B(x,s/32)}p^\LV_t(x,y)|u_s(y)-u_s(x)|\d\mu(y) \\
&\ +\int_{X}\left|p^\LV_t(x,y)-p^\L_t(x,y)\right|\d\mu(y)|u_s(x)|.
\end{align*}

{\bf Substep 2.1:} the limiting behavior of the first term. From the H\"{o}lder regularity of $u$ (see Theorem \ref{holder1}), we deduce that
\begin{align*}
\int_{B(x,s/32)}p^\LV_t(x,y)|u_s(y)-u_s(x)|\d\mu(y)
&\le C\int_{B(x,s/8)}p^\LV_t(x,y)\lf(\frac{d(x,y)}{s}\r)^\az\d\mu(y)\sup_{B(x,s/4)\times(3s/4,5s/4)}|u| \\
%&\ \le C\lf(\frac{t}{s}\r)^\az\int_X p^v_t(x,y)\lf(\frac{d(x,y)}{t}\r)^\az\d\mu(y)\lf(1+s^2\fint_{B(x,s/8)}V\d\mu\r)\oz(B(x,s))^{\frac{\lz-1}{p}}\|u\|_{L^\fz(L^{p,\lz}_\tau)(\d\oz\d t)} \\
& \le C\lf(\frac{t}{s}\r)^\az\oz(B(x,s))^{\frac{\lz-1}{p}}\|u\|_{L^\fz(L^{p,\lz}_\tau)(\d\oz\d t)}\to0,\quad t\to0.
\end{align*}
%\begin{align*}
%&\int_{B(x,s/32)}p^\LV_t(x,y)|u_s(y)-u_s(x)|\d\mu(y) \\
%&\ \le C\int_{B(x,s/8)}p^v_t(x,y)\lf(\frac{d(x,y)}{s}\r)^\az\d\mu(y)\lf(1+s^2\fint_{B(x,s/8)}V\d\mu\r)\sup_{B(x,s/4)\times(3s/4,5s/4)}|u| \\
%%&\ \le C\lf(\frac{t}{s}\r)^\az\int_X p^v_t(x,y)\lf(\frac{d(x,y)}{t}\r)^\az\d\mu(y)\lf(1+s^2\fint_{B(x,s/8)}V\d\mu\r)\oz(B(x,s))^{\frac{\lz-1}{p}}\|u\|_{L^\fz(L^{p,\lz}_\tau)(\d\oz\d t)} \\
%&\ \le C\lf(\frac{t}{s}\r)^\az\lf(1+s^2\fint_{B(x,s/8)}V\d\mu\r)\oz(B(x,s))^{\frac{\lz-1}{p}}\|u\|_{L^\fz(L^{p,\lz}_\tau)(\d\oz\d t)}\to0,\quad t\to0.
%\end{align*}

{\bf Substep 2.2:} the limiting behavior of the second term.
There holds by $(PUB)$ and (i) that
\begin{align*}
\int_{X\setminus B(x,s/32)}p^\LV_t(x,y)|u_s(y)-u_s(x)|\d\mu(y)
%&\le \int_{X\setminus B(x,s/8)}p^\LV_t(x,y)|u_s(y)|\d\mu(y)+\int_{X\setminus B(x,s/8)}p^\LV_t(x,y)|u_s(x)|\d\mu(y)\\
& \le \sum_{k=-4}^\fz\int_{B(x,2^ks)\setminus B(x,2^{k-1}s)}p^\LV_t(x,y)|u_s(y)|\d\mu(y)+C\lf(\frac ts\r) |u_s(x)|  \\
& \le C\lf(\frac ts\r)\sum_{k=-4}^\fz 2^{-k}\lf(1+\frac{2^ks}{\rho(x)}\r)^{-N}\fint_{B(x,2^ks)}|u_s|\d\mu+C\lf(\frac ts\r) |u(x,s)| \\
& \le C\lf(\frac ts\r)\oz(B(x,s))^{\frac{\lz-1}{p}}\|u\|_{L^\fz(L^{p,\lz}_\tau)(\d\oz\d t)} \to0,\quad t\to0.
\end{align*}

{\bf Substep 2.3:} the limiting behavior of the last term.
By the perturbation estimate of the Poisson kernel from Lemma \ref{lem-1},
%$$|e^{-t\LV}1(x)-1|\le C\lf(\frac{\sqrt{t}}{\sqrt{t}+\rho(x)}\r)^{2-Q/q}\le C\lf(\frac{\sqrt{t}}{\rho(x)}\r)^{\min\{1/2,2-Q/q\}}$$
%from \cite[Proposition 7.13]{BDL2018} (see also \cite[Corollary 5.3]{LI}), and the Bochner subordination formula,
we arrive at
$$
 \int_{X}\left|p^\LV_t(x,y)-p^\L_t(x,y)\right|\d\mu(y)|u_s(x)|
%\le \int_{X}\lf(\frac{{t}}{\rho(x)}\r)^{\min\{1/2,2-Q/q\}}\frac{t}{t+d(x, y)} \frac{\d\mu(y)}{\mu(B(x,t+d(x, y)))}|u_s(x)| \\
%&\le C \int_0^\fz\frac{t}{s^{1/2}}\exp\lf(-\frac{t^2}{4s}\r)| e^{-s\LV}1(x)-1|\frac{\d s}{s}|u_s(x)| \\
%&\le C \int_0^\fz\frac{t}{s^{1/2}}\exp\lf\{-\frac{t^2}{4s}\r\}\lf(\frac{\sqrt{s}}{\rho(x)}\r)^{\min\{1/2,2-Q/q\}}\frac{\d s}{s}|u_s(x)| \\
\le C\lf(\frac{t}{\rho(x)}\r)^{\min\{1/2,2-Q/q\}}|u_s(x)|\to0,\quad t\to0.
$$
%To handle the second term, applying the H\"{o}lder continuity of $u$ (see Proposition \ref{holder}), and invoking the upper bound of $p_t(x,y)$ yield that
%\begin{align*}
%|P_t(u_s-u_s(x))(x)|
%&=\int_X p_t(x,y)|u(y,s)-u(x,s)|\d\mu(y)\\
%&=\int_{B(x,s/4)} p_t(x,y)|u(y,s)-u(x,s)|\d\mu(y)+\int_{X\setminus B(x,s/4)}\cdots \d\mu(y)\\
%&=\int_{B(x,s/4)} p_t(x,y)\lf(\frac{t}{s}\r)^\eta\lf(1+s^2\fint_{B(x,s/4)}V\d\mu\r)\d\mu(y)A(x,s)\|u\|_{L^\fz(L^{p,\lz}_\tau)(\d\oz\d t)} \\
%&\ +\int_{X\setminus B(x,s/4)} \frac{t}{d(x,y)}\frac{1}{\mu(B(x,d(x,y)))}\d\mu(y)A(x,s)\|u\|_{L^\fz(L^{p,\lz}_\tau)(\d\oz\d t)} \\
%&\le C \max\lf\{\lf(\frac{t}{s}\r)^\eta,\frac{t}{s}\r\} A(x,s)\|u\|_{L^\fz(L^{p,\lz}_\tau)(\d\oz\d t)} \to0,\quad t\to0.
%\end{align*}

{\bf Step 3:} collect all terms.
Finally, one has
\begin{align*}
\lim_{t\to0}e^{-t\sqrt{\LV}}u_s(x)
&=\lim_{t\to0}\int_{B(x,s/32)}p^\LV_t(x,y)(u_s(y)-u_s(x))\d\mu(y)\\
&\ +\lim_{t\to0}\int_{X\setminus B(x,s/32)}p^\LV_t(x,y)(u_s(y)-u_s(x))\d\mu(y) \\
&\ +\lim_{t\to0}\int_{X}(p^\LV_t(x,y)-p^\L_t(x,y))\d\mu(y)u_s(x)+u_s(x)=u_s(x).
\end{align*}

(iii) Define
$$H(x,t)=u(x,t+s)-e^{-t\sqrt{\LV}} u_s(x).$$
First note that $H$ is $\mathbb{L}$-harmonic on $X\times\rr_+$ with
%we deduce from (ii) that
$$H(x,0)=\lim_{t\to0}H(x,t)=\lim_{t\to0}u(x,t+s)-\lim_{t\to0}e^{-t\sqrt{\LV}} u_s(x)=u(x,s)- u_s(x)=0$$
by the H\"{o}lder continuity of $u$ (see Theorem \ref{holder}).

We claim that $H$ has polynomial growth at most.
%Obviously $H$ is an $\mathbb{L}$-harmonic function on $X\times\rr_+$.
Indeed, it follows from Lemma \ref{lem0} and \eqref{critical} that
\begin{align*}
\oz(B(x,s))^{\frac{\lz-1}{p}}
%&\le \oz(B(x_0,s))^{\frac{\lz-1}{p}} \lf(\frac{\oz(B(x_0,s))}{\oz(B(x,s))}\r)^{\frac{1-\lz}{p}} \\
%&\le \oz(B(x_0,s))^{\frac{\lz-1}{p}} \lf(\frac{\oz(B(x,s+d(x,x_0)))}{\oz(B(x,s))}\r)^{1/p} \\
%&\le C\oz(B(x_0,s))^{\frac{\lz-1}{p}} \frac{\mu(B(x,s+d(x,x_0)))}{\mu(B(x,s))}\lf(1+\frac{s+d(x,x_0)}{\rho(x)}\r)^\theta \\
%&\le C\oz(B(x_0,s))^{\frac{\lz-1}{p}} \lf(1+\frac{d(x,x_0)}{s}\r)^Q\lf(1+\frac{s+d(x,x_0)}{\rho(x)}\r)^\theta \\
%&\le C\oz(B(x_0,s))^{\frac{\lz-1}{p}} \lf(1+\frac{d(x,x_0)}{s}\r)^Q\lf[1+\frac{s+d(x,x_0)}{\rho(x_0)}\lf(1+\frac{d(x,x_0)}{\rho(x_0)}\r)^{k_0}\r]^\theta \\
&\le C\oz(B(x_0,s))^{\frac{\lz-1}{p}} \lf(1+\frac{d(x,x_0)}{s}\r)^Q\lf(1+\frac{s+d(x,x_0)}{\rho(x_0)}\r)^{(k_0+1)\theta} \\
&\le C(x_0,s) (1+d(x,x_0))^{Q+(k_0+1)\theta},
\end{align*}
and similarly
\begin{align*}
\lf(1+\frac{s}{\rho(x)}\r)^{(k_0+1)\theta}+\lf(\frac{\rho(x)}{s}\r)^Q
& \le C\lf(1+\frac{s+d(x,x_0)}{\rho(x_0)}\r)^{(k_0+1)^2\theta}+C\lf[\frac{\rho(x_0)}{s}\lf(1+\frac{d(x,x_0)}{\rho(x_0)}\r)^{\frac{k_0}{k_0+1}}\r]^Q \\
%&\ \le C\lf(1+\frac{s+d(x,x_0)}{\rho(x_0)}\r)^{(k_0+1)^2\theta}+\lf(1+\frac{t}{s}\r)^Q+C\lf(1+\frac{\rho(x_0)+d(x,x_0)}{s}\r)^Q \\
& \le C(x_0,s)\lf(1+d(x,x_0)\r)^{(k_0+1)^2\theta+Q}.
\end{align*}
%\begin{align*}
%&\lf(1+\frac{s}{\rho(x)}\r)^{(k_0+1)\theta}+\lf(1+\frac{t}{s}\r)^Q+\lf(1+\frac{\rho(x)}{s}\r)^Q \\
%&\ \le C\lf(1+\frac{s+d(x,x_0)}{\rho(x_0)}\r)^{(k_0+1)^2\theta}+\lf(1+\frac{t}{s}\r)^Q+C\lf[1+\frac{\rho(x_0)}{s}\lf(1+\frac{d(x,x_0)}{\rho(x_0)}\r)^{\frac{k_0}{k_0+1}}\r]^Q \\
%%&\ \le C\lf(1+\frac{s+d(x,x_0)}{\rho(x_0)}\r)^{(k_0+1)^2\theta}+\lf(1+\frac{t}{s}\r)^Q+C\lf(1+\frac{\rho(x_0)+d(x,x_0)}{s}\r)^Q \\
%&\ \le C(x_0,s)\lf(1+t+d(x,x_0)\r)^{(k_0+1)^2\theta+Q}.
%\end{align*}
%\begin{align*}
%&\lf(1+\frac{s}{\rho(x)}\r)^{(k_0+1)\theta}+\lf(1+\frac{t}{s}\r)^Q+\lf(1+\frac{\rho(x)}{s}\r)^Q \\
%&\ \le \lf(1+\frac{s+d(x,x_0)}{\rho(x)}\r)^{(k_0+1)\theta}+\lf(1+\frac{t}{s}\r)^Q+C\lf[1+\frac{\rho(x_0)}{s}\lf(1+\frac{d(x,x_0)}{\rho(x_0)}\r)^{\frac{k_0}{k_0+1}}\r]^Q \\
%&\ \le C\lf(1+\frac{s+d(x,x_0)}{\rho(x_0)}\r)^{(k_0+1)^2\theta}+\lf(1+\frac{t}{s}\r)^Q+C\lf(1+\frac{\rho(x_0)+d(x,x_0)}{s}\r)^Q \\
%&\ \le C(x_0,s)\lf(1+t+d(x,x_0)\r)^{(k_0+1)^2\theta+Q}.
%\end{align*}
With two estimates above in hand, we obtain that
$$
|H(x,t)|
%&\le |u(x,t+s)|+|\PV_tu_s(x)| \\
%&\le C\oz(B(x,s))^{\frac{\lz-1}{p}}\lf[\lf(1+\frac{s}{\rho(x)}\r)^{(k_0+1)\theta}+\lf(1+\frac{t}{s}\r)^Q+\lf(1+\frac{\rho(x)}{s}\r)^Q\r]
% \|u\|_{L^\fz(L^{p,\lz}_\tau)(\d\oz\d t)} \\
\le C(x_0,s) \lf(1+d(x,x_0)\r)^{2Q+2(k_0+1)^2\theta}\|u\|_{L^\fz(L^{p,\lz}_\tau)(\d\oz\d t)},
$$
which indicates the $\mathbb{L}$-harmonic function $H$ has the polynomial growth at most.
%On the other hand, it follows from (i), (ii) and $(De)$ that
%\begin{align*}
%|H(x,t)|
%&\le |u(x,t+s)|+|\PV_tu_s(x)| \\
%&\le CA(x,t+s) \|u\|_{L^\fz(L^{p,\lz}_\tau)(\d\oz\d t)}+CA(x,s)\|u\|_{L^\fz(L^{p,\lz}_\tau)(\d\oz\d t)} \\
%&\le CA(x,s) \|u\|_{L^\fz(L^{p,\lz}_\tau)(\d\oz\d t)}.
%%& \le C|Q(x,s)|^\az \|U(t,x)\|_{L^\fz_t(L^{\vec{p},\az}_x)(\rr_+\times\rr^{n})}.
%\end{align*}

%Based on the above arguments, we see that $H$ is a bounded $\mathbb{L}$-harmonic function on $X\times\rr_+$ with zero trace.
Finally
%one can employ the reflection method to define
%$$
%\mathcal{H}(x,t)=
%\begin{cases}
%\displaystyle H(x,t),{\quad} &t\ge0,\\
%%\displaystyle 0, &t=0,\\
%\displaystyle H(x,-t),&t<0,
%\end{cases}
%$$
%which is $\mathbb{L}$-harmonic on the whole space $X\times\rr$ with polynomial growth at most.
%%Since $t^{1-2\sz}$ is a Muckenhoupt $A_2$-weight,
%However,
the Liouville theorem (see Theorem \ref{Liouville1}) tells us that there is no $\mathbb{L}$-harmonic function (with the polynomial growth) on $X\times\rr_+$
other than zero. This means
$$u(x,t+s)-e^{-t\sqrt{\LV}}u_s(x)=H(x,t)= 0,$$
% by noting $\mathcal{H}(x,0)=0$.
which completes the proof.
\end{proof}

With the help of these properties of the $\mathbb{L}$-harmonic function above, we can prove the main result in this subsection.

\begin{proof}[Proof of Theorem \ref{bddmo}]
%\textbf{Necessity: from solution to trace.}
%Without loss of generality, we may assume $q>(Q+1) / 2$ because of the self improvement of the $RH_q(X)$ class.
%Suppose  $u \in {\textup{HL}}^{2,\alpha}_{\sqrt{L}}(X \times \mathbb{R}_+)$.
\textbf{Step 1:} seek the weak-$*$ limit of $\{u_s\}_{s>0}$.
For all $s>0$, by the definition of the ${L^\fz(L^{p,\lz}_\tau)}$-function, we have
\begin{align}\label{eq1}
\|u_s\|_{L^{p,\lz}_\tau(\d\oz)} \le \|u\|_{L^\fz(L^{p,\lz}_\tau)(\d\oz\d t)}.
\end{align}

Next, we will fix a ball $B\subset X$ and
seek a function $f \in {L^{p,\lz}}(X,\d\oz)$ through ${L^{p}}(2^mB,\d\oz)$-boundedness of
$\{u_s\}_{s>0}$ with $m\in \nn$.
 Indeed, for every $m\in \mathbb{N}$, we use (\ref{eq1}) to obtain
$$
	\|u_s\|_{{L^{p}}(2^mB,\d\oz)} 	\leq   \oz(2^mB)^{\lz/p}\lf(1+\frac{2^mr_B}{\rho(x_B)}\r)^{-\tau} \|u\|_{L^\fz(L^{p,\lz}_\tau)(\d\oz\d t)},
$$
%This tells us that the sequence $\left\{f_{k}\right\}_{k=1}^{\infty}$ is bounded in $L^{2}(B(x_0,2^{j}))$. So, after passing to a subsequence, the sequence converges weakly to a function $g_{j} \in L^{2}(B(0,2^{j})) .$
%By Eberlein¡§C?mulian theorem, we can use diagonal method to extract a sequence $u_{\varepsilon_k}$, $\varepsilon_k \rightarrow 0$ when $k \rightarrow \infty$. It converges weakly to a function $g_{j} \in L^{2}(B(0,2^{j}))$, for any $j \in \mathbb{N}$.
which implies that the family $\{u_s\}_{s>0}$ is uniformly bounded in ${L^{p}}(2^{m}B,\d\oz)$.
Thanks to the reflexive of ${L^{p}}(2^{m}B,\d\oz)$, the Eberlein-\v{S}mulian theorem and the diagonal method imply that,
there exist a subsequence $s_k \rightarrow 0$ $(k\rightarrow\infty)$ and a function $g_m \in {L^{p}}(2^{m}B,\d\oz)$
such that $u_{s_k} \rightarrow g_m$ weakly in ${L^{p}}(2^{m}B,\d\oz)$ for each $m \in \mathbb{N}$.
Now we define a function $f$ by
%\[
%	f(x)=g_{j}(x), \quad \text { if } x \in B(x_0,2^{j}), \quad j=1,2, \ldots.
%\]
$$
f(x)=g_m(x)\mathbbm{1}_{2^{m}B}(x),\quad m=1,2,\cdots.
$$
It is easy to check that $f$ is well defined on $X=\bigcup_{m=1}^{\infty} 2^{m}B$, and
$$
\|f\|_{L^{p,\lz}_\tau(\d\oz)} \le \|u\|_{L^\fz(L^{p,\lz}_\tau)(\d\oz\d t)}.
$$
%Also we can check that for any ball $B\subset \rn$,
%$$
%	\lf(\fint_B|f|^p\d\mu\r)^{1/p}\leq  \az(B) \|u\|_{L^\fz(L^{p,\az})},
%$$
%which implies
%$$
%\|f\|_{L^{p,\az}} \le \|u\|_{L^\fz(L^{p,\az})}.
%$$

\textbf{Step 2:} the existence of the trace.
Let us verify that
$$u(x,t)= e^{-t\sqrt{\LV}} f(x).$$
Since $u$ is H\"{o}lder continuous (see Theorem \ref{holder1}),
we deduce from Lemma \ref{lem1} (iii) that
$$
u(x,t)=\lim_{k\to\fz}u(x,t+s_k)=\lim_{k\to\fz} e^{-t\sqrt{\LV}}  u_{s_k}(x).
$$
It reduces to show that
$$
	\lim_{k \rightarrow \infty} e^{-t\sqrt{\LV}}  u_{s_k}(x)= e^{-t\sqrt{\LV}} f(x).
$$
Indeed, for each $\ell \in \mathbb{N},$ one writes
$$
	 e^{-t\sqrt{\LV}}  u_{s_k}(x)
	=\int_{B(x, 2^{\ell} t)} p^\LV_{t}(x, y) u_{s_k} (y) \d\mu(y)
	+\int_{X\setminus B(x, 2^{\ell} t)} p^\LV_{t}(x, y) u_{s_k}(y) \d\mu(y).
$$
On the one hand, using $(PUB)$ and \eqref{q2}, we arrive at
\begin{align*}
\left|\int_{X\setminus B(x, 2^{\ell} t)} p^\LV_{t}(x, y) u_{s_k}(y) \d\mu(y)\right|
%&\leq C \sum_{i=\ell+1}^{\infty} 2^{-i}\lf(1+\frac{2^{i}}{\rho(x)}\r)^{-N} \fint_{B(x,2^{i}t)} |u_{s_k}| \d\mu \\
	%&\leq C \sum_{i=\ell+1}^{\infty} 2^{-2\sz i} |Q(x,2^it)|^{-\frac{1}{n}(\sum_{j=1}^n\frac{1}{p_j})} \|U(s_k,y)\mathbbm{1}_{Q(x,2^it)}(y)\|_{L_y^{\vec{p}}(\rn)} \\
&\leq C \sum_{i=\ell+1}^{\infty} 2^{-i}\lf(1+\frac{2^{i}}{\rho(x)}\r)^{-N} \lf(\fint_{B(x,2^{i}t)} |u_{s_k}| \d\oz\r)^{1/p}\lf(1+\frac{2^it}{\rho(x)}\r)^{\theta} \\
&\leq C \sum_{i=\ell+1}^{\infty} 2^{-i}  \oz(B(x,2^it))^{\frac{\lz-1}{p}} \|u_{s_k}\|_{L^{p,\lz}_{\tau}(\d\oz)}
 \lf(1+\frac{2^{i}}{\rho(x)}\r)^{-N} \lf(1+\frac{2^it}{\rho(x)}\r)^{\theta-\tau} \\
%&\leq C \sum_{i=\ell+1}^{\infty} 2^{- i} A(x,2^it) \|u\|_{L^\fz(L^{p,\az})}\\
&\le C  2^{- \ell} \oz(B(x,t))^{\frac{\lz-1}{p}}\|u\|_{L^\fz(L^{p,\lz}_\tau)(\d\oz\d t)}.
\end{align*}
%where $C>0$ is a constant.
This fact yields
$$
	0
	\leq
	\lim _{\ell \rightarrow\infty} \lim _{k \rightarrow\infty}
	\left| \int_{X\setminus B(x, 2^{\ell} t)} p^\LV_{t}(x, y) u_{s_k}(y) \d\mu(y)\right|
	\leq	C\lim _{\ell \rightarrow\infty}   2^{- \ell}  \oz(B(x,t))^{\frac{\lz-1}{p}}\|u\|_{L^\fz(L^{p,\lz}_\tau)(\d\oz\d t)}	 =0.
$$
%It follows that
%$$
%	\lim _{\ell \rightarrow+\infty} \lim _{k \rightarrow+\infty}
%	 \int_{X \backslash B\left(x, 2^{\ell} t\right)} p_{t}^\LV(x, y) u_{\varepsilon_k}(y) d\mu(y)
%	=0.
%$$
On the other hand, note by $(PUB)$ and the definition of $A_p^{\rho,\theta}(X)$ that %for given $t>0$ and $x\in \rn$, we note by the definition of $A_p$ that
\begin{align*}
\int_{B(x,2^\ell t)}\lf(\frac{p^\LV_t(x,y)}{\oz(y)}\r)^{p'}\d\oz(y)
%&=\int_{B(x,2^\ell t)}p_t^\LV(x,y)^{p'}\oz(y)^{1-p'}\d\mu(y)\\
%&=\int_{B(x,2^\ell t)}p_t^\LV(x,y)^{p'}\oz(y)^{-\frac{1}{p-1}}\d\mu(y)\\
%&\le C\sum_{i=0}^\ell \int_{B(x,2^i t)}\lf(\frac{t}{t+d(x,y)}\r)^{p'}\frac{1}{\mu(B(x,t+d(x,y)))^{p'}}\lf(1+\frac{t+d(x,y)}{\rho(x)}\r)^{-N p'}\oz(y)^{-\frac{1}{p-1}}\d\mu(y)\\
%&\le C\sum_{i=0}^\ell \int_{B(x,2^i t)}2^{-ip'}\frac{1}{\mu(B(x,2^i))^{p'}}\lf(1+\frac{2^i t}{\rho(x)}\r)^{-N p'}\oz(y)^{-\frac{1}{p-1}}\d\mu(y)\\
&\le C\sum_{i=0}^\ell\frac{2^{-ip'}}{\mu(B(x,2^it))^{p'}}\lf(1+\frac{2^i t}{\rho(x)}\r)^{-N p'} \int_{B(x,2^i t)}\oz^{-\frac{1}{p-1}}\d\mu\\
%&\le \sum_{i=0}^\ell2^{-ip'}\frac{\mu(B(x,2^i))}{\mu(B(x,2^i))^{p'}}\lf(1+\frac{2^i t}{\rho(x)}\r)^{-N p'} \fint_{B(x,2^i t)}\oz(y)^{-\frac{1}{p-1}}\d\mu(y) \\
%&\le \sum_{i=0}^\ell2^{-ip'}\frac{\mu(B(x,2^i))}{\mu(B(x,2^i))^{p'}}\lf(1+\frac{2^i t}{\rho(x)}\r)^{-N p'} \\
% & \lf[\lf(\fint_{B(x,2^i t)}\oz(y)\d\mu(y)\r)\lf(\fint_{B(x,2^i t)}\oz(y)^{-\frac{1}{p-1}}\d\mu(y)\r)^{p-1}\r]^{\frac{1}{p-1}}\lf(\fint_{B(x,2^i t)}\oz(y)\d\mu(y)\r)^{-\frac{1}{p-1}} \\
%&\le \sum_{i=0}^\ell2^{-ip'}\frac{\mu(B(x,2^it))}{\mu(B(x,2^it))^{p'}}\lf(1+\frac{2^i t}{\rho(x)}\r)^{-N p'}
%  \lf(1+\frac{2^it}{\rho(x)}\r)^{\frac{\theta p}{p-1}}\lf(\fint_{B(x,2^i t)}\oz(y)\d\mu(y)\r)^{-\frac{1}{p-1}} \\
%&\le \sum_{i=0}^\ell2^{-ip'}\frac{1}{\mu(B(x,2^it))^{\frac{p}{p-1}-1}}\lf(1+\frac{2^i t}{\rho(x)}\r)^{-N p'}
%  \lf(1+\frac{2^it}{\rho(x)}\r)^{\frac{\theta p}{p-1}}\lf(\frac{\oz(B(x,2^it))}{\mu(B(x,2^it))}\r)^{-\frac{1}{p-1}} \\
&\le \sum_{i=0}^\ell2^{-ip'}{\oz(B(x,2^it))^{-\frac{1}{p-1}}}\lf(\fint_{B(x,2^i t)}\oz\d\mu\r)^{-\frac{1}{p-1}}
\le C{\oz(B(x,t))^{-\frac{1}{p-1}}},
\end{align*}
which implies that
$$
\lim _{k \rightarrow\infty} \int_{B(x, 2^{\ell} t)} p^\LV_{t}(x, y) u_{s_k} (y) \d\mu(y)
=\lim _{k \rightarrow\infty}\int_{B(x, 2^{\ell} t)} \frac{p_t^\LV(x, y)}{\oz(y)} u_{s_k}(y)\d \oz(y)=\int_{B(x, 2^{\ell} t)} p_t^\LV(x, y) f(y)\d\mu(y).
$$
Therefore, combine the local part and the global part to obtain
\begin{align*}
 \lim_{k \rightarrow \infty} e^{-t\sqrt{\LV}}  u_{s_k}(x)
&=\lim _{\ell \rightarrow\infty} \lim _{k \rightarrow\infty} \int_{B(x, 2^{\ell} t)} p^\LV_{t}(x, y) u_{s_k} (y) \d\mu(y)
	+\lim _{\ell \rightarrow\infty} \lim _{k \rightarrow\infty} \int_{X\setminus B(x, 2^{\ell} t)} p^\LV_{t}(x, y) u_{s_k}(y) \d\mu(y)\\
&=\lim _{\ell \rightarrow\infty} \int_{B(x, 2^{\ell} t)} p^\LV_{t}(x, y) f (y) \d\mu(y)= e^{-t\sqrt{\LV}} f(x)
\end{align*}
as desired.

{\bf Step 3:} completion of the proof.
Based on the above argument, we derive that every weighted harmonic Morrey function $u$
 can be represented as $u=e^{-t\sqrt{\LV}}f$ for some weighted Morrey function $f$, and its trace admits
$$
\|f\|_{L^{p,\lz}_\tau(\d\oz)} \le \|u\|_{L^\fz(L^{p,\lz}_\tau)(\d\oz\d t)}
$$
for each $\tau\in\rr$.
This completes the proof.
%This completes the proof.
%The proof of Theorem \ref{prop:proof (1) of main theorem} is complete.
\end{proof}

\subsection{From trace to solution}
\hskip\parindent
In this subsection we complete the proof of Theorem \ref{thm3} by showing that
the Poisson integral of a weighted Morrey function is a still a weighted Morrey function with respect to the time variable uniformly.

\begin{theorem}\label{moo}
Let $(X,d,\mu,\mathscr{E})$ be a complete Dirichlet metric measure space satisfying a doubling property $(D)$ and admitting an $L^2$-Poincar\'{e} inequality $(P_2)$.
Suppose that $0\le V\in  A_\infty(X)\cap RH_{q}(X)$ for some $q>\max\{Q/2,1\}$.
%Assume that $\oz\in A_p^{\rho,\theta}(X)$ with $1<p<\fz$ and $\theta>0$, and the non-negative set function $\lz$ is ``decreasing'' $(De)$. %and translation invariant $(TI)$.
For any $1<p<\fz$, $0<\lz<1$ and $\theta>0$,
if $f\in L^{p,\lz}(X,\d\oz)$ with $\oz\in A_p^{\rho,\theta}(X)$, then $u=e^{-t\sqrt{\LV}}f\in L^\fz(L^{p,\lz})(X\times \rr_+,\d\oz\d t)$.
Moreover, for any $\tau\in\rr$, it holds that
$$ \|u\|_{L^\fz(L^{p,\lz}_\tau)(\d\oz\d t)}\le C\|f\|_{L^{p,\lz}_{\tau+\lz\theta}(\d\oz)}.$$
\end{theorem}

\begin{proof}
Fix a ball $B=B(x_B,r_B)\subset X$. Let us verify
$$
\lf(\frac{1}{\oz(B)^\lz}\int_B \big| e^{-t\sqrt{\LV}} f\big|^p\d\oz\r)^{1/p}\le C\lf(1+\frac{r_B}{\rho(x_B)}\r)^{-\tau}\|f\|_{L^{p,\lz}_{\tau+\lz\theta}(\d\oz)}.
$$
%where $C>0$ is a constant.

To this end, we consider two cases: the small time and the large time.

{\bf Case 1:}  the small time $t<r_B$.
One can write
$$
f=f\mathbbm{1}_{2B}+u\mathbbm{1}_{4B\setminus2B}+\cdots+f\mathbbm{1}_{2^kB\setminus2^{k-1}B}+\cdots=\sum_{k=1}^\fz f_k.
$$
For the first term $f_1$, we deduce from $(PUB)$ that
$$
\big| e^{-t\sqrt{\LV}}  f_1(x)\big|
%&\le C\int_X\lf(\frac{t}{t+d(x,y)}\r)\frac{1}{\mu(B(x,t+d(x,y)))}\lf(1+\frac{t+d(x,y)}{\rho(x)}\r)^{-N}|f_1(y)|\d\mu(y) \\
\le C\sum_{i=0}^\fz2^{-i} \lf(1+\frac{2^it}{\rho(x)}\r)^{-N}\fint_{B(x,2^it)}|f_1|\d\mu
\le C M_Nf_1(x),
$$
where $M_Nf_1$ is the variant Hardy--Littlewood maximal function of $f_1$ defined by
$$M_Nf_1(x)=\sup_{r>0}\lf(1+\frac{r}{\rho(x)}\r)^{-N}\fint_{B(x,r)}|f_1|\d\mu.$$
%For each $j=1,\cdots, n$, denote by $M_jv$ the Hardy--Littlewood maximal function of $v$ with respect to the $j$-th variable.
%Evidently one has
%\begin{align*}
% Mu^1(x)\le \sup_{r_n>0}\fint_{x_n-r_n}^{x_n+r_n}\cdots\lf(\sup_{r_1>0}\fint_{x_1-r_1}^{x_1+r_1}|u^1(y_1,\cdots,y_n)|\d y_1\r)\cdots\d y_n
% =  M_n\cdots M_1u^1(x_1,\cdots,x_n).
%\end{align*}
Therefore we employ the weighted $L^p$-boundedness of $M_N$ (see \cite[Proposition 2.7]{BBD2022} with $N=\theta/(p-1)$)
%the Fefferman-Stein vector-valued inequality in continuous form time and time again
%(see \cite[Section 3]{Ba1975} for example and \cite[Proposition 2.1]{CS2022} for a slightly general version)
to obtain
\begin{align*}
\lf(\frac{1}{\oz(B)^\lz}\int_B\big | e^{-t\sqrt{\LV}} f_1\big|^p\d\oz\r)^{1/p}
%&\le \lf(\frac{1}{\oz(B)^\lz}\int_X |M_Nf_1|^p\d\oz\r)^{1/p}\\
&\le C[\oz]^{\frac{1}{p-1}}_{A_p^{\rho,\theta}}\lf(\frac{1}{\oz(B)^\lz}\int_{2B} |f|^p\d\oz\r)^{1/p} \\
%&\le C[\oz]^{\frac{1}{p-1}}_{A_p^{\rho,\theta}}\lf(\frac{\oz(2B)}{\oz(B)}\r)^{\lz/p}\lf(\frac{1}{\oz(2B)^\lz}\int_{2B} |f|^p\d\oz\r)^{1/p} \\
%&\le C[\oz]^{\frac{1}{p-1}}_{A_p^{\rho,\theta}}\lf(1+\frac{2r_B}{\rho(x_B)}\r)^{-\tau-\lz \theta}
% \lf(\frac{\oz(2B)}{\oz(B)}\r)^{\lz/p}\lf(\frac{1}{\oz(2B)^\lz}\int_{2B} |f|^p\d\oz\r)^{1/p}
% \lf(1+\frac{2r_B}{\rho(x_B)}\r)^{\tau+\lz \theta} \\
&\le C \lf(1+\frac{2r_B}{\rho(x_B)}\r)^{-\tau-\lz\theta}\lf(\frac{\oz(2B)}{\oz(B)}\r)^{\lz/p}\|f\|_{L^{p,\lz}_{\tau+\lz\theta}(\d\oz)}
\le C \lf(1+\frac{r_B}{\rho(x_B)}\r)^{-\tau}\|f\|_{L^{p,\lz}_{\tau+\lz\theta}(\d\oz)}.
\end{align*}
%\begin{align*}
%\|P_t u^1(x)\mathbbm{1}_Q(x)\|_{L_x^{\vec{p}}(\rn)}
%&\le C\|M_n\cdots M_1u^1(x)\|_{L_x^{\vec{p}}(\rn)} \\
%&=\left\{\int_{\mathbb{R}} \cdots\left(\int_{\mathbb{R}}|M_n\cdots M_1u^1(x_1,\cdots,x_n)|^{p_1} \d x_1\right)^{p_2/{p_1}} \cdots \d x_n\right\}^{1/{p_n}}\\
%&\le C\left\{\int_{\mathbb{R}} \cdots\left(\int_{\mathbb{R}}|u\mathbbm{1}_{2Q}(x_1,\cdots,x_n)|^{p_1} \d x_1\right)^{p_2/{p_1}} \cdots \d x_n\right\}^{1/{p_n}} \\
%&= C\|u\mathbbm{1}_{2Q}(x)\|_{L_x^{\vec{p}}(\rn)} \\
%&\le C|Q|^{\az+\frac{1}{n}(\sum_{j=1}^n \frac{1}{p_j})}\|u(x)\|_{L^{\vec{p},\az}_x(\rn)}.
%\end{align*}
To estimate the term $f_k$, one may invoke $(PUB)$, \eqref{critical} and \eqref{q2} to deduce that for each $x\in B$,
\begin{align*}
\big| e^{-t\sqrt{\LV}}  f_k(x)\big|
%&\le C\int_{2^kB\setminus2^{k-1}B}\lf(\frac{t}{t+d(x,y)}\r)\frac{1}{\mu(B(x,t+d(x,y)))}\lf(1+\frac{t+d(x,y)}{\rho(x)}\r)^{-N}|f_k(y)|\d\mu(y) \\
&\le C\lf(\frac{t}{2^{k}r_B}\r)\lf(1+\frac{2^kr_B}{\rho(x)}\r)^{-N}\fint_{2^kB}|f|\d\mu \\
&\le C2^{-k}\lf(1+\frac{2^kr_B}{\rho(x_B)}\r)^{-\frac{N}{k_0+1}}\lf(\fint_{2^kB}|f|^p\d\oz\r)^{1/p}\lf(1+\frac{2^kr_B}{\rho(x_B)}\r)^\theta
%&\le C2^{-k}\oz(2^kB)^{\frac{\lz-1}{p}}\lf(1+\frac{2^kr_B}{\rho(x_B)}\r)^{-\frac{N}{k_0+1}}\lf(1+\frac{2^kr_B}{\rho(x_B)}\r)^{-\tau}
%\|f\|_{L^{p,\lz}_{\tau+\lz\theta}(\d\oz)} \\
%&\le C \lf(\frac{t}{2^{k}r_Q}\r)^{2\sz}|2^kQ|^{-\frac{1}{n}(\sum_{j=1}^n \frac{1}{p_j})}\|u\mathbbm{1}_{2^kQ}(y)\|_{L_y^{\vec{p}}(\rn)}  \\
\le C 2^{- k} \oz(B)^{\frac{\lz-1}{p}}\lf(1+\frac{r_B}{\rho(x_B)}\r)^{-\tau}\|f\|_{L^{p,\lz}_{\tau+\lz\theta}(\d\oz)},
\end{align*}
which implies that
$$
\lf(\frac{1}{\oz(B)^\lz}\int_B \big| e^{-t\sqrt{\LV}} f_k\big|^p\d\oz\r)^{1/p}\le C 2^{- k}\lf(1+\frac{r_B}{\rho(x_B)}\r)^{-\tau}\|f\|_{L^{p,\lz}_{\tau+\lz\theta}(\d\oz)}.
%&\le \lf(\frac{1}{w(Q)}\int_\rn |Mu^1(x)|^p\d w(x)\r)^{1/p}\\
%&\le C\lf(\frac{1}{w(Q)}\int_{2Q} |u(x)|^p\d w(x)\r)^{1/p} \le C\az(Q)\|u\|_{L^{p,\az}_w}.
$$
%\begin{align*}
%\|P_t u^k(x)\mathbbm{1}_Q(x)\|_{L_x^{\vec{p}}(\rn)}
%&\le C2^{-2\sz k} |Q|^{\az} \|\mathbbm{1}_Q(x)\|_{L_x^{\vec{p}}(\rn)}\|u(y)\|_{L^{\vec{p},\az}_y(\rn)}\\
%&\le C2^{-2\sz k} |Q|^{\az+\frac{1}{n}(\sum_{j=1}^n \frac{1}{p_j})}\|u(x)\|_{L^{\vec{p},\az}_x(\rn)}.
%\end{align*}

Summing over $k$ leads to
$$
\lf(\frac{1}{\oz(B)^\lz}\int_B \big| e^{-t\sqrt{\LV}} f\big|^p\d\oz)\r)^{1/p}
%&\le \sum_{k=1}^\fz\lf(\fint_B | e^{-t\sqrt{\LV}} f_k|^p\d\oz\r)^{1/p} \\
\le C \sum_{k=1}^\fz2^{- k} \lf(1+\frac{r_B}{\rho(x_B)}\r)^{-\tau}\|f\|_{L^{p,\lz}_{\tau+\lz\theta}(\d\oz)}
\le C \lf(1+\frac{r_B}{\rho(x_B)}\r)^{-\tau}\|f\|_{L^{p,\lz}_{\tau+\lz\theta}(\d\oz)}
%&\le \lf(\frac{1}{w(Q)}\int_\rn |Mu^1(x)|^p\d w(x)\r)^{1/p}\\
%&\le C\lf(\frac{1}{w(Q)}\int_{2Q} |u(x)|^p\d w(x)\r)^{1/p} \le C\az(Q)\|u\|_{L^{p,\az}_w}.
$$
%\begin{align*}
%\|U(t,x)\mathbbm{1}_Q(x)\|_{L_x^{\vec{p}}(\rn)}
%\le \sum_{k=1}^\fz\|P_t u^k(x)\mathbbm{1}_Q(x)\|_{L_x^{\vec{p}}(\rn)}\le C|Q|^{\az+\frac{1}{n}(\sum_{j=1}^n \frac{1}{p_j})}\|u(x)\|_{L^{\vec{p},\az}_x(\rn)}
%\end{align*}
as desired.

{\bf Case 2:}  the large time $t\ge r_B$.
In this case, we may employ $(PUB)$, \eqref{critical} and \eqref{q2}
to obtain that for each $x\in B\subset B(x_B,t)$,
%\begin{align*}
%u=u\mathbbm{1}_{Q(x_Q,2t)}+u\mathbbm{1}_{Q(x_Q,4t)\setminus Q(x_Q,2t)}+\cdots+u\mathbbm{1}_{Q(x_Q,2^kt)\setminus Q(x_Q,2^{k-1}t)}+\cdots=\sum_{k=1}^\fz u_k.
%\end{align*}
%For the term $u_k$ with $k\ge1$, one
$$
\big| e^{-t\sqrt{\LV}}  f(x)\big|
%&\le \lf\{\int_{B(x_B,2t)}+\sum_{k=2}^\fz\int_{B(x_B,2^kt)\setminus B(x_B,2^{k-1}t)}\r\} p^\LV_t(x,y)|f(y)| \d\mu(y) \\
%&\le C\sum_{k=1}^\fz2^{- k}\lf(1+\frac{2^kt}{\rho(x)}\r)^{-N} \fint_{B(x_B,2^kt)}|f|\d\mu\\
\le C\sum_{k=1}^\fz2^{- k}\lf(1+\frac{2^kt}{\rho(x_B)}\r)^{-\frac{N}{k_0+1}} \lf(\fint_{B(x_B,2^kt)}|f|^p\d\oz\r)^{1/p}\lf(1+\frac{2^kt}{\rho(x_B)}\r)^\theta
%&\le C\sum_{k=1}^\fz2^{- k}\lf(1+\frac{2^kt}{\rho(x_B)}\r)^{-\frac{N}{k_0+1}}\lf(1+\frac{2^kt}{\rho(x_B)}\r)^{-\tau}
% \lf(\fint_{B(x_B,2^kt)}|f|^p\d\oz\r)^{1/p}\lf(1+\frac{2^kt}{\rho(x_B)}\r)^{\theta+\tau}\\
%&\le C\sum_{k=1}^\fz2^{- k}\oz(2^kB)^{\frac{\lz-1}{p}}
% \lf(1+\frac{2^kt}{\rho(x_B)}\r)^{-\frac{N}{k_0+1}}\lf(1+\frac{2^kt}{\rho(x_B)}\r)^{-\theta-\tau-\lz\theta}\|f\|_{L^{p,\lz}_{\tau+\lz\theta}(\d\oz)}\\
%&\le C\sum_{k=1}^\fz2^{- k}\oz(B)^{\frac{\lz-1}{p}}\lf(1+\frac{r_B}{\rho(x_B)}\r)^{-\tau}\|f\|_{L^{p,\lz}_{\tau+\lz\theta}(\d\oz)}\\
%&\le C\sum_{k=1}^\fz2^{- k} A(x_B,2^kt)\|f\|_{L^{p,\az}_{\tau+\theta}(\d\oz)}\\
%&\le C\sum_{k=1}^\fz2^{-2\sz k} |Q|^{\az}\|u(y)\|_{L^{\vec{p},\az}_y(\rn)}\\
\le C \oz(B)^{\frac{\lz-1}{p}}\lf(1+\frac{r_B}{\rho(x_B)}\r)^{-\tau}\|f\|_{L^{p,\lz}_{\tau+\lz\theta}(\d\oz)},
$$
which yields
$$
\lf(\frac{1}{\oz(B)^\lz}\int_B \big| e^{-t\sqrt{\LV}} f\big|^p\d\oz\r)^{1/p}\le C\lf(1+\frac{r_B}{\rho(x_B)}\r)^{-\tau}\|f\|_{L^{p,\lz}_{\tau+\lz\theta}(\d\oz)}.
%&\le \lf(\frac{1}{w(Q)}\int_\rn |Mu^1(x)|^p\d w(x)\r)^{1/p}\\
%&\le C\lf(\frac{1}{w(Q)}\int_{2Q} |u(x)|^p\d w(x)\r)^{1/p} \le C\az(Q)\|u\|_{L^{p,\az}_w}.
$$
%\begin{align*}
%\|U(t,x)\mathbbm{1}_Q(x)\|_{L_x^{\vec{p}}(\rn)}
%\le C|Q|^{\az}\|\mathbbm{1}_Q(x)\|_{L_x^{\vec{p}}(\rn)}\|u(y)\|_{L^{\vec{p},\az}_y(\rn)}\le C|Q|^{\az+\frac{1}{n}(\sum_{j=1}^n \frac{1}{p_j})}\|u(x)\|_{L^{\vec{p},\az}_x(\rn)}
%\end{align*}
%as desired.

Finally, based on the above argument in Cases 1 and 2, we derive
$$
\|u\|_{L^\fz(L^{p,\lz}_\tau)(\d\oz\d t)}
=\sup_{t>0}\lf(\frac{1}{\oz(B)^\lz}\int_B  \big | e^{-t\sqrt{\LV}} f\big|^p\d\oz\r)^{1/p}\lf(1+\frac{r_B}{\rho(x_B)}\r)^{\tau}
\le C\|f\|_{L^{p,\lz}_{\tau+\lz\theta}(\d\oz)}
$$
as desired. This completes the proof.
\end{proof}

\subsection{Final remarks}\label{ss1}
\hskip\parindent
%We first mention
%This section will end by pointing out a further result.
%Since it is intended solely as a brief review and not as a rigorous development, the pertinent result is stated without proof.
This section will end by pointing out some further results. Since it is intended solely as a brief
review and not as a rigorous development, the pertinent results are stated without proof.

There are some points about previous articles \cite{SL2022,STY2018} need to be mentioned first.
\begin{enumerate}
\item[\rm{(i)}] One is the $\lz$'s range of the Morrey function in \cite{SL2022}.
Due to the complex structure of the metric measure space, the authors assumed $1-1/Q<\lz<1$ in \cite[Theorem 3]{SL2022} to ensure their conclusion holds.
In fact, with the help of $(PUB)$, we can relax the lower bound of $\lz$ to 0.

\item[\rm{(ii)}] The other is the definition of the Morrey space.
Note that the Morrey function in \cite{STY2018} is a special case of our unweighted Morrey function in the Euclidean setting, i.e., $\oz=1$ and $\tau=0$ in Definition \ref{def1}.
In \cite{STY2018}, Song--Tian--Yan obtained that an $\mathbb{L}$-harmonic function on $\rr^{n+1}_+$ satisfies the Carleson measure conditioin
if and only if it can be represented as the Poisson integral with Morrey trace.
%In \cite{SL2022,STY2018}, the Morrey space $L^{2,\lz}_0(X)$ with $0<\lz<1$ is defined to be the set of all measurable functions $f$ on $\rn$ such that their norms
%$$\|f\|_{L^{2,\lz}}=\sup_B\lf(\frac{1}{\mu(B)^\lz}\int_B|f|^2\d\mu\r)^{1/2}<\fz.$$
In the proof of \cite[Theorem 1.1 (i)]{STY2018},
the authors utilized the full gradient estimates on the kernel of the Poisson semigroup related to $V$.
We remark that, since $V\in RH_q(\rn)$ for some $q>n$, the Poisson kernel associated with $L=-\Delta+V$
is majorized by the classical one with an additional polynomial decay, namely, it holds
\begin{align}\label{q3}
|t\nabla_x p^\LV_t(x,y)|\le C\frac{t}{(t+|x-y|)^{n+1}} \left(1+\frac{t+|x-y|}{\rho(x)}\right)^{-N}.
\end{align}
However, in the proof of \cite[Theorem 2]{SL2022} (or see \cite[Theorem 1.1 (i)]{STY2018}), the authors merely utilized the classical Poisson upper bound.
The additional polynomial decay related to $V$ in \eqref{q3} is \textit{not} used in their arguments.
Even if $V\ge0$ merely, \cite[Theorem 1.1 (i)]{STY2018} and \cite[Theorem 2]{SL2022} still valid.
Therefore the Morrey space $L^{2,\lz}_0(X)$ is not an ``optimal'' initial value class.

\item[\rm{(iii)}] The last one is the theory of the function space associated to the operator plays a crucial role in the argument of \cite{STY2018}.
Song--Tian--Yan wrote in \cite[Abstract]{STY2018}

``\textit{Its proof heavily relies on investigate the intrinsic relationship
 between the classical Morrey spaces and the new Campanato spaces associated to the operator} $L$'',

and invested significant time and effort in \cite[Section 2]{STY2018} to establish
$$L^{2,\lz}_0(\rn)=\mathscr{L}^{2,\lz}_{L}(\rn),$$
where the new Campanato space $\mathscr{L}^{2,\lz}_{L}(\rn)$ is defined to be the set of all $f\in L^2(\rn,(1+|x|)^{-n-\bz}\d x)$ with $\bz>0$ such that their norms
$$
\|f\|_{\mathscr{L}^{2,\lz}_{L}}=\sup_B \lf(\frac{1}{|B|^\lz}\int_B \big |f(x)-e^{-r_B\sqrt{\LV}}f(x)\big|^2 \d x\r)^{1/2}<\fz.
$$
By the aid of a new Calder\'{o}n reproducing formula from \cite[Lemma 5]{SL2022} (which is different from that of \cite{STY2018}),
we can get rid of the new Campanato space $\mathscr{L}^{2,\lz}_{L}(\rn)$ above.
\end{enumerate}

Inspired by Theorem \ref{thm3} and these remarks above, we define a new solution space as follow.
\begin{definition}
An $\mathbb{L}$-harmonic function $u$ defined on $X\times \rr_+$
is said to be in $\mathrm{HMO}^{2,\lz}_\tau=\mathrm{HMO}^{2,\lz}_\tau(X\times \rr_+)$ with $0<\lz<1$ and $\tau\in\rr$,
the {space of functions of variant harmonic mean oscillation},
if $u$ satisfies
$$
\|u\|_{\HMO^{2,\lz}_\tau}=\sup_{B} \lf(\frac{1}{\mu(B)^\lz}\int_0^{r_B}\int_B  |t\nabla_{x,t} u|^2\d \mu\frac{\d t}{t}\r)^{1/2}\lf(1+\frac{r_B}{\rho(x_B)}\r)^\tau<\fz.
$$
\end{definition}

With the help of the variant harmonic mean oscillation function above,
we can improve all results of \cite{SL2022,STY2018} to the following theorem.
But to limit the length of this article, we leave the detail to the interested reader.
\begin{theorem}\label{thm4}
Let $(X,d,\mu,\mathscr{E})$ be a complete Dirichlet metric measure space satisfying a doubling property $(D)$ and admitting an $L^2$-Poincar\'{e} inequality $(P_2)$.
Suppose that $0\le V\in A_\infty(X)\cap RH_{q}(X)$ for some $q> \max\{Q/2,1\}$.
If $0<\lz<1$ and $\tau\in\rr$, then $u\in \mathrm{HMO}^{2,\lz}_\tau(X\times \rr_+)$ if and only if $u=e^{-t\sqrt{\LV}}f$ for some $f\in L^{2,\lz}_\tau(X)$.
Moreover, it holds that % there exists a constant $C>0$ such that
$$\|u\|_{\mathrm{HMO}^{2,\lz}_\tau} \approx\|f\|_{L^{2,\lz}_\tau}.$$
\end{theorem}

\begin{remark}
When $(X,d,\mu)=(\rn,|\cdot|,\d x)$, even if $\tau=0$, our Theorem \ref{thm4} is new.
%It is quite believable that all results of Song-Tian-Yan \cite{STY2018} and Shen-Li \cite{SL2022}
%may also be true for our unweighted Morrey function (i.e., $\oz=1$).
Moreover, a combination of Theorems \ref{thm3} and \ref{thm4} indicates that
%Therefore, the above argument and Theorem \ref{thm3} tells us that,
for every $\mathbb{L}$-harmonic function $u$ on $\rn\times \rr_+$, it holds
$$
\sup_{B} \lf(\frac{1}{\mu(B)^\lz}\int_0^{r_B}\int_B  |t\nabla_{x,t} u|^2\d \mu\frac{\d t}{t}\r)^{1/2}\lf(1+\frac{r_B}{\rho(x_B)}\r)^\tau<\fz
\Longleftrightarrow \sup_{t>0}\sup_{B}\lf(\frac{1}{\mu(B)^\lz}\int_B|u|^2\d \mu\r)^{1/2}\lf(1+\frac{r_B}{\rho(x_B)}\r)^\tau<\fz.
$$
This fact surprises somewhat us since we can use the integral average of an $\mathbb{L}$-harmonic function
to characterize the integral average of its derivative.
\end{remark}

Bypassing the Campanato spaces associated to the operator, we can define the new square Morrey space related to $\rho$ as follows.
\begin{definition}\label{def3}
For any  $0<\lz<1$,
%and $\lz$ is a non-negative set function.
a locally integrable function $f$ defined on $X$ is said to be in $L^{2,\lz}_{\rho}=L^{2,\lz}_{\rho}(X)$,
the square Morrey space related to $\rho$, if $f$ satisfies
$$
\|f\|_{L^{2,\lz}_{\rho}}=\sup_{B:r_B< \rho(x_B)}\lf(\frac{1}{\mu(B)^\lz}\int_B|f-f_B|^2\d \mu\r)^{1/2}
 +\sup_{B:r_B\ge \rho(x_B)}\lf(\frac{1}{\mu(B)^\lz}\int_B|f|^2\d \mu\r)^{1/2}<\fz.
$$
\end{definition}

\begin{theorem}\label{thm5}
Let $(X,d,\mu,\mathscr{E})$ be a complete Dirichlet metric measure space satisfying a doubling property $(D)$ and admitting an $L^2$-Poincar\'{e} inequality $(P_2)$.
Suppose that $0\le V\in A_\infty(X)\cap RH_{q}(X)$ for some $q> \max\{Q/2,1\}$.
If $0<\lz<1$, then the following statements are equivalent.
\begin{enumerate}
   \item[{(i)}] the function $f$ is in $L^{2,\lz}_0(X)$.
 \item[{(ii)}] the function $f$ is in $L^{2,\lz}_\rho(X)$.
 \item[{(iii)}] the Poisson integral $e^{-t\sqrt{\LV}}f$ is in $\mathrm{HMO}^{2,\lz}_0(X\times \rr_+)$
 for some $f\in L^2(X,(1+d(x,x_0))^{-Q-\bz}\d\mu(x))$ with $x_0\in X$ and $\bz>0$.
\end{enumerate}
%
%
% $u\in \mathrm{HMO}^{2,\lz}_\tau(X\times \rr_+)$ if and only if $u=\PV_tf$ for some $f\in L^{2,\lz}_\tau(X)$.
%Moreover, there exists a constant $C>0$ such that
%$$C^{-1}\|f\|_{L^{2,\lz}_\tau}\le \|u\|_{\mathrm{HMO}^{2,\lz}_\tau}\le C\|f\|_{L^{2,\lz}_\tau}.$$
Moreover, it holds
$$\|f\|_{L^{2,\lz}_0}\approx\|f\|_{L^{2,\lz}_\rho}\approx\big\|e^{-t\sqrt{\LV}}f\big\|_{\mathrm{HMO}^{2,\lz}_0}.$$
\end{theorem}

\begin{proof}
By the definitions of two classes of Morrey functions,  and Theorem \ref{thm4} (or see \cite[Theorem 3]{SL2022}),
we know that (i) implies (ii), and (iii) implies (i).
Therefore, to show Theorem \ref{thm5}, it suffices to prove that (ii) implies (iii).
The remainder of the argument is analogous to that in \cite[Proposition 5.3]{JL2022}, and is left to the interested reader.
This completes the proof.
\end{proof}

\begin{remark}
An important observation is that, when considering the metric measure space case,
the Morrey space $L^{2,\lz}_0(X)$ in Definition \ref{def1} and another Morrey space $L^{2,\lz}_\fz(X)$ in Definition \ref{def3} (i.e., $V=0$)
always coincide provided the underlying space admits the Ahlfors regular; see \cite{Na2006} for example.
However, in the Schr\"{o}dinger setting, we can relax the Ahlfors regularity to a doubling property due to the additional decay from $V$.
\end{remark}

\addcontentsline{toc}{section}{Acknowledgement}
\begin{flushleft}
\textbf{Acknowledgement.}
Bo Li is supported by NNSF of China (12471094,12201250), Zhejiang NSF of China (LQ23A010007), Jiaxing NSF of China (2023AY40003) and Qinshen Scholar Program of Jiaxing University.
Ji Li is supported by ARC DP 220100285.
Liangchuan Wu is  supported by NNSF of China (12201002) and  Excellent University Research and Innovation Team in Anhui Province (2024AH010002).
\end{flushleft}

%
%{\small \noindent {\textbf{Acknowledgements.}}
%Bo Li is supported by NNSF of China (12201250), NSF of Zhejiang (LQ23A010007) and NSF of Jiaxing (2023AY40003).
%}

\noindent B. Li (bli@zjxu.edu.cn), Department of Mathematics, Jiaxing University, Jiaxing 314001, China

\noindent J. Li (ji.li@mq.edu.au), Department of Mathematics, Macquarie University, Macquarie Park, NSW 2109, Australia

\noindent L. Wu (wuliangchuan@ahu.edu.cn), School of Mathematical Science, Anhui University, Hefei 230601, China

\end{document}